\documentclass[12pt,letterpaper]{article}

\usepackage[utf8]{inputenc}
\usepackage{amssymb}
\usepackage{stackrel}
\usepackage{comment}

\usepackage[textwidth=0.89in,textsize=scriptsize]{todonotes}

\usepackage{tikz}
\usepackage[all,cmtip,2cell]{xy}
\UseTwocells
\usepackage{mathtools}
\usepackage{tikz-cd}
\usepackage{bm}
\usepackage{mathrsfs}
\usepackage{cite}
\usepackage{amsmath}
\numberwithin{equation}{section}
\usepackage{hyperref}
\usepackage{url}
\usepackage{lineno}
\usepackage{amsfonts}
\usepackage{amsthm}
\usepackage{tikz-3dplot}
\usepackage{quiver}

\theoremstyle{definition}
\newtheorem{define}{Definition}[section]
\newtheorem{example}[define]{Example}
\newtheorem{construction}[define]{Construction}
\newtheorem{notation}[define]{Notation}

\theoremstyle{remark}
\newtheorem{remark}[define]{Remark}
\theoremstyle{plain}
\newtheorem{theo}[define]{Theorem}
\newtheorem{lemma}[define]{Lemma}
\newtheorem{prop}[define]{Proposition}
\newtheorem{cor}[define]{Corollary}

\newcommand{\C}{\mathscr C}
\newcommand{\E}{\mathbb E}
\newcommand{\G}{\mathscr G}

\newcommand{\B}{\mathscr B}

\newcommand{\ct}[1]{\mathscr #1}
\newcommand{\F}{\mathscr F}
\newcommand{\D}{\mathbb D}
\newcommand{\K}{\mathfrak K}

\newcommand{\Rho}{\mathrm{P}}
\newcommand{\Tau}{\mathrm{T}}
\newcommand{\Iota}{\mathrm{I}}

\newcommand{\mfrk}{\mathfrak}
\newcommand{\mcal}{\mathcal}
\newcommand{\mrm}{\mathrm}
\newcommand{\mbf}{\mathbf}
\newcommand{\mscr}{\mathscr}

\newcommand{\cat}{\mathbf{Cat}}
\newcommand{\pscat}{\mathbf{PsCat}}
\newcommand{\psmon}{\mathbf{PsMon}}

\newcommand{\twocat}{\mathbf{2Cat}}
\newcommand{\Montwocat}{\mathbf{Mon2Cat}}
\newcommand{\Mon}{\mathbf{Mon}}
\newcommand{\MonFib}{\mathbf{MonFib}}
\newcommand{\Moncat}{\mathbf{MonCat}}
\newcommand{\set}{\mathbf{Set}}

\newcommand{\Span}{\mathbb S\mathbf{pan}}
\newcommand{\Rel}{\mathbb R\mathbf{el}}

\newcommand{\dfib}{\mathbf{DFib}}
\newcommand{\fib}{\mathbf{Fib}}
\newcommand{\cfib}{\mathbf{cFib}}

\newcommand{\lax}{\mathbf{Lax}}
\newcommand{\icat}{\mathbf{ICat}}
\newcommand{\iset}{\mathbf{ISet}}
\newcommand{\dblelt}{\mathbb E\mathbf{l}}

\newcommand{\dblfam}{\mathbb F\mathbf{am}}
\newcommand{\dblfib}{\mathbf{DblFib}}

\newcommand{\prof}{\mathbb P\mathbf{rof}}
\newcommand{\dbltwocat}{\mbf {Dbl2Cat}}
\newcommand{\dom}{\mathsf{dom}}
\newcommand{\cod}{\mathsf{cod}}
\newcommand{\src}{\mathsf{src}}
\newcommand{\tgt}{\mathsf{tgt}}
\newcommand{\apx}{\mathsf{apx}}
\newcommand{\im}{\mathsf{im}}
\newcommand{\hto}{\slashedrightarrow}
\newcommand{\Dblcatv}{{\mathbf{Dbl}}}
\newcommand{\Dblcatvs}{{\mathbf{Dbl}_s}}
\newcommand{\Dblcatvl}{{\mathbf{Dbl}_\ell}}

\newcommand{\ol}{\overline}
\newcommand\dol[1]{\overline{\overline{#1}}}

\newcommand{\Arrs}{\mbox{\sf Arr}^s}
\newcommand{\Arrp}{\mbox{\sf Arr}^p}
\newcommand{\Arrl}{\mbox{\sf Arr}^\ell}

\newcommand{\Hom}{\mbox{\sf Hom}}

\newcommand{\triplear}[4]
{
\ar@/^/[#1]^{#2}
\ar@{}[#1]|{#4}
\ar@/_/[#1]_{#3}
}

\makeatother
\usepackage{amssymb}
\usepackage[left=2cm,right=2cm,top=2cm,bottom=2cm]{geometry}
\makeatletter
\def\slashedarrowfill@#1#2#3#4#5{%
  $\m@th\thickmuskip0mu\medmuskip\thickmuskip\thinmuskip\thickmuskip
  \relax#5#1\mkern-7mu%
  \cleaders\hbox{$#5\mkern-2mu#2\mkern-2mu$}\hfill
  \mathclap{#3}\mathclap{#2}%
  \cleaders\hbox{$#5\mkern-2mu#2\mkern-2mu$}\hfill
  \mkern-7mu#4$%
}
\def\rightslashedarrowfill@{%
  \slashedarrowfill@\relbar\relbar\mapstochar\rightarrow}
\newcommand\xslashedrightarrow[2][]{%
  \ext@arrow 0055{\rightslashedarrowfill@}{#1}{#2}}
\makeatother
\def\slashedrightarrow{\xslashedrightarrow{}}

\makeatletter
\newcommand{\xdashrightarrow}[2][]{\ext@arrow 0359\rightarrowfill@@{#1}{#2}}
\newcommand{\xdashleftarrow}[2][]{\ext@arrow 3095\leftarrowfill@@{#1}{#2}}
\newcommand{\xdashleftrightarrow}[2][]{\ext@arrow 3359\leftrightarrowfill@@{#1}{#2}}
\def\rightarrowfill@@{\arrowfill@@\relax\relbar\rightarrow}
\def\leftarrowfill@@{\arrowfill@@\leftarrow\relbar\relax}
\def\leftrightarrowfill@@{\arrowfill@@\leftarrow\relbar\rightarrow}
\def\arrowfill@@#1#2#3#4{%
  $\m@th\thickmuskip0mu\medmuskip\thickmuskip\thinmuskip\thickmuskip
   \relax#4#1
   \xleaders\hbox{$#4#2$}\hfill
   #3$%
}
\makeatother

\newcommand{\bb}{\mathbb}
\newcommand{\cc}[1]{\mathcal{#1}}
\newcommand{\Cat}{\cc{C}at}

\newcommand{\mr}[1]{\stackrel{#1}{\to}}
\newcommand{\mrdash}[1]{\xdashrightarrow{{\quad}#1{\quad}} }

\newcommand{\Mr}[1]{\stackrel{#1}{\Rightarrow}}

\newcommand{\mrh}[1]{\stackrel{#1}{\slashedrightarrow}}

\makeatletter

\newdimen\w@dth

\def\setw@dth#1#2{\setbox\z@\hbox{\scriptsize $#1$}\w@dth=\wd\z@
\setbox\@ne\hbox{\scriptsize $#2$}\ifnum\w@dth<\wd\@ne \w@dth=\wd\@ne \fi
\advance\w@dth by 1.2em}

\def\t@^#1_#2{\allowbreak\def\n@one{#1}\def\n@two{#2}\mathrel
{\setw@dth{#1}{#2}
\mathop{\hbox to \w@dth{\rightarrowfill}}\limits
\ifx\n@one\empty\else ^{\box\z@}\fi
\ifx\n@two\empty\else _{\box\@ne}\fi}}
\def\t@@^#1{\@ifnextchar_ {\t@^{#1}}{\t@^{#1}_{}}}

\def\t@left^#1_#2{\def\n@one{#1}\def\n@two{#2}\mathrel{\setw@dth{#1}{#2}
\mathop{\hbox to \w@dth{\leftarrowfill}}\limits
\ifx\n@one\empty\else ^{\box\z@}\fi
\ifx\n@two\empty\else _{\box\@ne}\fi}}
\def\t@@left^#1{\@ifnextchar_ {\t@left^{#1}}{\t@left^{#1}_{}}}

\def\two@^#1_#2{\def\n@one{#1}\def\n@two{#2}\mathrel{\setw@dth{#1}{#2}
\mathop{\vcenter{\hbox to \w@dth{\rightarrowfill}\kern-1.7ex
                 \hbox to \w@dth{\rightarrowfill}}%
       }\limits
\ifx\n@one\empty\else ^{\box\z@}\fi
\ifx\n@two\empty\else _{\box\@ne}\fi}}
\def\tw@@^#1{\@ifnextchar_ {\two@^{#1}}{\two@^{#1}_{}}}

\def\tofr@^#1_#2{\def\n@one{#1}\def\n@two{#2}\mathrel{\setw@dth{#1}{#2}
\mathop{\vcenter{\hbox to \w@dth{\rightarrowfill}\kern-1.7ex
                 \hbox to \w@dth{\leftarrowfill}}%
       }\limits
\ifx\n@one\empty\else ^{\box\z@}\fi
\ifx\n@two\empty\else _{\box\@ne}\fi}}
\def\t@fr@^#1{\@ifnextchar_ {\tofr@^{#1}}{\tofr@^{#1}_{}}}

\newdimen\W@dth
\def\setW@dth#1#2{\setbox\z@\hbox{$#1$}\W@dth=\wd\z@
\setbox\@ne\hbox{$#2$}\ifnum\W@dth<\wd\@ne \W@dth=\wd\@ne \fi
\advance\W@dth by 1.2em}

\def\T@^#1_#2{\allowbreak\def\N@one{#1}\def\N@two{#2}\mathrel
{\setW@dth{#1}{#2}
\mathop{\hbox to \W@dth{\rightarrowfill}}\limits
\ifx\N@one\empty\else ^{\box\z@}\fi
\ifx\N@two\empty\else _{\box\@ne}\fi}}
\def\T@@^#1{\@ifnextchar_ {\T@^{#1}}{\T@^{#1}_{}}}

\def\T@left^#1_#2{\def\N@one{#1}\def\N@two{#2}\mathrel{\setW@dth{#1}{#2}
\mathop{\hbox to \W@dth{\leftarrowfill}}\limits
\ifx\N@one\empty\else ^{\box\z@}\fi
\ifx\N@two\empty\else _{\box\@ne}\fi}}
\def\T@@left^#1{\@ifnextchar_ {\T@left^{#1}}{\T@left^{#1}_{}}}

\def\Tofr@^#1_#2{\def\N@one{#1}\def\N@two{#2}\mathrel{\setW@dth{#1}{#2}
\mathop{\vcenter{\hbox to \W@dth{\rightarrowfill}\kern-1.7ex
                 \hbox to \W@dth{\leftarrowfill}}%
       }\limits
\ifx\N@one\empty\else ^{\box\z@}\fi
\ifx\N@two\empty\else _{\box\@ne}\fi}}
\def\T@fr@^#1{\@ifnextchar_ {\Tofr@^{#1}}{\Tofr@^{#1}_{}}}

\def\Two@^#1_#2{\def\N@one{#1}\def\N@two{#2}\mathrel{\setW@dth{#1}{#2}
\mathop{\vcenter{\hbox to \W@dth{\rightarrowfill}\kern-1.7ex
                 \hbox to \W@dth{\rightarrowfill}}%
       }\limits
\ifx\N@one\empty\else ^{\box\z@}\fi
\ifx\N@two\empty\else _{\box\@ne}\fi}}
\def\Tw@@^#1{\@ifnextchar_ {\Two@^{#1}}{\Two@^{#1}_{}}}

\def\to{\@ifnextchar^ {\t@@}{\t@@^{}}}
\def\from{\@ifnextchar^ {\t@@left}{\t@@left^{}}}
\def\tofro{\@ifnextchar^ {\t@fr@}{\t@fr@^{}}}
\def\To{\@ifnextchar^ {\T@@}{\T@@^{}}}
\def\From{\@ifnextchar^ {\T@@left}{\T@@left^{}}}
\def\Two{\@ifnextchar^ {\Tw@@}{\Tw@@^{}}}
\def\Tofro{\@ifnextchar^ {\T@fr@}{\T@fr@^{}}}

\makeatother

\title{Double Fibrations}
\author{Geoffrey Cruttwell, Michael Lambert, Dorette Pronk, and Martin Szyld}

\begin{document}
\tdplotsetmaincoords{60}{40}
\maketitle

\begin{abstract}
    This paper defines double fibrations (fibrations of double categories) and describes their key examples and properties. In particular, it shows how double fibrations relate to existing fibrational notions such as monoidal fibrations and discrete double fibrations, proves a representation theorem for double fibrations, and shows how double fibrations are a type of internal fibration.  
\end{abstract}

\section{Introduction}

The goal of this paper is to define double fibrations (fibrations of double categories) and state and prove their essential properties.  However, it is useful to first put this structure in context: why are double fibrations an important structure to consider?  We will discuss three reasons for looking at double fibrations: (i) they form a common generalization for recent work on monoidal fibrations \cite{MV}, discrete double fibrations \cite{Lambert2021}, and 2-fibrations \cite{Buckley} (ii) recent work in applied category theory \cite{DJM} has demonstrated a need for double fibrations; (iii) to further the study of double-categorical logic, a precise theory of double fibrations is required.  Let us consider each of these in more detail.

In \cite{MV}, the authors define a notion of \emph{monoidal fibration}.  By definition, this is a pseudo monoid in a 2-category of fibrations.  The paper also defines the indexed version of monoidal fibrations, and proves the equivalence of monoidal fibrations to their indexed counterpart.  In \cite{Lambert2021}, the author defines a notion of \emph{discrete double fibration}.  By definition, this is a category object in a 2-category of \emph{discrete} fibrations.  The paper also gives an indexed version of such fibrations, and again proves the equivalence between the fibred notion and its indexed counterpart.  Thus, a natural generalization of both of these settings is to consider category objects in a 2-category of fibrations.  In fact, both papers refer to the possibility of such a generalization: see \cite[Remark 3.16]{MV} and \cite[Section 5.2]{Lambert2021}.  Similarly, various other works \cite{Gray_book, Hermida, Buckley} have developed fibration notions for 2-categories and bicategories. 

Coming from a different angle, \cite{DJM} has shown how double fibrations can be useful in investigating aspects of applied category theory.  In particular, the paper considers double categories associated to certain types of dynamical systems.  An important aspect of this work is the definition of an indexed double category and a double Grothendieck construction associated to it.  However, the paper does not give a general definition of double fibration, nor prove an equivalence between double fibrations and indexed double categories.

Finally, it is well-known that aspects of logic can be encoded as fibrations of categories (for example, see \cite{Jacobs}). Recent investigations, however, have suggested that for some aspects of logic, double categories might be even more useful than mere categories.  As such, to continue this development, a full theory of double fibrations is required.

This paper accomplishes the task.  In addition to defining double fibrations, the paper works out two key theoretical aspects of double fibrations: a representation theorem for double fibrations, showing how they are equivalent to \emph{indexed double categories} (Theorem \ref{theorem:RepresentationTheoremFinalForm}), and a result showing that double fibrations can be seen as a type of internal fibration (Corollary \ref{coro:double_fibration_iff_internal}).  

Moreover, we also show how aspects (i) and (ii) described above fit into the theory of double fibrations.  In particular, we show how monoidal fibrations, discrete double fibrations, and 2-fibrations (via the quintet construction) are particular examples of our general theory.  We also show that the double Grothendieck construction of \cite{DJM} is a particular example of our more general Grothendieck construction: see Example \ref{ex:jaz-myers}.   

We will not consider here how the theory of double fibrations can be used to further aspect (iii), namely double categorical logic, but anticipate this will occur in future work.

The paper is laid out as follows:
\begin{itemize}
    \item In Section \ref{sec:double_fibration}, we give the definition of a double fibration, present some basic examples, and then show how many existing generalizations of the fibration notion can be seen as double fibrations.  
    \item In Section \ref{sec:rep_theorem}, we state and prove a representation theorem for double fibrations.  In particular, in Theorem \ref{theorem:RepresentationTheoremFinalForm} we show that double fibrations correspond to ``span-valued lax double pseudo functors''.  
    \item Finally, in Section \ref{sec:internal}, we show how a double fibration can be seen as a type of internal fibration in the sense of Street \cite{internal_fibration}.
\end{itemize}

We believe this is only the start of the story of double fibrations.  As noted above, there are potential applications in both applied category theory and categorical logic, and we anticipate that developments in these areas will be greatly assisted by the results of this paper.  

\smallskip

\noindent{\bf Acknowledgements.}  
We acknowledge the support of the Natural Sciences and Engineering Research Council
of Canada (NSERC).
The last-named author also acknowledges the support of the Atlantic Association for Research in the Mathematical Sciences (AARMS).

We wish to thank Bob Par\'e, David Jaz Myers, Joe Moeller, and Christina Vasilakopoulou for several helpful conversations on the subject.

\section{Double Fibrations} \label{sec:double_fibration}

In this section we introduce double fibrations as pseudo category objects in the 2-category of fibrations, and we show how they correspond to double functors between double categories satisfying certain properties.

\subsection{Preliminaries on Fibrations}

We briefly recall here the classical notion of fibration of categories, mainly to fix the notation that we will use throughout the paper.

\begin{define} \label{def:plain_old_cartesian}
Let $P: \ct{E} \mr{} \B$ be a functor between categories. 
An arrow $f\colon X\to Y$ of $\ct{E}$ is \textbf{Cartesian} if
        whenever $g\colon Z\to Y$ is an arrow of $\ct{E}$ for which there is a morphism $h\colon PZ\to PX$ making a commutative triangle in $\B$ as on the right 
            $$\xymatrix{
                Z \ar@{-->}[d]_{\hat h} \ar@/^1.0pc/[dr]^g & & & PZ \ar[d]_h \ar@/^1.0pc/[dr]^{Pg} & \\
                X \ar[r]_f & Y & & PX\ar[r]_{Pf} & PY
            }$$
            there is a unique $\hat h\colon Z\to X$ in $\ct{E}$ over $h$ making a commutative triangle in $\ct{E}$ as on the left above.
\end{define}

\begin{define} \label{def:plain_old_fibration}
Let $P: \ct{E} \mr{} \B$ be a functor between categories. $P$ is a {\em fibration} when any arrow $u: B \mr{} PE$ in $\B$ has a Cartesian lift $u^*E: B^* \mr{} E$. A choice of a Cartesian lift for each arrow of $\B$ (and the family of so-chosen arrows) is called a {\em cleavage} of the fibration. 
Given two fibrations $P: \ct{E} \mr{} \B$ and $P': \ct{E}' \mr{} \B'$,
a {\em morphism} $f$ between them is given by a pair of functors $f^{\top}: \ct{E} \mr{} \ct{E}'$, $f^{\bot}: \ct{B} \mr{} \ct{B}'$, such that
$P' f^\top = f^\bot P$ and 
$f^{\top}$ preserves the Cartesian arrows. 
When $P$ and $P'$ are each equipped with a cleavage, 
$f$ is said to be {\em cleavage-preserving} when $f^{\top}$ maps the arrows of the cleavage of $P$ to arrows in the cleavage of $P'$. 
\end{define}

We fix throughout this section a 2-category $\mathfrak K$. We recall the construction of three {\em 2-categories of arrows} in $\mathfrak K$, or cylinder 2-categories  that can be defined (see for example \cite{Ben}), and we relate them to the 2-categories of fibrations.

\begin{define}
There is a 2-category $\Arrl(\K)$ whose objects are given by arbitrary arrows in $\K$. Arrows are given by squares filled by a 2-cell, and 2-cells are given by a pair of 2-cells satisfying an equation: for such a 2-cell $(\mu^\top,\mu^\bot)$ as on the left below, the two 2-cells $u_1^\bot f \Mr{} f' u_2^\top$ on the right obtained by pasting $\alpha_1$ with $\mu^\top$ and by pasting $\mu^\bot$ with $\alpha_2$ are required to be equal
$$
\vcenter{\xymatrix{
(C,D,f)
\ar@<2ex>[rr]^{(u_1^\top, u_1^\bot, \alpha_1)}
\ar@<-2ex>[rr]_{(u_2^\top, u_2^\bot, \alpha_2)}
\ar@{}[rr]|{\Downarrow (\mu^\top, \mu_\bot)} &&
(C',D',f')
}}
\qquad \qquad
\vcenter{\xymatrix{
&& C' \ar[dd]^{f'} \\
C \ar[dd]_{f} \ar@/^3ex/[rru]^{u_1^\top} \ar@/_3ex/[rru]_{u_2^\top}  \ar@{}[drr]|{\Uparrow \alpha_1 \;\; \Uparrow \alpha_2}  \ar@{}[rru]|{\Downarrow \mu^\top} \\
&& D' \\
D \ar@/^3ex/[rru]^{u_1^\bot} \ar@/_3ex/[rru]_{u_2^\bot} \ar@{}[rru]|{\Downarrow \mu^\bot}
}}
$$

$\Arrl(\K)$ comes equipped with two {\em top} and {\em bottom} 2-functors $\Arrl(\K) \rightarrow \K$. 
There are also two sub-2-categories $\Arrs(\K) \subseteq \Arrp(\K) \subseteq \Arrl(\K)$, with the same objects and 2-cells, where for the arrows the 2-cell filling the square is required to be either invertible or an identity. 

When $\K = \cat$, we also have a 2-category $\fib$ of fibrations with a chosen cleavage over arbitrary bases, whose arrows are given by squares such that the functor on top is Cartesian, that comes equipped with a forgetful 2-functor $\fib \mr{} \Arrs(\cat)$ (that is full on 2-cells).
Similarly, we have a 2-category $\mbf{Opf}$ of opfibrations. 
Finally, we have the sub-2-category 
$\cfib \subseteq \fib$, with the same objects and 2-cells, where for the arrows the functor on top is required to be cleavage-preserving. 
\end{define}

\begin{prop}\label{prop:pullbackspointwise}
\begin{enumerate}
    \item 2-pullbacks are computed pointwise in $\Arrs(\K)$ and preserved by the inclusion 2-functors $\Arrs(\K) \mr{} \Arrp(\K)$ and $\Arrs(\K) \mr{} \Arrl(\K)$.
    \item 2-pullbacks are computed pointwise in $\cfib$ and preserved by both the inclusion  2-functor $\cfib \mr{} \fib$ and the forgetful 2-functor $\cfib \mr{} \Arrs(\cat)$.
\end{enumerate}
\end{prop}

In order not to interrupt the exposition, we have deferred the proof of this proposition to Appendix \ref{sec:proof_in_Appendix}.

\subsection{Double Fibrations}

Next, we recall the definition of a pseudo category object in a 2-category $\mathfrak K$.

\begin{define}[\S 1, \cite{Nelson}]\label{def:pseudocategory} 
    Let $\mathfrak K$ denote a 2-category. A \textbf{pseudo category} in $\mathfrak K$, denoted by $\bb C$, consists of 
    \begin{enumerate}
        \item An object of objects $C_0$ and an object of arrows $C_1$, together with morphisms $\src,\tgt\colon C_1\rightrightarrows C_0$, standing for source and target respectively, such that the iterated 2-pullbacks of $C_1$ over $C_0$ (as appearing in the diagrams below) exist.
        \item two morphisms $y\colon C_0\to C_1$ and $\otimes \colon C_1\times_{C_0}C_1 \to C_1$ standing for identity and composition, respectively, where $C_1\times_{C_0}C_1$ denotes the 2-pullback
            \begin{equation} \label{eq:2pbinpseudocatobject}
            \xymatrix{
                C_1\times_{C_0}C_1 \ar[d]_{\pi_1} \ar[r]^-{\pi_2} & C_1 \ar[d]^{\src} \\
                C_1 \ar[r]_{\tgt} & C_0
            }
            \end{equation}
        \item and finally iso 2-cells
            $$\xymatrix{ 
                \ar@{}[drr]|{\substack{\mathfrak a\\\cong}} C_1\times_{C_0}\times C_1\times_{C_0}C_1 \ar[d]_{\otimes \times 1} \ar[rr]^{1\times \otimes} && C_1\times_{C_0}C_1 \ar[d]^\otimes \\
                C_1\times_{C_0}C_1 \ar[rr]_\otimes && C_1
            }$$
        and
            $$\xymatrix{  
                C_1 \ar[rr]^-{\langle y, 1\rangle} \ar@{=}@/_1.0pc/[drr] &  & \ar@{}[dl]|{\substack{\mathfrak l\\\cong}} C_1\times_{C_0}C_1 \ar[d]|\otimes \ar@{}[dr]|{\substack{\mathfrak r\\ \cong}}& & \ar[ll]_-{\langle 1,y\rangle} \ar@{=}@/^1.0pc/[dll] C_1 \\
                & & C_1 & & 
            }$$
    \end{enumerate}
    satisfying axioms as in the reference. A \textbf{category} in $\mathfrak K$ is a pseudo category in $\mathfrak K$ where $\mathfrak a$, $\mathfrak l$ and $\mathfrak r$ are identities. Given two arrows $\src,\tgt\colon C_1\rightrightarrows C_0$, we refer to the remaining data (satisfying the axioms), as a \textbf{pseudo-category structure for $\src$ and $\tgt$}.
\end{define}

\begin{remark}
    We will assume internal categories as above are \textbf{normalized} in the sense that $\mfrk l$ and $\mfrk r$ are identities. This follows the double-categorical conventions of \cite{GrandisPare1999}.
\end{remark}

\begin{example}
    A pseudo category in a 1-category $\C$ with finite limits regarded as a discrete 2-category is an internal category in the ordinary sense.
\end{example}

\begin{example} \label{ex:monoidascategory}
When $\mathfrak K$ has a terminal object $1$, and $C$ is an object of $\mathfrak K$, a pseudo-category structure for $C\rightrightarrows 1$ is a pseudo-monoid structure for $C$ (with the monoidal structure induced by the Cartesian product). In particular, for $\mathfrak K = \cat$, a pseudo-category structure for $\ct{C}\rightrightarrows 1$ is a monoidal-category structure for $\ct{C}$.
\end{example}

\begin{example} 
    A pseudo category in $\cat$ is a (pseudo) double category \cite[\S 7.1]{GrandisPare1999}. Thus, a double category $\D$ consists of objects $A, B, C, D \ldots$, arrows $f\colon A\to B$ (the objects and arrows of $\mathscr D_0$); proarrows $m\colon C\hto D$ and cells of the form
        $$\xymatrix{
            \ar@{}[dr]|{\theta} A\ar[d]_f\ar[r]^m|-@{|} & B \ar[d]^g \\
            C \ar[r]_n|-@{|} & D
        }$$
    (that is, the objects and arrows of $\mathscr D_1$). For such a cell, $m$ is the ``internal domain" of $\theta$, $n$ is its ``internal codomain," (the domain and codomain in $\mathscr D_1$) and $f$ and $g$ are the external source and target, respectively. External composition is denoted by $\otimes\colon \mathscr D_1\times_{\mathscr D_0}\mathscr D_1\to\mathscr D_1$ and the external unit is denoted by $y\colon \mathscr D_0\to \mathscr D_1$. Double categories will always be assumed to be pseudo in the sense that external composition is associative up to coherent isomorphism. They will also be assumed to be normalized in that the unit comparisons are strict identities.
\end{example}

The following is the main definition of this paper. 

\begin{define} \label{def:double_fibration}
    A \textbf {double fibration} $P\colon \bb E\to\bb B$ is given by a 
    pseudo-category structure in $\fib$, $(P_1,P_0,\src,\tgt,\iota,\phi,{\mathfrak a},{\mathfrak r},{\mathfrak l})$, such that $\src,\tgt: P_1 \rightrightarrows P_0$ preserve the chosen cleavages.
\end{define}

A {\bf cleavage} for a double fibration is a choice of a cleavage for $P_0$ and a cleavage for $P_1$ such that $\src,\tgt: P_1 \rightrightarrows P_0$ are cleavage-preserving. Note that, by definition, any double fibration is equipped with a cleavage, which we omit from the notation.

\begin{remark}
    Note that since $\src,\tgt: P_1 \rightrightarrows P_0$ are required to preserve the chosen cleavages, that is, are arrows in $\cfib$, by Proposition \ref{prop:pullbackspointwise} the pullbacks appearing in Definition \ref{def:pseudocategory} exist in $\fib$ as required by the definition.
\end{remark}

\begin{remark}
    Likewise, a \textbf{double opfibration} is a pseudo-category structure in $\mbf{Opf}$ such that the source and target preserve chosen ``opcleavages."
\end{remark}

We will show below how a double fibration can also be seen as a \emph{strict} double functor between pseudo double categories satisfying certain properties (for details, see Definition \ref{def:unwinded_double_fibration}).

\begin{define}[\S 2, \cite{Nelson}] \label{def:laxfunctor}
    A \textbf{lax functor} $F\colon \mathbb C\to\D$ between pseudo categories $\mathbb C$ and $\D$ in a 2-category $\K$ with 2-pullbacks consists of two arrows $F_0\colon C_0\to D_0$ and $F_1\colon C_1\to D_1$ and comparison morphisms
        $$\xymatrix{
            \ar@{}[dr]|{\substack{\phi\\\Rightarrow}} C_1\times_{C_0}C_1 \ar[d]_{F_1\times_{F_0}F_1} \ar[r]^-\otimes & C_1 \ar[d]^{F_1} & & \ar@{}[dr]|{\substack{\iota\\\Rightarrow}} C_0 \ar[d]_{F_0} \ar[r]^y & C_1 \ar[d]^{F_1} \\
            D_1\times_{D_0}D_1 \ar[r]_-\otimes & D_1 & & D_0 \ar[r]_y & D_1
        }$$
    strictly preserving source and target and satisfying three familiar equalities between 2-cells that can be found in \cite[(2.5),(2.6)]{Nelson} involving $\phi$, $\iota$, and the associators and unitors of $\bb C$ and $\bb D$. A lax functor is \textbf{pseudo} {(resp.~\textbf{strict})} if the cells $\phi$ and $\iota$ are invertible {(resp.~identities)}. A lax functor is \textbf{unitary} if $\iota$ is an identity.
\end{define}

We note that the three conditions in Definition \ref{def:laxfunctor} can be recovered by replacing in items 2 and 3 of Definition \ref{def:lax_double_pseudo_functor} the modifications $\Phi, \Lambda$, and $\Rho$ by identities.

\begin{example} \label{ex:laxfunctorbetweendoublecategories}
    A lax (resp. pseudo, resp. strict) functor in $\cat$ is an ordinary lax (resp. pseudo, resp. strict) functor between double categories. Lax functors between double categories are also sometimes called lax double functors.
\end{example}

\begin{example} \label{ex:monoidfunctoraslaxfunctor}
    Assume $\K$ has a terminal object $1$. When $C,D$ are pseudo monoids in $\K$, viewed as pseudo categories $\bb C, \bb D$ as shown in Example \ref{ex:monoidascategory}, a lax (resp. pseudo, resp. strict) functor $F\colon \mathbb C\to\D$ is the same as a lax (resp. strong, respt strict) monoidal functor $F\colon C\to D$.
\end{example}

Any pseudo-category structure in one of the three arrow 2-categories introduced above, for which the source and target are strict, is the same as a lax (resp. pseudo, resp. strict) functor:

\begin{lemma} \label{lem:laxfunctorascategoryobject}
Given a diagram  $\src: F_1 \rightarrow F_0 \leftarrow F_1:\tgt$  in $\Arrs(\K)$, a pseudo-category structure
$\bb{F} = (F_1,F_0,\src,\tgt,\iota,\phi,{\mathfrak a},{\mathfrak r},{\mathfrak l})$ 
in $\Arrl(\K)$ (resp. $\Arrp(\K)$, resp. $\Arrs(\K)$)
is the same as two pseudo-category structures $\bb{C}$ and $\bb{D}$ that are obtained by applying the top and bottom 2-functors respectively to $\bb{F}$, that we denote
$$\bb{C} = (C_1,C_0,\src^\top,\tgt^\top,\iota^\top,\phi^\top,{\mathfrak a}^\top,{\mathfrak r}^\top,{\mathfrak l}^\top), \qquad 
\bb{D} = (D_1,D_0,\src^\bot,\tgt^\bot,\iota^\bot,\phi^\bot,{\mathfrak a}^\bot,{\mathfrak r}^\bot,{\mathfrak l}^\bot);
$$
and a lax (resp. pseudo, resp. strict) functor 
$$
F = (F_1,F_0,\iota,\phi)
$$
\end{lemma}

\begin{proof}
First recall from Proposition \ref{prop:pullbackspointwise} that 2-pullbacks of strict diagrams are computed pointwise in the three 2-categories. Then it's just translation: the `three familiar equalities' of Definition \ref{def:laxfunctor} are precisely those making ${\mathfrak a},{\mathfrak r}$, and ${\mathfrak l}$ into 2-cells of  $\Arrl(\K)$.
\end{proof}

When $\mathfrak K = \cat$, the same proof of Lemma \ref{lem:laxfunctorascategoryobject} (using item 2 of Proposition \ref{prop:pullbackspointwise} instead of item 1) yields:

\begin{cor}[of Lemma \ref{lem:laxfunctorascategoryobject}] \label{cor:3differentcatobjectsinFib}
Given a diagram $\src: P_1 \rightarrow P_0 \leftarrow P_1: \tgt$  in $\cfib$, a pseudo-category structure
$(P_1,P_0,\src,\tgt,\iota,\phi,{\mathfrak a},{\mathfrak r},{\mathfrak l})$ 
in  $\fib$ (resp. $\cfib$)
is the same as two pseudo-double-category structures obtained by applying the top and bottom 2-functors,
$$\bb{E} = (\ct{E}_1,\ct{E}_0,\src^\top,\tgt^\top,\iota^\top,\phi^\top,{\mathfrak a}^\top,{\mathfrak r}^\top,{\mathfrak l}^\top), \qquad 
\bb{B} = (\ct{B}_1,\ct{B}_0,\src^\bot,\tgt^\bot,\iota^\bot,\phi^\bot,{\mathfrak a}^\bot,{\mathfrak r}^\bot,{\mathfrak l}^\bot)
$$
and a strict functor 
$$
P = (P_1,P_0,\iota,\phi)
$$ 
such that $\iota$ and $\phi$ are arrows in $\fib$ (resp. $\cfib$). \qed
\end{cor}

We give now the remaining 2-categorical structure on pseudo categories in a 2-category $\K$. The 2-cells are (strictly natural) transformations as in \cite{Nelson}.

\begin{define}[\S 3, \cite{Nelson}] \label{define:TransformationOfInternalPseudoFunctors}
    Let $F,G\colon \bb C\rightrightarrows\bb D$ denote lax functors of pseudo categories in a 2-category $\K$ with 2-pullbacks. A \textbf{transformation} $\tau\colon F\Rightarrow G$ consists of a pair of 2-cells $\tau_0\colon F_0 \Rightarrow G_0$ and $\tau_1\colon F_1\Rightarrow G_1$ satisfying the conditions that the internal domains and codomains are well-typed and  
    two equalities involving also the comparison cells of $F$ and $G$, that can be found in \cite[(3.2)]{Nelson}. We denote by $\pscat(\K)$ the 2-category of (normalized) pseudo categories, unitary pseudo functors, and transformations in $\K$. There are then further 2-categories 
        \[
            \pscat_s(\K) \subseteq \pscat(\K) \subseteq \pscat_\ell(\K)
        \]
    whose arrows are respectively the strict and lax functors.
\end{define}

Again, we note that the reader can also recover the two equalities in Definition  \ref{define:TransformationOfInternalPseudoFunctors} by replacing the modifications $\Tau$ and $\Iota$ by identities in items 2 and 3 of Definition \ref{define:DoublePseudoNaturalTransformation}.

\begin{example} \label{ex:MoninsidePsCat}
    Let $\K$ be a 2-category with 2-pullbacks and a terminal object, seen as a Cartesian monoidal 2-category. The constructions in Example \ref{ex:monoidfunctoraslaxfunctor} extend to an inclusion 2-functor $\Mon_\ell(\K) \mr{} \pscat_\ell(\K)$, that is full-and-faithful on 1- and 2-cells, where $\Mon_\ell(\K)$ is the 2-category of pseudo monoids in $\K$, lax morphisms, and 2-cells, recalled for example in \cite[pp. 1169-1170]{MV}. This restricts of course to  $\Mon(\K) \mr{} \pscat(\K)$ and  $\Mon_s(\K) \mr{} \pscat_s(\K)$
\end{example}

The following is the {\em 1-dimension up version} of Lemma \ref{lem:laxfunctorascategoryobject}:

\begin{lemma} \label{lem:transformationaslaxfunctor}
Let $F: \bb{C} \mr{} \bb{D}$, $G: \bb{A} \mr{} \bb{B}$ be two lax functors between 
pseudo categories in $\K$, corresponding by Lemma \ref{lem:laxfunctorascategoryobject} to two pseudo-category structures $\bb{F},\bb{G}$ in $\Arrl(\K)$ (with strict source and targets).
Then a lax (resp. pseudo, resp. strict) functor $T = (T_1,T_0,\iota,\phi): \bb{F} \mr{} \bb{G}$, with 
$T_1 = (T_1^\top,T_1^\bot,\tau_1)$ and
$T_0 = (T_0^\top,T_0^\bot,\tau_0)$,
is the same as two lax (resp. pseudo, resp. strict) functors ${H}$ and ${K}$ that are obtained by applying the top and bottom 2-functors respectively to $T$, that we denote
$$
{H} = (T_1^\top,T_0^\top,\iota^\top,\phi^\top) \qquad 
{K} = (T_1^\bot,T_0^\bot,\iota^\bot,\phi^\bot)
$$
and a transformation $\tau = (\tau_1,\tau_0)$ filling the square
$$
\xymatrix{
        \ar@{}[dr]|{\Uparrow \, \tau} {\bb C}\ar[d]_F\ar[r]^H & {\bb A} \ar[d]^G \\
        {\bb D} \ar[r]_K & {\bb B}
    }
$$
When $\bb{F},\bb{G}$ are pseudo categories in $\Arrp(\K)$, resp. $\Arrs(\K)$, internal lax functors $T$ in these two 2-categories correspond as above to $(H,K,\tau)$ such that $\tau$ is invertible, resp. an identity (and as above, $H$ and $K$ are {\em as strict as} $T$).
\end{lemma}

\begin{proof}
As in Lemma \ref{lem:laxfunctorascategoryobject}, this is just translation: the two equalities of Definition \ref{define:TransformationOfInternalPseudoFunctors} instantiated at $\tau$ are precisely those making $\iota$ and $\phi$ into 2-cells of $\Arrl(\K)$.
\end{proof}

\begin{define} \label{define:EquivalenceOfPseudoCats}
An \textbf{equivalence} of pseudo categories in a 2-category $\K$ with 2-pullbacks is an equivalence in the 2-category $\pscat(\K)$.
\end{define}

\begin{example} \label{ex:the2categoryDblcatv}
When $\K = \cat$, so that internal lax functors are lax functors between pseudo double categories, a transformation between them in the sense of Definition \ref{define:TransformationOfInternalPseudoFunctors} is the usual notion of transformation between lax functors. We have in this way a 2-category $\Dblcatv = \pscat(\cat)$ of pseudo double categories, pseudo double functors, and transformations; as well as its lax- and strict-functor versions $\Dblcatvl$, $\Dblcatvs$ that will be relevant in Section \ref{sec:internal}.

The construction in Example \ref{ex:MoninsidePsCat} yields in this case the usual way of seeing monoidal categories as double categories, as in Example \ref{ex:monoidascategory}, in the form of a 2-functor $\Moncat \mr{} \Dblcatv$, whose lax and strict versions we omit.
\end{example}

We consider now an arbitrary strict double functor $P$ between two pseudo double categories $\bb{E}$ and $\bb{B}$, given by the following diagram (where we omit the higher associativity and unitality invertible 2-cells of the pseudo double categories, which do not play an important role here)

\begin{equation} \label{eq:unwinded}
\xymatrix@R=3pc@C=3pc{
{\ct{E}}_1 \times_{{\ct{E}}_0} {\ct{E}}_1 \ar[d]_{P_1 \times_{P_0} P_1} \ar[r]^-{\otimes^\top} &
{\ct{E}}_1 \ar[d]^{P_1} \ar@<1.25ex>[r]^{\src^\top} \ar@<-1.25ex>[r]_{\tgt^\top} & {\ct{E}}_0 \ar[d]^{P_0} \ar[l]|{y^\top} \\ 
{\ct{B}}_1 \times_{{\ct{B}}_0} {\ct{B}}_1 \ar[r]^-{\otimes^\bot} & {\ct{B}}_1 \ar@<1.25ex>[r]^{\src^\bot}  \ar@<-1.25ex>[r]_{\tgt^\bot} & {\ct{B}}_0 \ar[l]|{y^\bot}}
\end{equation}

\begin{define}\label{def:unwinded_double_fibration}
A strict double functor $P$ \textbf{defines a double fibration} when $P_0$ and $P_1$ are fibrations, there exist a cleavage for $P_0$ and a cleavage for $P_1$ such that $\src^\top$ and $\tgt^\top$ are simultaneously cleavage-preserving, and $y^\top$ and $\otimes^\top$ preserve the Cartesian arrows.
\end{define}

Note that such a strict double functor $P\colon \bb E\to\bb B$, together with a choice of the cleavages of $P_0$ and $P_1$, defines indeed by Corollary \ref{cor:3differentcatobjectsinFib} a double fibration $P\colon \bb E\to\bb B$ as in Definition \ref{def:double_fibration}. 

In Section \ref{sec:internal}, we will show that the conditions in Definition \ref{def:unwinded_double_fibration} are equivalent to requiring $P$ to be an internal fibration in the 2-category $\Dblcatv$.

\subsection{Examples}

Our first example of a double fibration is a double-categorical analogue of the domain fibration $\dom\colon \C^{\mbf 2}\to\C$ for a category $\C$. 

\begin{example} \label{example:DomainDoubleFibration}
    Let $\bb D$ be an arbitrary double category.  Let $\bb D^{\mbf 2}$ denote the double category with category of objects $\bb D_0^{\mbf 2}$ and category of proarrows $\bb D_1^{\mbf 2}$. Its external operations are those induced from the external operations coming with $\bb D$. The projection $\dom\colon \bb D^{\mbf 2}\to \bb D$ taking an arrow to its domain and a cell to its (internal) domain is a double fibration. If $C$ is any object of $\bb D$, the double category $\bb D/C$ is formed by the comma construction of \S 1.7 of \cite{GrandisPare1999}. Equivalently, it is the sub-double category of $\bb D^{\mbf 2}$ consisting of those squares in $\bb D_0$ and in $\bb D_1$ that project to $C$ and the proarrow identity on $C$ when taking internal codomains. Again the domain projection $\dom\colon \bb D/X\to\bb D$ is a double fibration.
\end{example}
 
\begin{example} \label{example:ImageFromSpanstoRel}
    The double functor $\im\colon \Span \to\Rel$ taking a span $A \leftarrow S \to B$ to the image of $S\to A\times B$ is a double opfibration.
\end{example}

\begin{example}[Family Double Fibration] \label{example:FamilyDoubleFibration}
    Let $\C$ denote an ordinary small category. Define the pseudo double category $\dblfam(\C)$ in the following way:
        \begin{enumerate}
            \item objects: set-indexed families $\lbrace C_i\rbrace_{i\in I}$ of objects of $\C$, equivalently, functors $f\colon I\to \C$ with $I$ viewed as a discrete category;
            \item arrows $f\to g$: pairs $(h,\alpha)$ consisting of a set function $h\colon I\to J$ and a transformation $\alpha\colon f\Rightarrow gh$;
            \item proarrows: natural transformations of the form
                $$\xymatrix{
                    \ar@{}[dr]|{\mathclap{\substack{\theta \\ \Rightarrow }}} S \ar[d]_{d_0} \ar[r]^{d_1} & Y \ar[d]^g \\
                    X \ar[r]_f & \C
                }$$
            where $X\xleftarrow{d_0} S \xrightarrow{d_1} Y$ is a span of sets;
            \item cells: a cell from
                $$\xymatrix{ 
                    \ar@{}[dr]|{\mathclap{\substack{\theta \\ \Rightarrow }}} S \ar[d]_{d_0} \ar[r]^{d_1} & Y \ar[d]^g 
                        & \ar@{}[dr]|{\mbox{to}} && \ar@{}[dr]|{\mathclap{\substack{\delta \\ \Rightarrow }}} T \ar[d]_{d_0} \ar[r]^{d_1} & W \ar[d]^k \\
                    X \ar[r]_f & \C & && Z \ar[r]_h & \C
                }$$
            is a morphism of spans
                $$\xymatrix{ 
                    X\ar[d]_p & \ar[l]_{d_0} S \ar[r]^{d_1}  \ar[d]^q& Y \ar[d]^r \\
                    Z & \ar[l]^{d_0} T \ar[r]_{d_1} & W
                }$$
            together with cells $\alpha \colon f\Rightarrow hp$ and $\beta\colon g\Rightarrow kr$ providing the external source and target; this data should satisfy a cocycle condition, namely, that the equation $(\beta\ast d_1)\theta = (\delta\ast q)(\alpha\ast d_0)$ holds. This condition amounts to asking that each square 
                $$\xymatrix{
                    \cdot \ar[d]_{\alpha_{d_0s}} \ar[r]^{\theta_s} & \cdot \ar[d]^{\beta_{d_1s}} \\
                    \cdot \ar[r]_{\delta_{qs}} & \cdot
                 }$$
            commutes in $\C$. Internal composition of cells is just composing span morphisms and whiskering 2-cells. The square above makes it easy to check that this is well-defined, associative and unital. External composition of proarrows and cells is given by pulling back and using the composition coming with $\C$. External units are identity spans with identity transformations. This makes $\dblfam(\C)$ a pseudo double category.
        \end{enumerate}
    Notice that $\dblfam(\C)_0$ is the ordinary category of set-indexed families in $\C$. The projection to $\set$, denoted by $\Pi_0\colon \dblfam(\C)_0\to \set$ is a split fibration (see Definition 1.2.1 \cite{Jacobs}). The chosen Cartesian arrow above a given set function is given by reindexing the given family by that set function.  We then extend $\Pi_0$ to a strict double functor $\Pi\colon \dblfam(\C)\to \Span$ by taking a proarrow to its underlying span and a cell to its underlying span morphism.
\end{example}

\begin{prop} 
    The projection $\Pi\colon \dblfam(\C)\to \Span$ is a split double fibration.
\end{prop}
    \begin{proof} 
        $\Pi_0$ is a split fibration since it is the ordinary family fibration for $\C$ if we disregard the double structure. For the $\Pi_1$-part, take a cell $\delta$ as in the display above and a span morphism $q$ as immediately above, but viewed as a cell in $\Span$. Define the domain of the lift to be the cell $\delta\ast q$ of the form
            $$\xymatrix{\ar@{}[dr]|{\mathclap{\substack{\delta\ast q \\ \Rightarrow }}} S \ar[d]_{\partial_0} \ar[r]^{\partial_1} & Y \ar[d] \\
                    X \ar[r] & \C
            }$$
        By construction this defines a cell in $\dblfam(\C)$ to $\delta$ satisfying the required condition. The external source and target functors are thus cleavage-preserving by construction. External composition is splitting-preserving by the uniqueness aspect of the universal property of pullbacks.
\end{proof}

\begin{example}[Codomain Fibration] \label{example:DoubleCodomainFibration}
    If $\C$ is a finitely-complete category, the codomain functor $\cod\colon \C^\mbf 2\to\C$ is an ordinary cloven fibration. A cleavage is given by a choice of pullbacks. Thus, if we assume that ``finite limits" always means ``equipped with a choice of finite limits" then any such fibration is always cloven. A double-categorical analogue is given in the following way. Let $\bb D$ denote a pseudo category in the 2-category of finitely-complete categories, limit-preserving functors and transformations. So, in particular $\D_0$ and $\D_1$ have finite limits. Assume that $\src$ and $\tgt$ preserve finite limits on the nose; $\otimes$ and $y$ then preserve finite limits up to isomorphism. Let $\bb D^{\mbf 2}$ denote the double category
        $$\xymatrix{
            \bb D_1^\mbf 2 \times_{\bb D_0^\mbf 2} \bb D_1^\mbf 2 \ar[r]^-{\otimes^\mbf 2} & \bb D_1^\mbf 2 \ar@<1.2ex>[r]^{\src^\mbf 2} \ar@<-1.2ex>[r]_{\tgt^\mbf 2} & \ar[l]|{y^\mbf 2} \bb D_0^\mbf 2
        }$$
    with associators and unitors induced from $\bb D$. This is indeed a pseudo category in $\cat$ since the hom-functor $(-)^\mbf 2 = [\mbf 2,-]$ is a right adjoint and thus preserves finite limits.
\end{example}

\begin{prop}
    The functors $\cod_0\colon \bb D_0^{\mbf 2} \to \bb D_0$ and $\cod_1\colon \bb D_1^{\mbf 2} \to \bb D_1$ are the underlying functors of a double fibration $\cod \colon \bb D^{\mbf 2}\to\bb D$.
\end{prop}
\begin{proof}
    In the first place, since $\otimes$ is a functor, $\cod$ with underlying functors $\cod_0$ and $\cod_1$ is a double-functor:
        $$\xymatrix@R=3pc@C=3pc{
            \bb D_1^\mbf 2 \times_{\bb D_0^\mbf 2} \bb D_1^\mbf 2 \ar[d]_{\cod_1 \times_{\cod_0} \cod_1} \ar[r]^-{\otimes^\mbf 2} &
            \bb D_1^\mbf 2 \ar[d]^{\cod_1} \ar@<1.25ex>[r]^{\src^\mbf 2} \ar@<-1.25ex>[r]_{\tgt^\mbf 2} & \bb D_0^\mbf 2 \ar[d]^{\cod_0} \ar[l]|{y^\mbf 2} \\ 
            \bb D_1 \times_{\bb D_0} \bb D_1 \ar[r]_-{\otimes} & \bb D_1 \ar@<1.25ex>[r]^{\src}  \ar@<-1.25ex>[r]_{\tgt} & \bb D_0 \ar[l]|{y}
        }$$
    Both functors are fibrations since the base categories are finitely complete. That $\src$ and $\tgt$ are finite-limit preserving implies that $\src^\mbf 2$ and $\tgt^\mbf 2$ are cleavage-preserving by the discussion above. On the other hand, $\otimes^\mbf 2$ and $y^\mbf 2$ are Cartesian-morphism preserving, since $\otimes$ and $y$ are finite-limit-preserving in the usual, up-to-iso, sense. 
\end{proof}

\subsection{Relations to Other 2-Dimensional Notions of Fibration}

In this section we consider how double fibrations subsume several other fibration notions.  

\begin{example}
    Since any discrete fibration is a cloven fibration, any discrete double fibration \cite{Lambert2021} is a (cloven) double fibration.
\end{example}

\begin{example} \label{ex:monoids_in_Fib_as_double_fib}
The result in \cite[Prop. 3.2]{MV}, describing pseudo monoids in the Cartesian-monoidal 2-category $\fib$, can be recovered from Corollary \ref{cor:3differentcatobjectsinFib} when $P_0 = 1$ (the terminal object $1 \mr{} 1$ in $\fib$), in which case $\src,\tgt: P \mr{} 1$ are unique, by noting that the structures above, in this case, correspond to the ones in op. cit. Indeed, for an arbitrary fibration $P: {\ct{E}} \mr{} {\ct{B}}$:

\begin{itemize} \label{ex:monoidalfibrationleftleg}
    \item As shown in Example \ref{ex:monoidascategory}, in this case a pseudo-category structure for $\src$ and $\tgt$ is the same as a pseudo-monoid structure for $P$ (in $\fib$ with its Cartesian product, i.e. that of a {\em monoidal fibration} as in [Shu08], see \cite[Prop. 3.2]{MV}).

    \item Similarly the two pseudo-double-category structures $\E$ and $\bb{B}$ above amount to monoidal-category structures for ${\ct{E}}$ and ${\ct{B}}$.

    \item As shown in Example \ref{ex:monoidfunctoraslaxfunctor}, in this case $P$ is a strict double functor between the (pseudo) double categories if and only if it is a strict monoidal functor between the monoidal categories.

    \item The condition that  $\otimes_{\ct{E}}$ is Cartesian-morphism preserving appears in \cite{MV} and [Shu08] as well ($y_{\ct{E}}$ is trivially Cartesian-morphism preserving when ${\ct{E}}_0 = 1$ in \eqref{eq:unwinded}).
\end{itemize}
\end{example}

The previous example is showing how any monoidal fibration can canonically be seen as a double fibration. The next one will show that a different monoidal object that is considered in the context of (op)fibrations can also be seen as a double fibration.

\begin{example} \label{ex:monoidalfibrationrightleg} \label{ex:monoids_in_Fib(B)_as_double_fib}
Recall that, for a fixed {\em base} category ${\ct{B}}$, there is a Cartesian monoidal structure ($\boxtimes$,$1_{\ct{B}}$) in the 2-category $\fib({\ct{B}})$, of fibrations with ${\ct{B}}$ as a base, that is different to the Cartesian monoidal structure of $\fib$ (details can be found in \cite[p. 74]{Jacobs}, see also \cite[p.1176]{MV}).
Somewhat surprisingly, the description in \cite[p.1176]{MV} of pseudo monoids in $\fib({\ct{B}})$ can also be recovered from Corollary \ref{cor:3differentcatobjectsinFib}, in the same sense as Example \ref{ex:monoidalfibrationleftleg}, as follows. 
Let $P: {\ct{E}} \mr{} {\ct{B}}$ be a fibration.  In the context of Corollary \ref{cor:3differentcatobjectsinFib}, set $P_1 = P$,
$P_0 = 1_{\ct{B}}$, and 
$\src = \tgt = (P,1_{\ct{B}}): P \mr{} 1_{\ct{B}}$. The reader should compare the following two diagrams respectively with \eqref{eq:cubeoffibrations} and \eqref{eq:unwinded}:

\begin{equation}
\begin{tikzpicture}[scale=2.5,tdplot_main_coords,baseline=(current bounding box.center)]
\draw[shorten >=0.3cm,shorten <=.3cm,<-] (0,0,0) node{${\ct{B}}$} -- node[left]{$P$} (0,0,1) node{${\ct{E}}$} ;
\draw[shorten >=0.3cm,shorten <=.3cm,<-] (1,1,0) node{${\ct{B}}$} -- node[right]{$P$} (1,1,1) node{${\ct{E}}$} ;
\draw[shorten >=0.3cm,shorten <=.3cm,<-] (1,0,0) node{${\ct{B}}$} -- node[left]{$1_{\ct{B}}$} (1,0,1) node{${\ct{B}}$} ;
\draw[shorten >=0.3cm,shorten <=.3cm,->] (0,0,1) -- node[above] {$P$} (1,0,1);
\draw[shorten >=0.3cm,shorten <=.3cm,<-] (1,0,1) -- node[above] {$P$} (1,1,1);
\draw[shorten >=0.3cm,shorten <=.3cm,->](0,0,0) --  node[below] {$1_{\ct{B}}$}  (1,0,0) ;
\draw[shorten >=0.3cm,shorten <=.3cm,<-] (1,0,0) -- node[below] {$1_{\ct{B}}$} (1,1,0);
\end{tikzpicture}
\qquad \leadsto \qquad
\begin{tikzpicture}[scale=3.5,tdplot_main_coords,baseline=(current bounding box.center)]
\draw[shorten >=0.3cm,shorten <=.3cm,<-] (0,0,0) node{${\ct{B}}$} -- node[left]{$P$} (0,0,1) node{${\ct{E}}$} ;
\draw[shorten >=0.3cm,shorten <=.3cm,<-] (1,1,0) node{${\ct{B}}$} -- node[right]{$P$} (1,1,1) node{${\ct{E}}$} ;
\draw[dashed, shorten >=0.3cm,shorten <=.3cm,<-] (0,1,0) node{${\ct{B}}$} -- node[left]{$P \boxtimes P$} (0,1,1) node{${\ct{E}} \times_{\ct{B}} {\ct{E}}$};
\draw[shorten >=0.3cm,shorten <=.3cm,<-] (1,0,0) node{${\ct{B}}$} -- node[left]{$1_{\ct{B}}$} (1,0,1) node{${\ct{B}}$} ;
\draw[shorten >=0.3cm,shorten <=.3cm,->] (0,0,1) -- node[below] {$P$} (1,0,1);
\draw[shorten >=0.3cm,shorten <=.3cm,<-] (0,0,1) -- node[above] {} (0,1,1);
\draw[shorten >=0.3cm,shorten <=.3cm,->] (0,1,1) -- node[above] {}(1,1,1);
\draw[shorten >=0.3cm,shorten <=.3cm,<-] (1,0,1) -- node[below] {$P$} (1,1,1);
\draw[shorten >=0.3cm,shorten <=.3cm,->](0,0,0) --  node[below] {$1_{\ct{B}}$}  (1,0,0) ;
\draw[dashed, shorten >=0.3cm,shorten <=.3cm,<-] (0,0,0) -- node[above] {} (0,1,0);
\draw[dashed, shorten >=0.3cm,shorten <=.3cm,->] (0,1,0) -- node[above] {} (1,1,0);
\draw[shorten >=0.3cm,shorten <=.3cm,<-] (1,0,0) -- node[below] {$1_{\ct{B}}$} (1,1,0);
\end{tikzpicture}
\end{equation}

\begin{equation} \label{eq:unwindedinFibB}
\vcenter{\xymatrix@R=3pc@C=3pc{
{\ct{E}} \times_{{\ct{B}}} {\ct{E}} \ar[d]_{P \boxtimes P} \ar[r]^-{\otimes} &
{\ct{E}} \ar[d]^{P} \ar@<1.25ex>[r]^{P} \ar@<-1.25ex>[r]_{P} & {\ct{B}} \ar[d]^{1_{\ct{B}}} \ar[l]|{y} \\ 
{\ct{B}} \ar[r]^-{1_{\ct{B}}} & {\ct{B}} \ar@<1.25ex>[r]^{1_{\ct{B}}} \ar@<-1.25ex>[r]_{1_{\ct{B}}} & {\ct{B}} \ar[l]|{1_{\ct{B}}}}}
\end{equation}
(Regarding the bottom line in \eqref{eq:unwindedinFibB}, note that a  pseudo-category structure whose structural arrows $\src$ and $\tgt$ are identities is necessarily trivial).
Note that the fibred functors $\otimes$, $y$ are exactly as those appearing in \cite[(29), p.1176]{MV} (labeled respectively $m$ and $j$ in op.cit.). 
In view of this, Corollary \ref{cor:3differentcatobjectsinFib}, when applied to this case, is showing that a pseudo-category structure (in $\fib$, for these $\src = \tgt = (P,1_{\ct{B}})$) is the same as a pseudo-monoid structure for $P$ in $\fib({\ct{B}})$.
\end{example}

Here we recall the notion of a 2-fibration, due to \cite{Hermida} and \cite{Buckley}. The definition of a 2-Cartesian arrow is a categorification of that of a Cartesian arrow, recalled above in Definition \ref{def:plain_old_cartesian}. Note that the references use the term ``Cartesian" where we use ``2-Cartesian." This is just to distinguish between the usual 1-dimensional notion for ordinary fibrations and this 2-dimensional analogue.

\begin{define}[\S 2 \cite{Hermida}, \S 2.1.1 \cite{Buckley}] \label{define:2Cartesian} 
     Let $P\colon \mfrk E\to\mfrk B$ denote a 2-functor. An arrow $f\colon X\to Y$ of $\mfrk E$ is \textbf{2-Cartesian} if
        \begin{enumerate}
            \item $f:X\to Y$ is Cartesian in the usual sense of Definition \ref{def:plain_old_cartesian};
            \item and additionally, whenever $\theta\colon g\Rightarrow k$ is a 2-cell of $\mfrk E$ for which there is a 2-cell $\gamma\colon h \Rightarrow l$ of $\mfrk B$ such that $P\theta = Pf\ast \gamma$ holds, there is a unique lift 2-cell $\hat \gamma\colon \hat h\Rightarrow \hat l$ in $\mfrk E$ over $\gamma$ such that $f\ast \hat \gamma = \theta$ holds.
        \end{enumerate}
\end{define}

\begin{define}[\S 2.1.6-2.1.7 \cite{Buckley}]
    A \textbf{2-fibration} is a 2-functor $P\colon \mfrk E\to\mfrk B$ such that 
        \begin{enumerate}
            \item every morphism $B\to PY$ in $\mfrk B$ has a 2-Cartesian arrow above it;
            \item each functor of hom-categories $P\colon \mfrk E(X,Y)\to \mfrk B (PX,PY)$ is an ordinary fibration;
            \item horizontal compositions of Cartesian 2-cells are again Cartesian.
        \end{enumerate}
    A 2-fibration is \textbf{cloven} if it is equipped with a choice of 2-Cartesian arrows and is locally a cloven fibration. All 2-fibrations are assumed to be cloven.
\end{define}

To understand how double fibrations relate to 2-fibrations, we need to recall the quintet construction \cite[pp. 272-273]{BastEhr}.  The \textbf{double category of quintets} $\bb Q(\K)$ associated to a 2-category $\K$ has as objects and arrows those of $\K$. Its proarrows are arrows of $\K$. Finally, cells are 2-cells of $\K$ of the form
    $$\xymatrix{
        \ar@{}[dr]|{\Downarrow} A \ar[d]_f \ar[r]^m|-@{|} & B \ar[d]^g &\ar@{}[dr]|{:=} & & \ar@{}[dr]|{\Downarrow} A \ar[d]_f \ar[r]^m & B \ar[d]^g \\
        C \ar[r]_n|-@{|} & D & & & C \ar[r]_n & D
    }$$
where all the arrows on the right are just arrows of $\K$. Notice that such cells point from the internal proarrow domain to the internal proarrow codomain. This ends up making them point from the external source to the external target. This departs from the convention of \cite[\S 1.3]{GrandisPare1999} which has cells pointing from source to target. Note that any 2-functor $F\colon \K \to \mfrk L$ induces a strict double functor $\bb Q(F)\colon \bb Q(\K) \to\bb Q(\mfrk L)$.  Via this construction, 2-fibrations are closely related to double fibrations:

\begin{prop}[Quintets] \label{ex:Quintets}
Let $P\colon \mfrk E\to\mfrk B$ be a 2-functor.  Then $P$ is a 2-fibration if and only if $\bb Q(P)\colon \bb Q(\mfrk E)\to\bb Q(\mfrk B)$ is a double fibration.  
\end{prop}
\begin{proof}
    The 2-Cartesian 1-cells $\sigma(f,X)\colon f^*X\to X$, considering just the 1-dimensional aspect of their lifting property, make the underlying 1-functor $\bb Q(P)=|P|\colon |\mfrk E|\to|\mfrk B|$ into an ordinary fibration. It remains to see that $\bb Q(P)_1$ is a fibration and that the external structure maps of $\bb Q(\mfrk E)$ are suitably cleavage-preserving or Cartesian morphism-preserving, as the case may be. So, starting with a cell in the base of the form as at left
        $$\xymatrix{
            \ar@{}[dr]|{\Downarrow\alpha} A \ar[d]_f \ar[r]^u & B \ar[d]^g & & & \ar@{}[dr]|{\Downarrow\hat\alpha}f^*X \ar[d]_{\sigma(f,X)} \ar@/^1pc/[dr]\ar@{-->}[r]^{\hat u} & g^*Y\ar[d]^{\sigma(g,Y)} \\
            PX \ar[r]_{Pv} & PY & & & X\ar[r]_v & Y
        }$$
    the cell in $\bb Q(\mfrk E)_1$ can be constructed as follows. First, owing to the fact that $P$ is locally a fibration, the Cartesian lift $\hat\alpha$ exists. Then, because $\sigma(g,Y)$ is Cartesian, the dashed arrow $\hat u$ exists making a commutative triangle.  The claim is that the right cell is Cartesian with respect to $\bb Q(P)_1$. In brief, the 2-cell is $\bb Q(P)_1$-Cartesian because the two vertical arrows are Cartesian and the 2-cell is Cartesian (both in the sense of Definition \ref{define:2Cartesian}). Before that, notice that the external source and target maps will be cleavage-preserving and that the external identity and composition maps will both be Cartesian-morphism-preserving. The latter is owing to the fact that horizontal compositions of Cartesian 2-cells are again Cartesian. Now, to prove that the cell is Cartesian, suppose there is another cell $\beta$ in $\bb Q(\mfrk E)_1$ such that in the base there is an equality
        $$\xymatrix{
            \ar@{}[dr]|{\Downarrow\theta} PZ \ar[d]_x\ar[r]^{Pw} & PW \ar[d]^y  & & \ar@{}[ddr]|{\Downarrow P\beta} PZ\ar[dd]_{Ph} \ar[r]^{Pw} & PW \ar[dd]^{Pk}\\
            \ar@{}[dr]|{\Downarrow\alpha} A \ar[d]_f \ar[r]_u & B \ar[d]^g & = & &  \\
            PX \ar[r]_{Pv} & PY & & PX \ar[r]_{Pv} & PY
        }$$
    arising from some cell $\theta$. Note that the unique lifts $\hat x$ and $\hat y$ exist making appropriate commutative triangles since $\sigma(f,X)$ and $\sigma(g,Y)$ are Cartesian. The goal is to produce a cell $\hat \theta$ over $\theta$ making a commutative triangle of cells in $\bb Q(\mfrk E)_1$. For this, use the 2-dimensional lifting property of $\sigma(g,Y)$. First note that identity 2-cells in $\bb Q(\mfrk E)_1$ are Cartesian and since horizontal composition preserves Cartesian 2-cells $\hat\alpha\ast\hat x$ is Cartesian over $\alpha\ast x$. Therefore, there exists a unique lift of $g\ast\theta$ in $\bb Q(\mfrk E)$ whose composite with $\hat\alpha\ast\hat x$ is precisely $\beta$ as in 
        $$\xymatrix{
            \ar@{}[dr]|{\Downarrow \widehat{g\ast\theta}} Z \ar[d]_x\ar[r]^{w} & W \ar[dd]^k  & & \ar@{}[ddr]|{\Downarrow \beta} Z\ar[dd]_{h} \ar[r]^{w} & W \ar[dd]^{k}\\
            \ar@{}[dr]|{\Downarrow\hat\alpha} A \ar[d]_f \ar@/^1pc/[dr] & & = & &  \\
            X \ar[r]_{v} & Y & & X \ar[r]_{v} & Y
        }$$
    Second, note that $k$ and $\dom(\hat \alpha)$ both factor through $\sigma(g,Y)$ via the unique lifts $\hat x$ and $\hat y$ above. So, the lifted cell is of the form
        \begin{equation}
            \widehat{g\ast\theta}\colon \sigma(g,Y)\hat y w \Rightarrow \sigma(g,Y)\hat u\hat x.
        \end{equation}
    using these factorizations. Now, by the 2-dimensional lifting property of $\sigma(g,Y)$, there is a unique 2-cell $\hat\theta$ such that 
        \begin{equation}
            \sigma(g,Y)\hat\theta = \widehat{g\ast\theta}
        \end{equation} 
    Therefore, by the definition of external composition in $\bb Q(\mfrk E)_1$, this cell $\hat\theta$ composes with $\alpha$ in $\bb Q(\mfrk E)_1$ to give $\beta$ as required. It is unique by construction.
        
    For the converse, since $\bb Q(P)_0$ is a fibration, so is $|P|$. It needs to be seen that the chosen Cartesian morphisms are 2-Cartesian and that locally $P$ is a fibration. The latter, however, is simply an application of the fact that $\bb Q(P)_1$ is a fibration and restricting to the vertically globular cells. Thus, we will prove that the chosen Cartesian morphisms are 2-Cartesian.  Again, we show this using the fact that $\bb Q(P)_1$ is a fibration, but using horizontally globular cells instead. Take as given a 2-cell $\alpha\colon h\Rightarrow k$ with $h,k\colon X\rightrightarrows Y$. Let $\sigma(f,Y)\colon f^*Y\to Y$ denote the chosen lift of $f\colon A\to PY$. Suppose now that there is a cell in the base $\theta\colon u\Rightarrow v$ with $f\ast\theta = P\alpha$. The goal is to produce a lift $\hat\theta$ satisfying $\sigma(f,Y)\ast\hat\theta = \alpha$. But one has only to view these cells as living in the total category of $P_1\colon \bb Q(\mfrk E)_1\to\bb Q(\mfrk B)_1$. Since this is a cloven fibration, there is a unique cell $\hat\theta$ as below making an equality of cells
        $$\xymatrix{
            \ar@{}[dr]|{\hat\theta}X\ar[r]|-@{|} \ar[d]_{\hat v} & X \ar[d]^{\hat u} & & \ar@{}[ddr]|{P\alpha} X\ar[dd]_k \ar[r]|-@{|} & X \ar[dd]^h \\
            \ar@{}[dr]|{1} f^*Y \ar[d]_{\sigma(f,Y)} \ar[r]|-@{|} & f^*Y \ar[d]^{\sigma(f,Y)} & = & & \\
            Y\ar[r]|-@{|} & Y & & Y \ar[r]|-@{|} & Y
        }$$
    where all the proarrows are identities; that is, actually identity arrows in $\mfrk E$. The identity cell on $\sigma(f,Y)$ is Cartesian because $y$ preserves Cartesian arrows. The arrows $\hat u$ and $\hat v$ are precisely those given by the fact that $\sigma(f,Y)$ is Cartesian, since $\src$ and $\tgt$ are cleavage-preserving functors. Thus, $\hat \theta$ as above is the required cell.
\end{proof}

Taking a different angle, one might like to know whether starting with a double functor, the fibration properties of either the corresponding ``vertical 2-functor" or ``horizontal pseudo functor" can be characterized in terms of those of the original double functor. For this purpose, recall that every double category $\bb D$ has (1) a \textbf{horizontal bicategory} $\mcal H(\bb D)$ obtained by discarding the ordinary arrows and taking the ``vertically globular" cells; and (2) a \textbf{vertical 2-category} $\mcal V(\bb D)$ obtained in a similar but dual manner by forgetting the proarrows and keeping only the ordinary arrows and ``horizontally globular cells." Each construction induces a suitable functor. That is, if $F\colon \bb D\to \bb E$ is a double functor, then there is a corresponding 2-functor $\mcal V(P)\colon \mcal V(\bb D)\to\mcal V(\bb E)$ and a corresponding pseudo functor $\mcal H(P)\colon \mcal H(\bb D)\to\mcal H(\bb E)$, in each case given, essentially, by restricting the original double functor. Likewise, each 2-category $\K$ has specialized double categories associated to it. There is the \textbf{vertical double category} $\bb V(\K)$ given as a sub-double category of $\bb Q(\K)$ by taking just the vertical structure with identity proarrows. There is also the \textbf{horizontal double category} $\bb H(\mfrk E)$ given by taking just the proarrow stucture and only identity ordinary arrows. Any 2-functor $F\colon \K\to \mfrk L$ has corresponding double functors $\bb V(P)$ and $\bb H(P)$ between vertical or horizontal double categories. Notice that $\mcal V(\bb V(F)) = F$ for any 2-functor $P$. Likewise $\mcal H(\bb H(F)) = F$ holds, but the operations are not genuine inverses since reversing their order does not in general produce the same double functor (either the vertical or horizontal structure is lost in the process).

\begin{prop} \label{prop:DoubleFibrationToVerticalFibrations}
    If $P\colon \bb E\to \bb B$ is a double fibration, then
        \begin{enumerate}
            \item $\mcal VP$ has enough 2-Cartesian arrows; and
            \item $\mcal HP$ is locally a fibration.
        \end{enumerate}
\end{prop}
\begin{proof}
    The second statement is immediate because $P_1$ is a fibration. Applying this property to the vertically globular cells, it follows that $\mcal HP$ is a fibration locally. We prove the first statement in detail. Since $P_0$ is (cloven) fibration, given $f\colon A\to PX$ in $\bb B$, there is a chosen Cartesian arrow $\sigma(f,X)\colon f^*X\to X$ in $\bb E$ over $f$. The claim is that this is 2-Cartesian. As set-up take a cell with prorarrow identities
        $$\xymatrix{
            \ar@{}[dr]|{\beta} Z \ar[d] \ar[r]|-@{|} & Z \ar[d] \\
            X \ar[r]|-@{|} & X
        }$$
    and suppose there is a fill cell $\theta$ in $\bb B$ satisfying $f\ast\theta = P\beta$. Since $P$ is a double fibration, the external identity $y\colon \bb E_0\to\bb E_1$ preserves Cartesian arrows. So there is a unique $\hat\theta$ in $\bb E$ satisfying
        $$\xymatrix{
            \ar@{}[dr]|{\hat\theta} Z \ar[d] \ar[r]|-@{|} & Z \ar[d] & &  \ar@{}[ddr]|{\beta} Z \ar[dd] \ar[r]|-@{|} & Z \ar[dd] \\
            \ar@{}[dr]|{y_{\sigma}} f^*X \ar[d]_{\sigma(f,X)} \ar[r]|-@{|} & f^*X \ar[d]^{\sigma(f,X)} & = &  &  \\
            X \ar[r]|-@{|} & X & &  X \ar[r]|-@{|} & X 
        }$$
    where again all proarrows in sight are identities. Since these cells are thus all horizontally globular, this proves the required statement, namely, that $\sigma(f,X)$ is 2-Cartesian viewed as a 2-cell of $\mcal VP$.
\end{proof}

The converse is not true, however.

\begin{example}
    If $P$ is a 2-functor $P\colon \mfrk E\to\mfrk B$ with enough 2-Cartesian arrows, then the induced double functor $\bb VP\colon \bb V\mfrk E\to\bb V\mfrk B$ has enough 2-Cartesian arrows since $\mcal V(\bb V(P)) = P$. Since there are no non-identity proarrows, its horizontal bicategory is locally a fibration. In other words, for any such $P$, the corresponding $\bb V(P)$ satisfies the two conditions of the conclusion of the proposition. Thus, for a counterexample to the converse of the proposition, it suffices to exhibit a 2-functor $P$ with enough 2-Cartesian arrows for which $\bb V(P)_1$ is not a fibration. For this take the domain projection 2-functor $P = \dom\colon \cat/\C \to \cat$, sending a functor $F\colon \F\to\C$ to its domain category $\F$ and suitably extended to commutative triangles and fibered transformations. This has enough 2-Cartesian arrows, as can be easily checked. Now, a cell of $\bb V(\cat/\C))$ is just  an appropriately fibered transformation $\alpha\colon H\Rightarrow K\colon \F\to\G$, that is, with ``vertical" components (i.e. components over identity morphisms in $\C$). This is sent to $\alpha$ itself via $\dom$. But a given cell in $\cat$ does not necessarily have such a lift, since a given transformation does not necessarily have suitably vertical components. \qed
\end{example}

We can also show that the conclusion of the first statement in Proposition \ref{prop:DoubleFibrationToVerticalFibrations} is the best that can be given. That is, if $P$ is a double fibration, then $\mcal VP$ need not be a 2-fibration.

\begin{example}
    As a special case of Example \ref{example:DomainDoubleFibration}, for any object in a 2-category $C\in \K$, the projection $\dom\colon \bb Q(\K)/C \to\bb Q(\K)$ is a double fibration.  We claim that this is not vertically a 2-fibration. The reason is essentially the same as in the previous example. Namely, taking $\K=\cat$ and restricting to the horizontally globular cells, a lift of a given transformation would need to be suitably fibered along the functor whose domain is the target of the given transformation. But not all such transformations have such suitably vertical components.
\end{example}

Finally, the second statement in Proposition \ref{prop:DoubleFibrationToVerticalFibrations} cannot be strengthened to the conclusion that $\mcal HP$ is a fibration of bicategories \cite[\S 3.1.5]{Buckley}. For this to be the case, $\mcal HP$ would need to be equipped with enough Cartesian arrows, which would be proarrows from the original double-categorical structures. But the concept of a ``Cartesian proarrow'' is evidently one orthogonal to our developments.

\section{Representation Theorem}\label{sec:rep_theorem}

The goal of this section is to state and prove a representation theorem for double fibrations. This says that double fibrations correspond to contravariant span-valued lax ``double pseudo functors" on a double category. That is, for a fixed double category $\bb B$, there is an equivalence of categories 
    \begin{equation} \label{equation:IntroductionRepresentationEquivDisplayforBaseB}
        \dblfib(\bb B)\simeq \dbltwocat(\bb B^{op},\Span(\cat))
    \end{equation}
between double fibrations over $\bb B$ and lax double pseudo functors valued in the ``double 2-category" $\Span(\cat)$. More generally, there is an equivalence 
    \begin{equation} \label{equation:IntroductionRepresentationEquivDisplayforArbitraryBases}
        \dblfib \simeq \mbf{I}\Span(\cat)
    \end{equation}
between double fibrations over arbitrary base double categories and ``indexed spans", that is, contravariant lax double pseudo functors valued in $\Span(\cat)$. The former equivalences (\ref{equation:IntroductionRepresentationEquivDisplayforBaseB}), varying over $\bb B$, occur as the fibers of the latter equivalences (\ref{equation:IntroductionRepresentationEquivDisplayforArbitraryBases}) relative to the projection to double categories. These equivalences are proved in Theorem \ref{theorem:RepresentationTheoremFinalForm} below. 

The proof of the main theorem is accomplished through the device of pseudo monoids in a double 2-category. As a bit of background, recall \cite[Example 2.10]{GeneralizedFramework} that ordinary categories are monoids in the double category of spans in sets. Likewise homomorphisms of such monoids are ordinary functors between categories. This correspondence works for spans in any finitely-complete category $\C$, meaning that internal categories in $\C$ in the strict sense are strict monoids in spans in $\C$. In particular this works for strict double categories. However, most double categories are not strict, but rather are pseudo. Accordingly, one expects, or at least hopes, that a \textit{pseudo}-category is a \textit{pseudo}-monoid in some higher double-structure of spans. This is precisely what is provided by the notion of a ``double 2-category." These structures have the higher-dimensional cells for the coherent associators and unitors appearing in Definition \ref{def:pseudocategory}. The connection to double fibrations is that pseudo categories are the pseudo monoids in suitable double 2-categories. In particular, double fibrations are pseudo monoids in the double 2-category of spans $\Span_c(\fib)$ from Definition \ref{define:Double2CategoryofSpansInFibrations} (see Proposition \ref{prop:SpanValuedLaxFunctorsArePseudoMonoids} below). This approach allows us to balance the ``strict" source and target structure with the ``pseudo" coherence structure in the definition of a double fibration and define a category of double fibrations as a category of pseudo monoids.

The proof of the theorem has two parts. The first is the recognition that strict 2-pullbacks of arrows in $\icat$ for which the triangle is filled with a 2-natural transformation correspond under the elements equivalence $\icat\simeq\fib$ to strict 2-pullbacks of cleavage-preserving morphisms of fibrations. As a consequence, pseudo categories on cleavage-preserving source and target morphisms on one side of the equivalence correspond to pseudo categories on 2-natural source and target structure on the other side. Viewing this through the lens of pseudo monoids, the observation amounts to an equivalence of categories
    \begin{equation}
        \dblfib:=\psmon(\Span_c(\fib))\simeq \psmon(\Span_t(\icat)).
    \end{equation}
Therefore, the representation theorem for double fibrations reduces to the second part, namely, the question of what is a pseudo monoid on 2-natural source and target in $\icat$. The answer is that such a pseudo monoid is a $\Span(\cat)$-valued ``lax double pseudo functor" on a double category viewed as a locally discrete double 2-category. These lax double pseudo functors are an appropriate notion of weak morphism between double 2-categories.

This correspondence between pseudo monoids in $\icat$ and lax double pseudo functors could, with some patience, be proved directly. However, we have already seen that pseudo categories as in Definition \ref{def:pseudocategory} include higher-order associator and unitor 2-cells satisfying a number of coherence conditions. Showing that the associator and unitor 2-cells on the side of lax double pseudo functors induce corresponding associator and unitor cells on the side of pseudo categories in $\icat$ and vice versa is possible but requires considerable effort. This problem exhibits a second utility for pseudo monoids in that equivalent double 2-categories have equivalent categories of pseudo monoids. In more detail, we circumvent the coherence details by producing an equivalent double 2-category $\bb P(\Span(\icat))$ whose pseudo monoids are more clearly lax double pseudo functors. That is, the point of \S \ref{subsection:LaxFunctorsAsPseudoMonoids} is to construct a double 2-category $\bb P(\Span(\icat))$, which is equivalent to $\Span_t(\icat)$, but for which it is easier to derive the rightmost equivalence in the chain:
    \begin{equation}
        \psmon(\Span_t(\icat)) \simeq \psmon(\bb P(\Span(\icat))) \simeq \mbf{I}\Span(\cat)
    \end{equation}
The required equivalence of double 2-categories $\apx\colon \Span_t(\icat)\simeq \bb P(\Span(\icat))$ is the topic of \S \ref{subsection:Translation}. The rest of the proof is simply to pass to pseudo monoids and unpack the data on the right side as a lax double pseudo functor valued in $\Span(\cat)$.
We do this in \S \ref{subsection:LaxFunctorsAsPseudoMonoids}, letting for convenience an arbitrary double 2-category $\bb E$ play the role of $\Span(\cat)$.

\subsection{Double 2-Categories}\label{subsection:Double2Categories}

Here we introduce the notion of a ``double 2-category'' and the appropriately weak morphisms between them. Roughly speaking, double 2-categories are pseudo categories in the 2-category of 2-categories. The corresponding notion of weak morphism between them needs to be defined directly, since it is not obviously a kind of internal functor. As discussed above, the introduction of double 2-categories serves two purposes. First, it will be seen that double fibrations over a fixed double category correspond to certain lax functors valued in a double 2-category of spans. Secondly, this correspondence will be proved by working with pseudo monoids in certain double 2-categories.

\begin{define}
    A \textbf{double 2-category} is a pseudo category in $\twocat$, the 2-category of 2-categories, 2-functors and 2-natural transformations.
\end{define}

\begin{notation} \label{not:underlying_double_cat}
Applying the 2-functor $|-|\colon \twocat \to \cat$, mapping a 2-category to its underlying category, 
we have for each double 2-category $\bb D$ an underlying double category that we denote by $|\bb D|$, or just by $\bb D$ when there is no risk of confusion.
\end{notation}

Generally double 2-categories will be denoted using blackboard font as in $\bb E$. These consist of underlying 2-categories of objects and arrows $\bb E_0$ and $\bb E_1$ together with the usual structure morphisms $\src$, $\tgt$, $\otimes$ and $y$ as described in Definition \ref{def:pseudocategory}. Such an $\bb E$ is like a double category but has extra 2-cells between morphisms and cells. A double 2-category is thus a certain type of ``intercategory" \cite{intercategories}, which is a pseudo category in the 2-category of double categories, pseudo functors and vertical transformations, in which one of the three directions of morphism consists only of identities.

\begin{example} \label{ex:loc_discrete}
  Any ordinary double category $\bb D$ can be seen as a double 2-category $\bb D_d$ with no further 2-cells. Any such double 2-category is said to be \textbf{locally discrete}. When it is clear from the context, we abuse the notation and omit the `$d$'.
\end{example}

\begin{example}
  Any monoidal 2-category $\K$ (from \cite{DayStreet} but cf. \S 2.5 of \cite{MV}) is a double 2-category $\bb K$ with $\bb K_0 = 1$ and $\bb K_1=\K$. Since a monoidal category is a category in monoids, this is like viewing an ordinary monoid as a monoidal category on one object. 
\end{example}

One canonical example is central to subsequent constructions in the proof of our representation theorem. This is the double 2-category of spans in a suitably structured 2-category. Just as spans in sets is a ``primordial" double category (cf. \cite{PareYoneda}), we anticipate that spans in categories is a similarly fundamental example of a double 2-category. We will apply the following construction formalizing this idea to three cases of interest and require some flexibility in doing so. Thus, we ask for a 2-category $\K$ and a class of arrows $\Sigma_1$ of $\K$, generating a full-sub-2-category, that is closed under 2-pullbacks. Think of $\Sigma_1$ as being the class of 2-natural transformations in $\icat$ or the class of cleavage-preserving morphisms in $\fib$.

\begin{construction} \label{construction:SpanCatAsDouble2Category}
    Let $\K$ denote a 2-category. Let $\Sigma_1$ denote a class of morphisms of $\K$ that
        \begin{enumerate}
            \item is closed under composition and identities;
            \item is closed under 2-pullbacks in the sense that the full sub-2-category $\Sigma\subset \K$ consisting of the objects of $\K$, morphisms in $\Sigma_1$ and all 2-cells between them, has all 2-pullbacks.
        \end{enumerate}
    Let $\Span_{\Sigma}(\K)_0$ be $\K$ itself. Let $\Span(\K)_1$ be the following 2-category, namely, the lax limit of the identity cospan on $\Sigma$ which can be constructed as
        $$\xymatrix{
            \Span_{\Sigma}(\K)_1 \ar[d] \ar[r] & \ar@{}[dr]|{\Rightarrow} \Sigma^\mbf 2 \ar[d]_{d_0} \ar[r]^{d_1} & \Sigma \ar@{=}[d] \\
            \ar@{}[dr]|{\Downarrow}\Sigma^\mbf 2 \ar[d]_{d_1} \ar[r]^{d_0} & \Sigma \ar@{=}[d] \ar@{=}[r] & \Sigma \ar@{=}[d] \\
            \Sigma \ar@{=}[r] & \Sigma \ar@{=}[r] & \Sigma.
        }$$
    in $\twocat$, taking cotensors with $\mbf 2$ and a strict 3-pullback in the upper-left corner. Thus, the objects are spans $A \xleftarrow[]{l} S \xrightarrow[]{r} B$ with legs in $\Sigma$ and the arrows are span morphisms
        $$\xymatrix{
            A \ar[d]_f & \ar[l]_l S \ar[d]^v \ar[r]^r & B \ar[d]^g \\
            X & \ar[l]^l T \ar[r]_r & Y.
        }$$
    The 2-cells are triples $(\alpha, \sigma,\beta)$ of cells $\alpha\colon f\Rightarrow f'$, $\sigma\colon v\Rightarrow v'$, $\beta\colon g\Rightarrow g'$ of $\K$ satisfying the compatibility conditions expressed by the two commutative diagrams of composite cells:
        $$\xymatrix{  
            A\ar@/^1.0pc/[d]^{f'} \ar@/_1.0pc/[d]_f \ar@{}[d]|{\alpha} & \ar[l]_l S \ar@/^1.0pc/[d]^{v'}  & \ar@{}[d]|{=} &  A \ar@/_1.0pc/[d]_f  & \ar[l]_l S \ar@/^1.0pc/[d]^{v'} \ar@/_1.0pc/[d]_v \ar@{}[d]|{\sigma} & & & S\ar@/^1.0pc/[d]^{v'} \ar@/_1.0pc/[d]_v \ar@{}[d]|{\sigma} \ar[r]^r &  B \ar@/^1.0pc/[d]^{g'}  & \ar@{}[d]|{=}  & S \ar@/_1.0pc/[d]_v \ar[r]^r &  B \ar@/^1.0pc/[d]^{g'} \ar@/_1.0pc/[d]_g \ar@{}[d]|{\beta}\\
            X &  \ar[l]^l T &  & X &  \ar[l]^l T & & & T \ar[r]_r &   Y &  & T \ar[r]_r &  Y.
        }$$
    Think of these as ``cylinder conditions." Identities and compositions making $\Span(\K)_1$ into a 2-category are inherited from $\mathfrak K$. Now, the external structure giving the double 2-category comes about in the following way. The source and target 2-functors $\src,\tgt\colon \Span(\K)_1\rightrightarrows \K$ send a morphism of spans as above to $f$ and to $g$, respectively; similarly, they send a 2-cell as above to $\alpha$ and to $\beta$, respectively. The identity $i\colon \K\to\Span(\K)_1$ sends an object to the span consisting of identity arrows. External composition of spans is given by 2-pullback in $\Sigma$. By the universal property of 2-pullbacks, this extends to a genuine 2-functor 
        \[ 
            -\otimes -\colon\Span(\K)_1\times_{\K}\Span(\K)_1 \to \Span(\K)_1
        \]
    Up-to-iso associativity also follows from the universal property of 2-pullbacks. This makes a double 2-category $\Span_\Sigma(\K)$.  \qed
\end{construction}

As mentioned above, there are three specialized instances of this construction needed in subsequent developments. These are
        \begin{enumerate}
            \item $\K = \cat$ and $\Sigma = \cat$;
            \item $\K = \fib$ and $\Sigma_1$ is the class of morphisms in $\fib$ that are cleavage-preserving;
            \item $\K = \icat$ and $\Sigma_1$ is the class of morphisms in $\icat$ for which the pseudo natural transformation filling the triangle is a 2-natural transformation.
        \end{enumerate}
    The first results in the canonical example $\Span(\cat)$ of spans of functors between categories, which we leave undecorated since $\Sigma$ is just $\cat$. The other two we define explicitly and indicate their distinctive notation.

\begin{define} \label{define:SpansWRTaSubsetofArrows}
     Let $\Span_t(\icat)$ denote the double 2-category of Construction \ref{construction:SpanCatAsDouble2Category} whose underlying 2-category is $\icat$ and whose proarrows are spans in $\icat$
        $$\xymatrix{
            \ar@{}[dr]|{\substack{\lambda\\\Leftarrow}}\mscr A^{op} \ar[d]_F & \ar@{}[dr]|{\substack{\rho\\\Rightarrow}} \mscr B^{op} \ar[d]^{G} \ar[l]_{S^{op}} \ar[r]^{T^{op}} & \mscr C^{op} \ar[d]^H \\
            \cat & \ar@{=}[l] \cat \ar@{=}[r] & \cat,
        }$$
    whose cells are \emph{2-natural transformations}. That is, $\Sigma_1$ is the class of morphisms of $\icat$ whose cells are 2-natural transformations, 
    satisfying the conditions in Construction \ref{construction:SpanCatAsDouble2Category} by Remark \ref{remark:EquivalenceOfPullbacksICatFib}. The subscripted `$t$' (for transformation)  is in the place of $\Sigma$ as a reminder of this class of morphisms giving the proarrows of the double 2-category.
\end{define}

\begin{remark} \label{rem:SpansWRTaSubsetofArrows}   
It is worth looking more precisely at the structure of $\Span_t(\icat)$ since it will figure prominently in the calculation of \S \ref{subsection:Translation}. In particular, a morphism in $\Span_t(\icat)_1$, displayed as
        $$\xymatrix{
          \ar@{}[dr]|{\substack{\lambda'\\\Leftarrow}}\mscr X^{op} \ar[d]_{M} & \ar@{}[dr]|{\substack{\rho'\\\Rightarrow}} \mscr Y^{op} \ar[d]^{N} \ar[l]_{S^{op}} \ar[r]^{T^{op}} & \mscr Z^{op} \ar[d]^{L} &\ar@{}[dr]|{\substack{\langle \alpha,\beta,\gamma\rangle \\ \longrightarrow}} & &  \ar@{}[dr]|{\substack{\lambda\\\Leftarrow}}\mscr A^{op} \ar[d]_F & \ar@{}[dr]|{\substack{\rho\\\Rightarrow}} \mscr B^{op} \ar[d]^{G} \ar[l]_{S^{op}} \ar[r]^{T^{op}} & \mscr C^{op} \ar[d]^H \\
          \cat & \ar@{=}[l] \cat \ar@{=}[r] & \cat & & &  \cat & \ar@{=}[l] \cat \ar@{=}[r] & \cat,
        }$$
    consists of three pseudo natural transformations
        $$\xymatrix{
            \ar@{}[dr]|{\;\;\;\substack{\alpha\\\Rightarrow}}\mscr X^{op} \ar@/_1pc/[dr]_{M} \ar[r]^{U^{op}} & \mscr A^{op} \ar[d]^F & & \ar@{}[dr]|{\;\;\;\substack{\beta\\\Rightarrow}}\mscr Y^{op} \ar@/_1pc/[dr]_{N} \ar[r]^{V^{op}} & \mscr B^{op} \ar[d]^G & & \ar@{}[dr]|{\;\;\;\substack{\gamma\\\Rightarrow}}\mscr Z^{op} \ar@/_1pc/[dr]_{L} \ar[r]^{W^{op}} & \mscr C^{op} \ar[d]^H \\
            & \cat & & & \cat & & & \cat
        }$$
    satisfying the conditions:
        $$\xymatrix{
            \mscr A^{op} \ar@/_1pc/[dr]_F & \mscr X^{op}\ar@{}[dl]|{\;\;\;\substack{\alpha \\\Leftarrow}} \ar[l]_{U^{op}} \ar[d] & \ar[l] \mscr Y^{op} \ar@{}[dl]|{\substack{\lambda' \\\Leftarrow}} \ar[d]^{G'} &\ar@{}[dr]|{=} & & \mscr A^{op} \ar[d]_F & \mscr B^{op} \ar@{}[dl]|{\substack{\lambda \\\Leftarrow}} \ar[l] \ar[d] & \mscr Y^{op}\ar[l]_{V^{op}} \ar@{}[dl]|{\substack{\beta \\\Leftarrow}\;\;\;} \ar@/^1pc/[dl]^{N}\\
            & \cat & \ar@{=}[l] \cat & & & \cat & \ar@{=}[l]\cat &
        }$$
     and 
        $$\xymatrix{
            \mscr Y^{op} \ar@/_1pc/[dr]_{N} \ar[r]^{V^{op}} & \mscr B^{op}\ar@{}[dl]|{\;\;\substack{\beta \\\Rightarrow}}  \ar[d] \ar[r] & \mscr Z^{op} \ar@{}[dl]|{\substack{\rho \\\Rightarrow}} \ar[d]^{H} &\ar@{}[dr]|{=} & & \mscr Y^{op} \ar[d]_{N} \ar[r] & \mscr Z^{op} \ar@{}[dl]|{\substack{\rho' \\\Rightarrow}} \ar[r]^{W^{op}} \ar[d] & \mscr C^{op} \ar@{}[dl]|{\substack{\gamma \\\Rightarrow}\;\;} \ar@/^1pc/[dl]^{H}\\
            & \cat & \ar@{=}[l] \cat & & & \cat & \ar@{=}[l]\cat &
        }$$
    In other words, a morphism $\langle \alpha, \beta,\gamma\rangle$ forms two prisms with a shared triangular face $\beta$. In equations, these two prism conditions are
        \begin{equation} \label{equations:PrismEquationsCat}
            (\alpha\ast S^{op})\lambda' = (\lambda\ast V^{op})\beta \qquad\text{and}\qquad (\rho\ast V^{op})\beta = (\gamma\ast T^{op})\rho'
        \end{equation}
    Note that two such morphisms are equal if, and only if, their components are equal. \qed
\end{remark}

The final application of Construction \ref{construction:SpanCatAsDouble2Category} is the double 2-category of spans in fibrations with cleavage-preserving morphisms. Proposition \ref{prop:pullbackspointwise_in_App} shows it is well-defined, in the sense that the required pullbacks for composition exist. 

\begin{define} \label{define:Double2CategoryofSpansInFibrations}
    Let $\Span_c(\fib)$ denote the double 2-category as in Construction \ref{construction:SpanCatAsDouble2Category} given by $\Span_c(\fib)_0 = \fib$ and whose 2-category of spans is formed with respect to $\Sigma_1$, the class of cleavage-preserving morphisms of fibrations. Here `$c$' (for cleavage) serves as a reminder that the spans are formed by taking the class of cleavage-preserving morphisms of fibrations.
\end{define}

The next result shows that in fact the two foregoing applications of Construction \ref{construction:SpanCatAsDouble2Category} result in equivalent double 2-categories. Recall from Definition \ref{define:EquivalenceOfPseudoCats} that an equivalence of double 2-categories is an equivalence in the 2-category $\pscat(\twocat)$.

\begin{prop} \label{prop:EquivalentDbl2CatsOfSpans}
    The elements construction $\fib\simeq \icat$ induces an equivalence
        \begin{equation}
            \Span_c(\fib) \simeq \Span_t(\icat)
        \end{equation}
    between the double 2-categories of spans from Definitions \ref{define:Double2CategoryofSpansInFibrations} and \ref{define:SpansWRTaSubsetofArrows}.
\end{prop}
\begin{proof}
    Note, more precisely, that $\fib$ and $\icat$ are equivalent in $\twocat$. The equivalence restricts to one
        $$\xymatrix{
            \fib \ar[r]^\simeq & \icat \\
            \Sigma \ar[u] \ar[r]_\simeq & \Tau \ar[u]
        }$$
    between the full sub-2-category $\Sigma$ of cleavage-preserving morphisms of fibrations and the full sub-2-category $\Tau$ of 2-natural transformations between indexed categories. By construction of $\Span_c(\fib)$ and $\Span_t(\icat)$ as higher-order limits in $\twocat$, the equivalences above lift to equivalences between the arrow 2-categories $\Sigma^\mbf 2\simeq \Tau^\mbf 2$ and the span 2-categories
        \begin{equation}
            \Span_c(\fib)_1 \simeq \Span_t(\icat)_1
        \end{equation}
    both in $\twocat$. As observed in Remark \ref{remark:EquivalenceOfPullbacksICatFib}, the equivalence $\Sigma\simeq\Tau$ is 2-pullback-preserving. Thus, the displayed equivalence between 2-categories of spans immediately above commutes with external composition.
\end{proof}

The utility of this result for our purposes is that these double 2-categories have equivalent categories of pseudo monoids. This will be used in the proof of the representation theorem below.

\subsection{Lax Double Pseudo Functors}
\label{subsection:LaxFunctorsOfDouble2Categories}

Our notion of weak morphism between double 2-categories is that of a ``lax double pseudo functor". This is much like a lax functor between ordinary double categories, but suitably adapted for the 2-categorical structures involved in double 2-categories. However, this is not an internal functor of internal categories in $\twocat$ since the two components $F_0\colon \D_0\to \E_0$ and $F_1\colon \D_1\to\E_1$ are allowed to be pseudo. The notational conventions follow those of Definition \ref{def:laxfunctor} above.

\begin{define} \label{def:lax_double_pseudo_functor}
    A \textbf{lax double pseudo functor} $F\colon \D\to\E$ between double 2-categories $\mathbb D$ and $\mathbb E$ consists of two pseudo functors $F_0\colon \D_0\to \E_0$ and $F_1\colon \D_1\to\E_1$ together with 
        \begin{enumerate}
            \item comparison pseudo natural transformations for external composition and unit
                $$\xymatrix{ 
                    \ar@{}[dr]|-{\substack{\phi \\ \Rightarrow}}\D_1\times_{\D_0}\D_1 \ar[d]_{F_1\times F_1} \ar[r]^-{\otimes} & \D_1 \ar[d]^{F_1} & & & \ar@{}[dr]|{\substack{ \iota \\ \Rightarrow}} \D_0 \ar[r]^y \ar[d]_{F_0} & \D_1 \ar[d]^{F_1} \\
                    \E_1\times_{\E_0}\E_1 \ar[r]_-{\otimes} & \E_1 & & & \E_0 \ar[r]_y & \E_1
                }$$
            \item an invertible associativity modification
                $$\xymatrix{
                    &\D_1^{(3)} \ar[dl]_{F_1^{(3)}}\ar[dd]\ar[dr]^{\otimes \times 1} & & & & &  \D_1^{(3)} \ar@{}[dd]|{\substack{\phi\times 1 \\ \Rightarrow}} \ar[dl]_{F_1^{(3)}}\ar[dr]^{\otimes \times 1} & \\
                    \ar@{}[dr]|{\substack{1\times \phi \\ \Rightarrow}}\E_1^{(3)}\ar[dd]_{1\times \otimes} & \ar@{}[dr]|{\substack{\mathfrak a \\ \cong}} & \D_1^{(2)} \ar[dd]^{\otimes} & & &   \E_1^{(3)} \ar[dd]_{1\times \otimes} \ar[dr] & & \D_1^{(2)}  \ar[dl] \ar[dd]^{\otimes} \\
                    & \D_1^{(2)} \ar[dl]\ar[dr] \ar@{}[dd]|{\substack{\phi \\ \Rightarrow}} & & \ar@{}[r]|{\substack{\Phi \\ \cong}} & &   \ar@{}[dr]|{\substack{\mathfrak a \\ \cong}} &  \E_1^{(2)} \ar[dd] \ar@{}[dr]|{\substack{\phi \\ \Rightarrow}}&  \\
                    \E^{(2)} \ar[dr]_{\otimes} &  & \D_1 \ar[dl]^{F_1} & &  &\E_1^{(2)}\ar[dr]_{\otimes} & & \D_1\ar[dl]^{F_1} \\
                    & \E_1 & & & & &  \E_1 & \\
                }$$
            \item and invertible unitor modifications
                $$\xymatrix{  
                    \D_1 \ar[dd]_{F_1} \ar[dr]_{\langle y, 1\rangle} \ar@{=}@/^2.0pc/[drr] & \ar@{}[d]|(.3){\substack{\mathfrak l \\ \cong}} & & & & \D_1 \ar[dd]_{F_1} \ar@{=}@/^2.0pc/[drr] & & \\
                    \ar@{}[dr]|{\substack{\langle \iota, 1\rangle \\ \Rightarrow}} & \ar@{}[ddr]|{\substack{\phi\\\Rightarrow}} \D_1^{(2)}\ar[dd]\ar[r]_\otimes & \D_1 \ar[dd]^{F_1} &\ar@{}[dr]|{\substack{\Lambda \\ \cong}} & & &   & \D_1\ar[dd]^{F_1} \\
                    \E_1 \ar[dr]_{\langle y,1\rangle} & &  & & & \E_1\ar[dr]_{\langle y,1\rangle} \ar@{=}@/^2.0pc/[drr]  & \ar@{}[d]|(.3){\substack{\mathfrak l \\ \cong}} & \\
                    & \E_1^{(2)}\ar[r]_\otimes & \E_1 & & & & \E_1^{(2)}\ar[r]_\otimes & \E_1 
                }$$
            and
                 $$\xymatrix{  
                    \D_1 \ar[dd]_{F_1} \ar[dr]_{\langle 1,y\rangle} \ar@{=}@/^2.0pc/[drr] & \ar@{}[d]|(.3){\substack{\mathfrak r \\ \cong}} & & & & \D_1 \ar[dd]_{F_1} \ar@{=}@/^2.0pc/[drr] & & \\
                    \ar@{}[dr]|{\substack{\langle 1,\iota \rangle \\ \Rightarrow}} & \ar@{}[ddr]|{\substack{\phi\\\Rightarrow}} \D_1^{(2)}\ar[dd]\ar[r]_\otimes & \D_1 \ar[dd]^{F_1} &\ar@{}[dr]|{\substack{\Rho\\\cong}} & & &   & \D_1\ar[dd]^{F_1} \\
                    \E_1 \ar[dr]_{\langle 1,y\rangle} & &  & & & \E_1\ar[dr]_{\langle 1,y\rangle} \ar@{=}@/^2.0pc/[drr]  & \ar@{}[d]|(.3){\substack{\mathfrak r \\ \cong}} & \\
                    & \E_1^{(2)}\ar[r]_\otimes & \E_1 & & & & \E_1^{(2)}\ar[r]_\otimes & \E_1 
                }$$
        \end{enumerate}
    satisfying the conditions:
        \begin{enumerate}
            \item{}[Well-definition] $F$ preserves sources and targets in that the equations $\src\,F_1 = F_0\src$ and $\tgt\,F_1 = F_0\tgt$ both hold;
            \item{}[Globularity] The comparison cells have trivial external source and target in the sense that
                \begin{enumerate}
                    \item $\src\,\phi = F_1\pi_1$ and $\tgt\,\phi = F_1\pi_2$ 
                    \item $\src\,\iota = F_0$ and $\tgt\,\iota = F_0$
                \end{enumerate}
            all hold; 
            \item{}[Unit Coherence] the composite modifications are equal
                $$
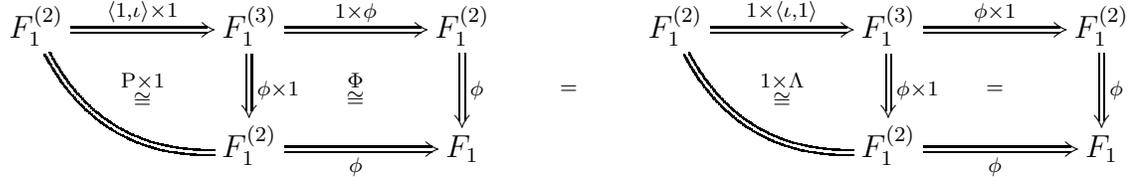
\xymatrix{ 
                    \ar@{}[drr]|{\substack{\Rho\times 1\\\cong}} F_1^{(2)} \ar@{=}@/_1.5pc/[drr] \ar@{=>}[rr]^-{\langle 1,\iota\rangle \times 1} &&\ar@{}[drr]|{\substack{\Phi\\\cong}} F_1^{(3)} \ar@{=>}[d]^{\phi \times 1} \ar@{=>}[rr]^-{1\times \phi} && F_1^{(2)} \ar@{=>}[d]^\phi \ar@{}[drr]|{=} & & \ar@{}[drr]|{\substack{1 \times \Lambda\\\cong}} F_1^{(2)} \ar@{=}@/_1.5pc/[drr]\ar@{=>}[rr]^-{1 \times \langle \iota,1\rangle} &&\ar@{}[drr]|{=} F_1^{(3)} \ar@{=>}[d]^{\phi \times 1} \ar@{=>}[rr]^-{\phi\times 1} && F_1^{(2)} \ar@{=>}[d]^\phi  \\
                    & & F_1^{(2)} \ar@{=>}[rr]_\phi && F_1 & & & & F_1^{(2)} \ar@{=>}[rr]_\phi && F_1
                }$$
                ignoring the tensors $\otimes$ for readability;
            \item{}[Composition Coherence] there is an equality
                $$
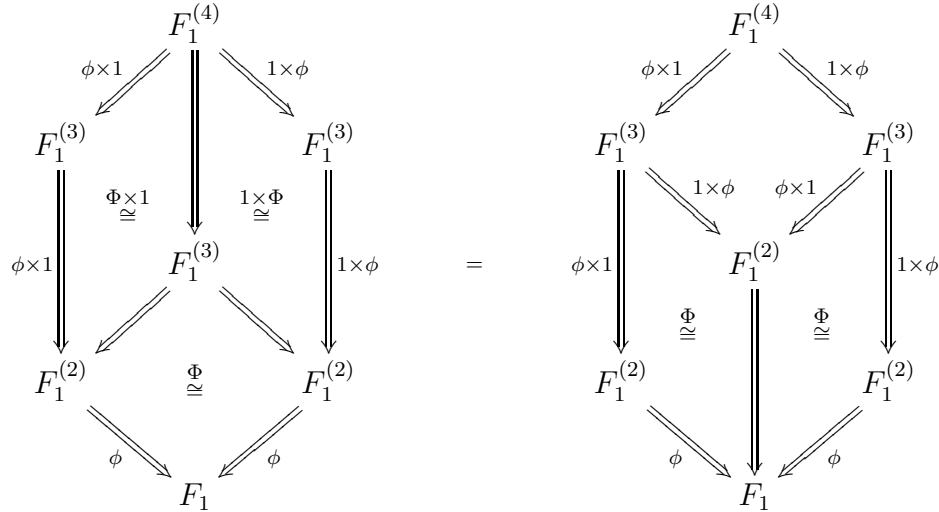
\xymatrix{
                    &F_1^{(4)} \ar@{=>}[dl]_{\phi \times 1}\ar@{=>}[dd]\ar@{=>}[dr]^{1\times \phi} & & & & &  F_1^{(4)}  \ar@{=>}[dl]_{\phi \times 1}\ar@{=>}[dr]^{1\times \phi} & \\
                    \ar@{}[dr]|{\substack{\Phi\times 1 \\ \cong}}F_1^{(3)}\ar@{=>}[dd]_{\phi\times 1} & \ar@{}[dr]|{\substack{1\times \Phi \\ \cong}} & F_1^{(3)} \ar@{=>}[dd]^{1\times\phi} & & &   F_1^{(3)} \ar@{=>}[dd]_{\phi\times 1} \ar@{=>}[dr]^{1\times\phi} & & F_1^{(3)}  \ar@{=>}[dl]_{\phi\times 1} \ar@{=>}[dd]^{1\times \phi} \\
                     & F_1^{(3)} \ar@{=>}[dl]\ar@{=>}[dr] \ar@{}[dd]|{\substack{\Phi \\ \cong}} & & \ar@{}[r]|{=} & &   \ar@{}[dr]|{\substack{\Phi \\ \cong}} &  F_1^{(2)} \ar@{=>}[dd] \ar@{}[dr]|{\substack{\Phi \\ \cong}}&  \\
                    F_1^{(2)} \ar@{=>}[dr]_{\phi} &  & F_1^{(2)} \ar@{=>}[dl]^\phi & & &  F_1^{(2)}\ar@{=>}[dr]_{\phi} & & F_1^{(2)}\ar@{=>}[dl]^\phi \\
                     & F_1 & & & & &  F_1 & \\
                }$$
                again suppressing tensors $\otimes$ and the underlying associator isos.
        \end{enumerate}
    Such a functor $F\colon \D\to\E$ is \textbf{pseudo} if the comparison cells are invertible; and is \textbf{strict} if they are identities. A lax/pseudo/strict double pseudo functor is \textbf{normalized} if the unitor modifications $\Lambda$ and $\Rho$ are identities. A \textbf{lax (resp. pseudo or strict) double 2-functor} is a lax (resp. pseudo or strict) double pseudo functor such that $F_0$ and $F_1$ are 2-functors. A \textbf{unitary} lax double pseudo functor is one for which $\iota$ is an identity.
\end{define}

In general, we will not assume that lax double pseudo functors are unitary. The examples below include those whose composition and unit comparison cells are genuinely lax. We will, however, assume that given arbitrary lax double pseudo functors are normalized.

\begin{example} \label{ex:lax_double_pseudo_functor_Dd--Ed}
    If $\D$ and $\E$ are pseudo double categories seen as locally discrete double 2-categories (see Example \ref{ex:loc_discrete}), then a lax (resp. pseudo or strict) double pseudo functor $F\colon \D_d\to\E_d$ is just a lax (resp. pseudo or strict) functor $F\colon \D\to\E$ as in Definition \ref{def:laxfunctor}.
\end{example}

\begin{example}
Recalling Notation \ref{not:underlying_double_cat} and Example \ref{ex:loc_discrete}, for each double 2-category $\bb D$ we have an inclusion $|{\bb D}|_d \to \bb D$ that is a strict double 2-functor.
\end{example}

\begin{example} \label{ex:for_item_1_1}
Any monoidal 2-category $\mathfrak B$ is the same as a double 2-category $\bb B$ whose object 2-category is the terminal 2-category $1$. 
  A lax monoidal pseudo functor between monoidal 2-categories, as defined in \cite[\S 2.5]{MV} and previously in the references therein, is then the same as a lax double pseudo functor between the double 2-categories.
\end{example}

\begin{example}
   A double pseudo functor \cite{ShulmanLR2011} between strict double categories $F\colon \bb D\to \bb E$ is a specialization of the lax double pseudo functors introduced here. Roughly, double pseudo functors are double functors that are allowed to be pseudo functorial on both ordinary arrows and on proarrows. This means that $F$ restricts to a pseudo functor $\mcal VF\colon \cal V \bb D\to \mcal V \bb E$ of vertical 2-categories and an ordinary functor $F_1\colon \bb D_1\to\bb E_1$ (since $\bb D_1$ is just a category). The globular comparison isos for horizontal composition then provide the laxity cells for external composition in our definition. Of course in this case, they are invertible.
\end{example}

\begin{example}\label{example:SliceLaxDoublePseudoFunctor}
    Start with a double category $\mathbb D$ where both $\mathbb D_0$ and $\mathbb D_1$ have pullbacks, and these are preserved strictly by the external source and target of $\mathbb D$ but only up to isomorphism by the tensor and identity. We will exhibit an example of such a double category in Example \ref{example:SliceOfArrowLaxFunctor} below. In this case, the standard correspondence $\mathbb D_0^{op} \to\cat$ given by $D\mapsto \mathbb D_0/D$ extends to a lax double pseudo functor $\mathbb D^{op}\to\Span(\cat)$. The assignment on proarrows takes $m\mapsto \mathbb D_1/m$ which projects to the slices over the source and target of $m$ via the given source and target functors coming with $\mathbb D$. The transition morphisms associated to arrows $f\colon A\to B$ and to cells $\theta\colon m\Rightarrow n$ come from the existence of ordinary pullbacks in $\mathbb D_0$ and $\mathbb D_1$. That is, for an ordinary morphism $f\colon A\to B$, take $f^*\colon \mathbb D_0/B\to \mathbb D_0/A$ to be given by pulling back an arrow $X\to B$ along $f\colon A\to B$. Similarly, for a cell $\theta\colon m\Rightarrow n$, take $\theta^*\colon \mathbb D_1/n\to\mathbb D_1/m$ to be given by pulling back a cell $\delta\colon p\Rightarrow n$ along $\theta$. In the case of $\D_1$, this results in a well-defined morphism of spans
        $$\xymatrix{
            \mathbb D_0/C \ar[d]_{f^*} & \ar[l]_\src \mathbb D_1/n \ar[d]^{\theta^*} \ar[r]^\tgt &\mathbb D_0/D \ar[d]^{g^*} \\
            \mathbb D_0/A & \ar[l]^\src \mathbb D_1/m \ar[r]_\tgt &\mathbb D_0/B
        }$$ 
    as a result of the fact that $\src,\tgt\colon \mathbb D_1\rightrightarrows \mathbb D_0$ preserve pullbacks on the nose. The assignments are pseudo functorial. Comparison cells associated to composable proarrows $m\colon A\slashedrightarrow B$ and $n\colon B\slashedrightarrow C$ are given by external composition, as indicated by the dashed arrow
        $$\xymatrix{
            \mathbb D_0/A \ar@{=}[ddd] & \ar[l]_{\src} \mathbb D_1/m \ar[r]^{\tgt} & \mathbb D_0/B & \ar[l]_{\src} \mathbb D_1/n \ar[r]^{\tgt} & \mathbb D_0/C \ar@{=}[ddd] \\
            & & \ar[ul] \mathbb D_1/m\times_{\mathbb D_0/B}\mathbb D_1/n \ar@{-->}[dd]^{-\otimes -} \ar[ur] & &  \\
            & & & & \\
            \mathbb D_0/A & & \ar[ll]^{\src} \mathbb D_1/m\otimes n \ar[rr]_\tgt & & \mathbb D_0/C
        }$$
    from the composite of the spans associated to $m$ and $n$ to the span associated to the composite $m\otimes n$. This results in a pseudo natural transformation 
        $$\xymatrix{
            \ar@{}[drr]|{\Rightarrow} \mathbb D_1^{op}\times_{\mathbb D_0^{op}}\mathbb D_1^{op} \ar[d] \ar[rr]^-{\otimes^{op}} & & \mathbb D_1^{op} \ar[d]\\
            \Span(\cat)_1\times_{\cat}\Span(\cat)_1 \ar[rr]_-{\otimes} & & \Span(\cat)_1 
        }$$
    For an arrow $(\theta,\delta)$, the coherence iso
        $$\xymatrix{
            \ar@{}[drr]|{\cong}\mathbb D_1/p\times_{\mathbb D_0/Y}\mathbb D_1/q \ar[d]_{\theta^*\times \delta^*} \ar[rr]^-{\otimes} & & \mathbb D_1/p\otimes q \ar[d]^{(\theta\otimes\delta)^*}\\
            \mathbb D_1/m\times_{\mathbb D_0/B}\mathbb D_1/n \ar[rr]_-{\otimes} & & \mathbb D_1/m\otimes n 
        }$$ 
    arises from the fact that external composition preserves finite limits up to isomorphism. Similarly, the composition condition for pseudo naturality follows by the interchange law in $\mathbb D$. There is no cell condition, since $\D_0$ and $\D_1$ are locally discrete as 2-categories. The proarrow unit comparison cell is the morphism of spans
        $$\xymatrix{
            \mathbb D_0/D \ar@{=}[d] & \ar@{=}[l] \mathbb D_0/D \ar@{-->}[d]^{y_{(-)}} \ar@{=}[r] & \mathbb D_0/D \ar@{=}[d] & & f\colon C\to D \ar@{|->}[d] \\
            \mathbb D_0/D & \ar[l]^{\src} \mathbb D_1/y_D \ar[r]_{\tgt} & \mathbb D_0/D & & y_f\colon y_C\Rightarrow y_D
        }$$
    given by sending an arrow $f\colon C\to D$ to its external unit cell $u_f$ as indicated in the picture above. This results in the required transformation
        $$\xymatrix{
            \ar@{}[dr]|{\Rightarrow} \bb D_0^{op} \ar[d] \ar[r]^{y^{op}}  & \bb D_1^{op} \ar[d] \\
            \cat \ar[r]_-\Delta & \Span(\cat)
        }$$
    The unit and associativity laws as in the definition of a lax double pseudo functor follow from the double-category structure on $\mathbb D$. This is a genuinely lax double pseudo functor that is normalized but not unitary.
\end{example}

\begin{example} \label{example:SliceOfArrowLaxFunctor}
    As a special case of the previous example, given an ordinary category $\C$ with finite limits, the usual associated pseudo functor
        \[ 
        \C^{op}\to\cat \qquad X\mapsto \C/X
        \]
    extends to a lax double pseudo functor on $\C^\mbf 2$. Send a proarrow $m\colon X\slashedrightarrow Y$ (an ordinary arrow $m\colon X\to Y$ of $\C$) to the span
        \[ 
            \C/X \leftarrow \C^{\mathbf 2}/m \to\C/Y
        \]
    where $\C^{\mathbf 2}/f$ is the slice of the arrow category $\C^{\mathbf 2}$ over $f\colon X\to Y$. A cell in $\C^{\mathbf 2}$ is a commutative square and is sent to the span morphism 
        $$\xymatrix{
            \C/Z \ar[d]_{f^*} & \ar[l] \ar[d] \C^\mathbf 2/n \ar[r] & \C/W \ar[d]^{g^*} \\
            \C/X & \ar[l] \C^\mathbf 2/m \ar[r] & \C/Y
        }$$
    where every vertically displayed functor is given by appropriate pullbacks. Laxity cells are given by composition in $\C$. For example, take composable (pro)arrows $m\colon X\to Y$ and $n\colon Y\to Z$. The corresponding laxity cell is the dashed arrow in the diagram
        $$\xymatrix{
            \C/X \ar@{=}[ddd] & \ar[l] \C/m \ar[r] & \C/Y & \ar[l] \C/n \ar[r] & \C/Z \ar@{=}[ddd] \\
            & & \ar[ul] \C/m\times_{\C/Y}\C/n \ar@{-->}[dd]^{-\otimes -} \ar[ur] & &  \\
            & & & & \\
            \C/X & & \ar[ll]^{\src} \C/nm \ar[rr]_\tgt & & \C/Z
        }$$
    given by composition in $\C$. This illustrates how the assignments are genuinely lax functorial.  That is, a square over the composite $nm$ does not necessarily factor as a square over $m$ and one over $n$. Similarly for the unit comparison cell. Again this is normalized, but not unitary.
\end{example}

\begin{example} \label{example:FamilyLaxDoublePseudoFunctor}
    Let $\C$ denote a small category. The ordinary functor $\set^{op}\to \cat$ given by indexed families $I\mapsto [I,\C]$ and pulling back on arrows extends to a lax double pseudo functor on spans $[-,\C]\colon \Span^{op}\to\Span(\cat)$. Take a span of set functions $I \xleftarrow[]{d} S \xrightarrow[]{c} J$ to the span of functors formed by the projections from the comma category occurring as the apex of the diagram
        $$\xymatrix{
            \ar@{}[dr]|{\Rightarrow} d^*/c^* \ar[d] \ar[r] & [J,\C] \ar[d]^{c^*} \\
            [I,\C] \ar[r]_{d^*} & [S,\C].
        }$$
    Accordingly, a morphism of spans of sets is sent to a morphism of spans of functors induced by the universal property of the corresponding comma category. Laxity cells are induced using composition in $\C$. Again these are genuinely non-invertible cells.
\end{example}

\begin{define}[Cf. Definition \ref{define:TransformationOfInternalPseudoFunctors}] \label{define:DoublePseudoNaturalTransformation}
    A \textbf{lax double pseudo natural transformation} of lax double pseudo functors $\tau\colon F\Rightarrow G$ where $F,G\colon \D \rightrightarrows \E$ consists of a pair of pseudo natural transformations $\tau_0\colon F_0\Rightarrow G_0$ and $\tau_1\colon F_1\Rightarrow G_1$ and modifications
        $$\xymatrix{
            \ar@{}[dr]|{\substack{\tau_1\times_{\tau_0}\tau_1\\\Rightarrow}} \D_1\times_{\D_0}\D_1 \ar[d]_{F_1\times_{F_0}F_1}\ar@{=}[r] & \ar@{}[dr]|{\substack{\phi\\\Rightarrow}} \D_1\times_{\D_0}\D_1 \ar[d]\ar[r]^-\otimes & \D_1 \ar[d]^{G_1} &  \ar@{}[dr]|{\substack{\Tau\\\Rrightarrow}} & &  \ar@{}[dr]|{\substack{\phi\\\Rightarrow}} \D_1\times_{\D_0}\D_1 \ar[d]_{F_1\times_{F_0}F_1} \ar[r]^-\otimes & \D_1\ar[d] \ar@{}[dr]|{\substack{\tau_1\\\Rightarrow}} \ar@{=}[r] & \D_1 \ar[d]^{G_1}    \\
            \E_1\times_{\E_0}\E_1 \ar@{=}[r] & \E_1\times_{\E_0}\E_1 \ar[r]_-\otimes & \E_1 & & & \E_1\times_{\E_0}\E_1 \ar[r]_-\otimes & \E_1 \ar@{=}[r] & \E_1
        }$$
    and
        $$\xymatrix{
            \ar@{}[dr]|{\substack{\tau_0\\\Rightarrow}} \D_0 \ar[d]_{F_0} \ar@{=}[r] & \ar@{}[dr]|{\substack{\iota\\\Rightarrow}} \D_0 \ar[r]^y \ar[d] & \D_1 \ar[d]^{G_1} & \ar@{}[dr]|{\substack{\Iota\\\Rrightarrow}} & & \ar@{}[dr]|{\substack{\iota\\\Rightarrow}} \D_0 \ar[d]_{F_0} \ar[r]^y & \ar@{}[dr]|{\substack{\tau_1\\\Rightarrow}} \D_1 \ar@{=}[r] \ar[d] & \D_1 \ar[d]^{G_1} \\
            \E_0 \ar@{=}[r] & \E_0 \ar[r]_y & \E_1 & & & \E_0 \ar[r]_y & \E_1 \ar@{=}[r] & \E_1
        }$$
    satisfying the conditions
    (where we omit the tensors $\otimes$ for readability)
    \begin{enumerate}
        \item{}[Multiplicativity] there is an equality of cells
            $$\xymatrix{
                &F_1^{(3)} \ar@{=>}[dl]_{\tau_1^{(3)}}\ar@{=>}[dd]\ar@{=>}[dr]^{\phi \times 1} & & & & &  F_1^{(3)} \ar@{}[dd]|{\substack{\Tau \times 1 \\ \Rrightarrow}} \ar@{=>}[dl]_{\tau_1^{(3)}}\ar@{=>}[dr]^{\phi \times 1} & \\
                \ar@{}[dr]|{\substack{1\times \Tau \\ \Rrightarrow}} G_1^{(3)}\ar@{=>}[dd]_{1\times \gamma} & \ar@{}[dr]|{\substack{\Phi \\ \cong}} & F_1^{(2)} \ar@{=>}[dd]^{\phi} & & &   G_1^{(3)} \ar@{=>}[dd]_{1\times \gamma} \ar@{=>}[dr]^{\gamma\times 1} & & F_1^{(2)}  \ar@{=>}[dl]_{\tau_1^{(2)}} \ar@{=>}[dd]^{\phi} \\
                & F_1^{(2)} \ar@{=>}[dl]^{\tau_1^{(2)}} \ar@{=>}[dr]_\phi \ar@{}[dd]|{\substack{\Tau \\ \Rrightarrow}} & & \ar@{}[r]|{=} & &   \ar@{}[dr]|{\substack{\Gamma \\ \cong}} &  G_1^{(2)} \ar@{=>}[dd] \ar@{}[dr]|{\substack{\Tau \\ \Rrightarrow}} &  \\
                G^{(2)} \ar@{=>}[dr]_{\gamma} &  & F_1 \ar@{=>}[dl]^{\tau_1} & &  & G_1^{(2)}\ar@{=>}[dr]_{\gamma} & & F_1\ar@{=>}[dl]^{\tau_1} \\
                & G_1 & & & & &  G_1 & \\
            }$$
        \item{}[Unitality] there are equalities of cells
            $$\xymatrix{
                \ar@{}[dr]|{\substack{1\times \Iota\\\Rrightarrow}} F_1 \ar@{=>}[d]_{\tau_1} \ar@{=>}[r]^-{\langle 1,\iota\rangle} &  \ar@{}[dr]|{\substack{\Tau \\\Rrightarrow}} F_1^{(2)} \ar@{=>}[d] \ar@{=>}[r]^\phi & F_1 \ar@{=>}[d]^{\tau_1} \ar@{}[drr]|{=} & & F_1 \ar@{=>}[d]_{\tau_1} \ar@{}[drr]|{=} & & \ar@{}[dr]|{\substack{\Iota\times 1 \\\Rrightarrow}} F_1 \ar@{=>}[d]_{\tau_1} \ar@{=>}[r]^-{\langle \iota,1\rangle} & \ar@{}[dr]|{\substack{\Tau \\\Rrightarrow}} F_1^{(2)} \ar@{=>}[d] \ar@{=>}[r]^\phi & F_1 \ar@{=>}[d]^{\tau_1} \\
                G_1 \ar@{=>}[r]_-{\langle 1,\iota\rangle} & G_1^{(2)} \ar@{=>}[r]_\gamma & G_1 & & G_1 & & G_1 \ar@{=>}[r]_-{\langle \iota,1 \rangle} & G_1^{(2)} \ar@{=>}[r]_\gamma & G_1
            }$$
    \end{enumerate}
    A \textbf{lax double 2-natural transformation} is a lax double pseudo natural transformation in which the transformations $\tau_0$ and $\tau_1$ are both 2-natural.
\end{define}

When it is clear from the context, we refer to lax double pseudo natural transformations of lax double pseudo functors simply as transformations.

\begin{prop} 
    Lax double pseudo functors $\D\to\E$ and their transformations form a category, denoted by $\dbltwocat(\D,\E)$.
\end{prop}
\begin{proof} 
    Composition of transformations $\tau\colon F\Rightarrow G$ and $\sigma \colon G \Rightarrow H$ is defined component-wise. That is, $\sigma\tau\colon F\Rightarrow H$ is given by
        \[ 
            (\sigma\tau)_A := \sigma_A\tau_A \qquad\qquad (\sigma\tau)_m := \sigma_m\tau_m
        \]
    on objects $A$ and proarrows $m$ in $\D$ using the internal composition of arrows and cells in $\E$. Associativity follows since internal compositions are strictly associative. Unit transformations $1\colon F \Rightarrow F$ are given by internal identity arrows and cells in $\E$. 
\end{proof}

\begin{example} \label{ex:for_item_1_2}
    A lax double pseudo natural transformation between lax double pseudo functors on double 2-categories whose 2-categories of objects are trivial is a weakly monoidal pseudo natural transformation in the sense of \cite[\S 2.5]{MV}.
\end{example}

\begin{example}
\label{ex:both_cases_as_indexed}
    Let $\K$ be a 2-category with 2-pullbacks and a terminal object $1$, seen as a Cartesian monoidal 2-category.
    For any double category $\bb{D}$, we denote by $\dbltwocat(\bb{D}, \Span(\K))^{\triangle 1}$ the full category of $\dbltwocat(\bb{D}, \Span(\K))$ whose objects
    $F: \bb{D} \mr{} \Span(\K)$ satisfy that
    $F_0: \bb{D}_0 \mr{} \K$ is the pseudo functor constant at $1$.
        \begin{enumerate}
             \item 
                When $\ct{B}$ is a monoidal category seen as a double category $\bb{B}$ as in Example \ref{ex:monoidascategory} (with 
                $\ct{B}_1=\ct{B}$ and $\ct{B}_0=1$), or equivalently as a locally discrete double 2-category as in Example \ref{ex:for_item_1_1},
                then there is an isomorphism of categories
                    \begin{equation} \label{eq:item1}
                        \dbltwocat(\bb{B}, \Span(\K))^{\triangle 1} \cong \Montwocat_{ps} (\ct{B}, \K),
                    \end{equation}
where the category on the right is that of lax monoidal pseudo functors between monoidal 2-categories and their weakly monoidal pseudo natural transformations as in \cite[\S 2.5]{MV}.
            \item  
                When $\ct{B}$ is a category seen as a {\em discrete} double category $\bb{D}(\ct{B})$ with $\bb{D}(\ct{B})_0 = \bb{D}(\ct{B})_1 = \ct{B}$ and $\src = \tgt = \otimes = y = id_\ct{B}$, we have an isomorphism of categories
                    \begin{equation} \label{eq:item2} \dbltwocat(\bb{D}(\ct{B}), \Span(\K))^{\triangle 1} \cong \twocat_{ps}(\ct{B}, \Mon(\K)), 
                    \end{equation}
where the category on the right is that of  pseudo functors between 2-categories and their  pseudo natural transformations.
        \end{enumerate}  
\end{example}

\begin{proof}
Both results depend ultimately on the fact that for spans in $\K$ whose source and target are $1$, their composition in $\Span(\K)$ is given by the Cartesian product of $\K$. This can be stated by saying that the 2-functor $inc: \K \mr{} \Span(\K)_1$, mapping an object of $\K$ to the unique such span with that apex, is the ``1-part" of a strict double functor $\bb{K} \mr{} \Span(\K)$ (where $\K$ is seen as a double 2-category $\bb{K}$ as in Example \ref{ex:for_item_1_1}). 
Note that $inc$ is an isomorphism of 2-categories with its image (that is, the spans in $\K$ whose source and target are $1$).

To show item 1, we recall first that by Examples \ref{ex:for_item_1_1} and \ref{ex:for_item_1_2}
the category on the right of \eqref{eq:item1} is isomorphic to
$\dbltwocat(\bb{B},\bb{K})$, and then composing with $\bb{K} \mr{} \Span(\K)$ yields the desired isomorphism.

For item 2, we consider first an object of the category on the left of \eqref{eq:item2}.
It includes a pseudo functor $F_1: \ct{B} \mr{} \Span(\K)_1$ that, in view of the above, corresponds to a unique $G: \ct{B} \mr{} \K$ such that $inc \ G = F_1$. Furthermore, the components of the pseudo natural transformations $\phi$ and $\iota$ as in Definition \ref{def:lax_double_pseudo_functor} provide a monoidal structure $\phi_B: GB \times GB \mr{} GB$, $\iota_B: 1 \mr{} GB$ for each $GB$. 
Note that the fact that $\phi$ and $\iota$ are pseudo natural means that $G(f)$ is a (strong) morphism of monoids for each arrow $f$ of $\mscr{B}$, so that $G$ becomes a pseudo functor $\ct{B} \mr{} \Mon(\K)$.
What this is showing is that $inc$ provides a bijection between the objects of the categories as in \eqref{eq:item2}, 
$$\twocat_{ps}(\ct{B}, \Mon(\K)) \mr{} \dbltwocat(\bb{D}(\ct{B}), \Span(\K))^{\triangle 1},$$
mapping $G$ to the lax double pseudo functor $F$ defined to have $F_0$ constant at $1$ and $F_1 = inc \ G$. 
It can be checked that this bijection extends in fact to an isomorphism of categories.
\end{proof}

We will be interested in lax slice categories for various double 2-categories $\bb E$ and certain subcategories. For example, the representation theorem is achieved by a chain of equivalences and isomorphisms involving such a slice, but it is restricted to \textbf{locally discrete} double 2-categories, that is, \emph{bona fide} double categories, as the domain of the lax double pseudo functors. The slice category of contravariant lax double pseudo functors from double categories over a fixed double 2-category $\bb E$ will be denoted by $\mbf{I}\bb E$. This notation is meant to recall the notation $\iset$ or $\icat$ indicating indexed spans and indexed categories respectively. 
\begin{define} \label{define::LaxSliceofLaxDblPseudoFunctors}
    Let $\mbf{I}\bb E$ denote the category with
        \begin{enumerate}
            \item 
                objects the contravariant lax double pseudo functors $F: {\bb A}^{op} \to \bb E$, where $\bb A$ is an {\em indexing} pseudo double category, and
            \item 
                morphisms $F\to G$ those pairs $(H,\tau)$ consisting of a lax double pseudo functor $H$ between the indexing double categories and a lax double pseudo natural transformation $\tau\colon F\Rightarrow GH^{op}$.
        \end{enumerate}
    Refer to $\mbf I\bb E$ as the category of \textbf{indexed proarrows} of $\bb E$.
\end{define}

\begin{example}
    For $\bb E = \Span(\cat)$, the category of indexed spans of categories is $\bf I \Span(\cat)$. Our representation theorem will show that this is equivalent to the category of double fibrations.
\end{example}

\subsection{Lax Functors as Pseudo Monoids}
\label{subsection:LaxFunctorsAsPseudoMonoids}

Here we will give the definition of a pseudo monoid in a double 2-category. These are needed to describe pseudo categories as certain structures in double 2-categories of spans. It is known that  categories are strict monoids in spans in sets. Only the 1-category structure figures in this correspondence, since the composition for ordinary categories involves no further 2-dimensional coherent isos. However, pseudo categories as in Definition \ref{def:pseudocategory} include precisely this higher-dimensional coherence structure in the form of associators and unitors. Pseudo monoids are thus required to describe this structure, and double 2-categories of spans provide the forum in which to study them. That is, double categories, of spans in particular, on their own do not have enough structure to give the coherent iso 2-cells required for the associators and unitors. More precisely, in a double 2-category $\bb D$, there are cells
  $$\xymatrix{
    \ar@{}[dr]|{\alpha} A \ar[d]_f \ar[r]|-@{|}^m & B \ar[d]^g \\
    C\ar[r]|-@{|}_n & D 
  }$$
(as in an ordinary double category) which are morphisms in $\bb D_1$. But $\bb D_1$ also has 2-cells $\Theta\colon\alpha\Rrightarrow \beta$, which are of the form
  $$\xymatrix{
    \ar@{}[dr]|{\gamma} A \ar@{=}[r] \ar[d] & \ar@{}[dr]|{\alpha} A \ar[d] \ar[r]|-@{|} & B \ar[d] \ar@{}[drr]|{\Theta\atop\Rrightarrow}& & \ar@{}[dr]|{\beta} A \ar[d] \ar[r]|-@{|} & \ar@{}[dr]|{\delta} B \ar[d] \ar@{=}[r] & B\ar[d] \\
    C\ar@{=}[r] &C\ar[r]|-@{|} & D  & & C\ar[r]|-@{|} & D \ar@{=}[r] & D 
  }$$
and are absent in an ordinary double category. Invertible 2-cells in $\bb D_1$ will provide the associators and unitors for our notion of pseudo monoid. Here is the precise definition.

\begin{define} \label{define:PseudoMonoidInDouble2Cat}
  Let $\bb D$ denote a double 2-category. A \textbf{pseudo monoid} in $\bb D$ is an endo-proarrow $m\colon A\hto A$ equipped with multiplication and unit cells
    $$\xymatrix{
      \ar@{}[drr]|{\mu} A \ar@{=}[d] \ar[r]|-@{|}^m & A \ar[r]|-@{|}^m & A \ar@{=}[d] & & & A\ar@{}[dr]|{\iota} \ar@{=}[d] \ar[r]|-@{|}^y & A \ar@{=}[d] \\ 
      A \ar[rr]|-@{|}_m & & A & & & A \ar[r]|-@{|}_m & A 
    }$$
  and globular iso 3-cells
    \begin{enumerate}
      \item for associativity:
        $$\xymatrix{
          \ar@{}[drr]|{\mu} \cdot \ar@{=}[d] \ar[r]|-@{|}^m & \cdot \ar[r]|-@{|}^m & \ar@{}[dr]|{1} \cdot \ar@{=}[d] \ar[r]|-@{|}^m &  \cdot \ar@{=}[d] & & \ar@{}[dr]|{1}  \cdot \ar@{=}[d] \ar[r]|-@{|}^m &\ar@{}[drr]|{\mu} \cdot\ar@{=}[d] \ar[r]|-@{|}^m &  \cdot  \ar[r]|-@{|}^m &  \cdot \ar@{=}[d]  \\
          \ar@{}[drrr]|{\mu} \cdot \ar@{=}[d] \ar[rr]|-@{|}_m & & \cdot \ar[r]|-@{|}_m & \cdot \ar@{=}[d] &\mrm{A}\atop \cong & \ar@{}[drrr]|{\mu} \cdot \ar@{=}[d] \ar[r]|-@{|}_m &   \cdot \ar[rr]|-@{|}_m & & \cdot \ar@{=}[d] \\
          \cdot \ar[rrr]|-@{|}_m & & & \cdot & &   \cdot \ar[rrr]|-@{|}_m & & & \cdot
        }$$
      \item and for left and right units:
        $$\xymatrix{
          \ar@{}[dr]|{\iota} \cdot \ar@{=}[d] \ar[r]|-@{|}^y & \ar@{}[dr]|{1} \cdot \ar@{=}[d] \ar[r]|-@{|}^m &\cdot \ar@{=}[d] & & \ar@{}[ddr]|{1}\cdot \ar@{=}[dd] \ar[r]|-@{|}^m & \cdot\ar@{=}[dd] && \ar@{}[dr]|{1} \cdot \ar@{=}[d] \ar[r]|-@{|}^m & \ar@{}[dr]|{\iota}\cdot \ar@{=}[d]  \ar[r]|-@{|}^y &\cdot \ar@{=}[d] \\
          \ar@{}[drr]|{\mu}\cdot \ar@{=}[d] \ar[r]|-@{|}^m &\cdot \ar[r]|-@{|}^m &\cdot \ar@{=}[d] &\Lambda\atop\cong &&& \mrm{P}\atop \cong &\ar@{}[drr]|{\mu}\cdot \ar@{=}[d] \ar[r]|-@{|}^m & \cdot \ar[r]|-@{|}^m &\cdot \ar@{=}[d] \\
          \cdot \ar[rr]|-@{|}_m & &\cdot && \cdot \ar[r]|-@{|}_m &\cdot& &\cdot  \ar[r]|-@{|}_m & \cdot \ar[r]|-@{|}_m &\cdot  
        }$$
    \end{enumerate}
  satisfying the axioms:
    \begin{enumerate}
      \item pentagonal identity: $[\mu(1\otimes \mrm A)][\mrm A(1\otimes \mu\otimes 1)][\mu(\mrm A\otimes 1)] = [\mrm A (1\otimes 1\otimes \mu)][\mrm A(\mu\otimes 1\otimes 1)]$ holds;
      \item left and right unit compatibility: $\Lambda[\mu(1\otimes \mrm P)][\Lambda(\iota\otimes 1\otimes 1)] = \mrm P[\mu(\Lambda\otimes 1)]$ holds.
    \end{enumerate}
A pseudo monoid is \textbf{normalized} if $\Lambda$ and $\mrm P$ are identity cells. A normalized pseudo monoid is \textbf{strict} if $\mrm A$ is an identity cell too.
\end{define}

Pseudo monoids in a double 2-category generalize many other structures:

\begin{example} 
  Pseudo monoids in a monoidal 2-category \cite[\S 2.5]{MV} are the same as pseudo monoids as defined above in that same monoidal 2-category viewed as a double 2-category whose 2-category of objects is trivial.
\end{example}

\begin{example} \label{ex:pseudo_categories_as_monoids_in_spans}
  A pseudo monoid in $\Span(\mbf{Mon})$ is a (weak) monoidal category. A pseudo monoid in $\Span(\cat)$ is a pseudo double category. More generally, a pseudo monoid in $\Span(\K)$ is a pseudo category in $\K$ according to Definition \ref{def:pseudocategory}. 
\end{example}

Note that in this correspondence between pseudo monoids and pseudo categories the source and target of the pseudo category are given by the proarrow $m\colon A\hto A$ (the span of arrows of $\K$, i.e. the pro-arrow of $\Span(\K)$). 

\begin{example} \label{example:PsMonsInSpansReltoSigma}
    More generally, recall that if $\K$ is equipped with a distinguished family of arrows $\Sigma_1$ the double 2-category $\Span_\Sigma(\K)$ arises in Construction \ref{construction:SpanCatAsDouble2Category} by requiring the proarrows to be in $\Sigma_1$. Hence a pseudo monoid in $\Span_\Sigma(\K)$ is precisely a pseudo category in $\K$ such that its source and target are in $\Sigma_1$. So, in particular, double fibrations are equivalently pseudo monoids in $\Span_c(\fib)$ from Definition \ref{define:Double2CategoryofSpansInFibrations}.
\end{example}

Another class of examples is provided by taking pseudo categories in various arrow (2-)categories. These are analogues of the case where an internal functor between internal categories is equivalently described as an internal category in a category of arrows. The difference is that in taking ``arrow 2-categories of 2-categories" the formerly commuting squares can now be filled with identities, isos, or mere cells. These correspond, roughly speaking, to strict, pseudo  and lax-internal functors.

\begin{example} \label{ex:before_PE}
    For any 2-category $\K$, we consider the 2-category $\Arrl(\K)$ and 
    $\Sigma_1$ the class of morphisms in $\Arrs(\K)$, defining thus as in Construction \ref{construction:SpanCatAsDouble2Category} a double 2-category that we denote $\Span_s(\Arrl(\K))$. By Example \ref{ex:pseudo_categories_as_monoids_in_spans}, we get that a pseudo monoid in $\Span_s(\Arrl(\K))$ is then a pseudo category in $\Arrl(\K)$ whose source and target are in $\Arrs(\K)$. But by Lemma \ref{lem:laxfunctorascategoryobject}, such a  pseudo category is a lax functor between internal pseudo categories. As mentioned above, this example can be restricted to pseudo and strict internal functors by considering respectively $\Arrp$ and $\Arrs$ in place of $\Arrl$.

    If one writes explicitly the correspondence between pseudo monoids in $\Span_s(\Arrl(\K))$ and lax functors between arbitrary internal pseudo categories, it will be seen that, for any internal pseudo category $\bb E$ in $\K$, one can {\em restrict it} to those lax functors that have $\bb E$ as their codomain (of course, the same can be done for the domain instead). That is, one can take a full sub-double 2-category $\bb{P}_s(\bb E)$ of $\Span_s(\Arrl(\K))$ corresponding via the correspondence to the lax functors $\bb X \to \bb E$.
\end{example}

\begin{remark} \label{rem:PE_high_level}
The example above illustrates how internal lax functors can be seen as categories internal to certain arrow 2-categories. Equivalently, they are thus pseudo monoids in certain double 2-categories of spans. Inasmuch as our goal is to realize double fibrations as certain lax functors into $\Span(\cat)$ and this is achieved by ``taking pseudo categories," we need an ``arrow category-like structure" in which to describe these lax functors as pseudo categories, or as pseudo monoids in an appropriate double 2-category of spans. One might then expect to be able to consider a sub-double 2-category $\bb P_s(\bb E)$ of $\Span_s(\Arrl(\K)))$ as above, for some judicious choice of $\K$. However, recalling that lax double pseudo functors are not internal lax functors in $\twocat$, 
the issue is that generally there is no \emph{strict} 3-category on 2-categories with enough ``pseudoness" to accommodate our constructions. 
Rather than work with tricategories, which would extend the scope of this paper, we opt for a construction that, though slightly \emph{ad hoc}, serves our purpose.

The rest of the subsection is meant to provide the required structure in the form of a double 2-category $\bb P(\bb E)$ where $\bb E$ is a double 2-category given and fixed throughout. The special case for the representation theorem will be $\bb E = \Span(\cat)$. 
In essence, $\bb P(\bb E)$ is somewhat like a sub-double 2-category of $\Span_s(\Arrl(\K)))$ where $\K = \twocat$, and thinking about it in this way may help the reader understand Construction \ref{construction:ConstructionOfSliceOfSpan2Cat} and Proposition \ref{prop:SpanValuedLaxFunctorsArePseudoMonoids}, but the difference is that pseudo functors and pseudo natural transformations are allowed at certain levels of structure. The fixed double 2-category $\bb E$ appears as the receiving structure for our lax functors because the external structure of $\bb E$ is built into the definition of composition and identities of $\bb P(\bb E)$.
\end{remark}

\begin{construction} \label{construction:ConstructionOfSliceOfSpan2Cat}
    Fix a double 2-category $\bb E$. Let $\src,\tgt\colon \bb E_1 \rightrightarrows \bb E_0$ denote its source-target structure. Here we define the structure of a double 2-category $\bb P(\E)$ built upon this span. The 2-category $\bb P(\bb E)_0$ has
        \begin{enumerate}
            \item objects: pseudo functors $F\colon \mscr A^{op} \to \bb E_0$ where $\mscr A$ is a category;
            \item arrows: pseudo natural transformations
                $$\xymatrix{
                    \ar@{}[dr]|{\Downarrow\alpha} \mscr A^{op} \ar[d]_{H^{op}} \ar[r]^F & \bb E_0 \ar[d] \\
                    \mscr B^{op} \ar[r]_G &\bb E_0
                }$$
            \item 2-cells: pairs consisting of a pseudo natural transformation $\theta\colon H\Rightarrow K$ and a modification
                $$\xymatrix{
                    \ar@{}[dr]|{\Downarrow\alpha} \mscr A^{op} \ar[d]_{H^{op}} \ar[r]^F & \bb E_0 \ar[d] &\ar@{}[dr]|{\Theta\atop\Rrightarrow} & & \ar@{}[dr]|{\theta^{op}\Downarrow}\mscr A^{op} \ar[d]_{H^{op}} \ar@{=}[r] & \ar@{}[dr]|{\Downarrow\beta} \mscr A^{op} \ar[d] \ar[r]^F & \bb E_0 \ar[d] \\
                     \mscr B^{op} \ar[r]_G &\bb E_0 & & & \mscr B^{op} \ar@{=}[r] & \mscr B^{op} \ar[r]_G &\bb E_0
                }$$
            Note that these are $\bb E_0$-analogues of the ``indexed 2-cells" displayed in \cite[\S 2, Eqn (9)]{MV}.
        \end{enumerate}
    So defined, $\bb P(\bb E)_0$ is a 2-category since $\bb E_0$ is one. In the case that $\bb E_0$ is $\cat$, $\bb P(\bb E)_0$ is precisely $\icat$. 
    
    The idea for the proarrow part $\bb P(\bb E)_1$ is that it should be a 2-category of \emph{certain} spans over the source and target span coming with $\bb E$. These spans are the ones needed for forming monoids. In detail, take $\bb P(\bb E)_1$ to have as objects span morphisms over $(\src, \tgt)$ of the form 
      $$\xymatrix{
        \mscr A^{op} \ar[d]_F & \mscr B^{op} \ar[d]^G \ar[l]_{S^{op}} \ar[r]^{T^{op}} & \mscr C^{op} \ar[d]^H \\
        \bb E_0 & \ar[l]^-\src \bb E_1 \ar[r]_-\tgt & \bb E_0 
      }$$
    where $F$, $G$ and $H$ are allowed to be pseudo but the legs of the top span are strict 2-functors. Note that we ask both squares to commute strictly. A morphism in $\bb P(\bb E)_1$ is displayed as
        $$\xymatrix{
          \bar{\mscr A}^{op} \ar[d]_{\bar F} &  \bar{\mscr B}^{op} \ar[d]^{\bar G} \ar[l]_{\bar S^{op}} \ar[r]^{\bar T^{op}} & \bar{\mscr C}^{op} \ar[d]^{\bar H} &\ar@{}[dr]|{\substack{\langle \alpha,\beta,\gamma\rangle \\ \longrightarrow}} & &  \mscr A^{op} \ar[d]_F & \mscr B^{op} \ar[d]^{G} \ar[l]_{S^{op}} \ar[r]^{T^{op}} & \mscr C^{op} \ar[d]^H \\
          \bb E_0 & \ar[l]^-\src \bb E_1 \ar[r]_-\tgt & \bb E_0 & & &  \bb E_0 & \ar[l]^-\src \bb E_1 \ar[r]_-\tgt & \bb E_0,
        }$$
    and consists of three pseudo natural transformations between pseudo functors 
        $$\xymatrix{
          \ar@{}[dr]|{\;\;\;\substack{\alpha\\\Rightarrow}}\bar{\mscr A}^{op} \ar@/_1pc/[dr]_{\bar F} \ar[r]^{P^{op}} & \mscr A^{op} \ar[d]^F & & \ar@{}[dr]|{\;\;\;\substack{\beta\\\Rightarrow}}\bar{\mscr B}^{op} \ar@/_1pc/[dr]_{\bar G} \ar[r]^{Q^{op}} & \mscr B^{op} \ar[d]^G & & \ar@{}[dr]|{\;\;\;\substack{\gamma\\\Rightarrow}}\bar{\mscr C}^{op} \ar@/_1pc/[dr]_{\bar H} \ar[r]^{R^{op}} & \mscr C^{op} \ar[d]^H \\
          & \bb E_0 & & & \bb E_1 & & & \bb E_0
        }$$
    satisfying the conditions that
        \begin{equation} \label{equation:PrismEquationsSpanCat}
          \src\ast\beta = \alpha\ast S'\qquad \text{and} \qquad \tgt \ast \beta = \gamma \ast T'
        \end{equation}  
    both hold. So, the transformations $\alpha$, $\beta$ and $\gamma$ together with the faces of the span morphisms form two prisms with a shared triangular face $\beta$. This gives a good idea of what the 2-cells should be. That is, a 2-cell consists of three modifications $\Sigma\colon \alpha \Rightarrow \alpha'$, $\Xi\colon \beta\Rrightarrow\beta'$ and $\mrm T\colon \gamma\Rrightarrow\gamma'$ satisfying the same conditions, namely, that $\Sigma \ast S' = \src\ast\Xi$ and $\mrm T\ast T' = \tgt\ast\Xi$ both hold. The fact that $\bb P(\bb E)_1$ is a 2-category follows from interchange for whiskering of 2-cells both in 2-categories of pseudo functors and in the 2-categories $\bb E_0$ and $\bb E_1$.
    
    The 2-functors $\src,\tgt \colon \bb P(\bb E)_1 \rightrightarrows \bb P(\bb E)_0$ for source and target are now the natural ones. That is, all the data of $\bb P_1$ is summarized in a single 2-cell $\langle \Sigma, \Xi,\mrm T\rangle$ as discussed above. So, the source of such a cell is $\Sigma$ and the target is $\mrm T$. The assignments on arrows and objects are then the natural ones. For external composition, take a composable span
        $$\xymatrix{
          \mscr C^{op} \ar[d]_H & \mscr D^{op} \ar[d]^K \ar[l] \ar[r] & \mscr E^{op} \ar[d]^L \\
         \bb E_0 & \ar[l]^-\src \bb E_1 \ar[r]_-\tgt & \bb E_0.
        }$$
    The composite span is defined to be the span morphism
        $$\xymatrix{
          \mscr A^{op} \ar[d]_F && (\mscr B\times_{\mscr C}\mscr D)^{op} \ar[d]^{G\times_HK} \ar[ll] \ar[rr] && \mscr E^{op} \ar[d]^H \\
          \bb E_0 \ar@{=}[d] && \ar[ll] \bb E_1\times_{\bb E_0} \bb E_1 \ar[d]^\otimes \ar[rr] && \bb E_0 \ar@{=}[d] \\ 
          \bb E_0 && \ar[ll]^-\src \bb E_1 \ar[rr]_-\tgt && \bb E_0 
        }$$
    taking 2-pullbacks and using the given external composition $\otimes$ coming with $\bb E$. Note that $G\times_HK$ is well-defined since the span morphisms giving the objects of $\bb P(\bb E)_1$ are required to commute strictly. This extends to an honest 2-functor
        \[
          \otimes\colon \bb P(\bb E)_1\times_{\bb P(\bb E)_0} \bb P(\bb E)_1 \longrightarrow \bb P(\bb E)_1
        \]
    giving the external composition. There is some work behind this statement, but it is straightforward. The assignments on morphisms and cells are those suggested by object assignment above. Well-definition in each case depends on the prism equations. That $\otimes$ is a 2-functor is a result of componentwise composition in $\bb E_1\times_{\bb E_0}\bb E_1$ and the fact that $\otimes$ for $\bb E$ is a 2-functor. Finally, the external identity on a pseudo functor $F\colon \mscr A^{op}\to \bb E_0$ is the span
        $$\xymatrix{
          \mscr A^{op} \ar[d]_F & \mscr A^{op} \ar[d]^{\Delta_F} \ar[l] \ar[r] & \mscr A^{op} \ar[d]^F \\
          \bb E_0 & \ar[l]^-\src \bb E_1 \ar[r]_-\tgt & \bb E_0.
        }$$
    where $\Delta_F$ takes $A\in\mscr A$ to the external identity on $FA$. For any transformation $\alpha\colon F'\Rightarrow F$ or modification $\Theta\colon \alpha'\Rrightarrow \alpha$, there are corresponding cells and modifications between $\Delta_{F'}$ and $\Delta_F$ giving the required structures for an external identity 2-functor $y\colon \bb P(\bb E)_0\to \bb P(\bb E)_1$.
\end{construction}

\begin{lemma}
    The structure $\src,\tgt\colon \bb P(\bb E)_1 \rightrightarrows \bb P(\bb E)_0$ with $\otimes$ and $y$ as above defines a double 2-category denoted by $\bb P(\bb E)$.
\end{lemma}
\begin{proof}
    Up-to-iso associativity for $\bb P(\bb E)$ holds by the fact that this composition morphism is defined using $\otimes$ for $\bb E$ and 2-pullbacks of 2-functors which are both already known to be associative up to iso. The unit laws hold by normalization.
\end{proof}

\begin{prop} \label{prop:SpanValuedLaxFunctorsArePseudoMonoids}
    Pseudo monoids in $\bb P(\bb E)$ are precisely contravariant lax double pseudo functors on double categories that are valued in $\bb E$.
\end{prop}
\begin{proof}
     A pseudo monoid in $\bb P(\bb E)$ starts with an underlying endo-proarrow, that is, a span
        $$\xymatrix{
          \mscr A^{op} \ar[d]_{F_0} & \mscr M^{op} \ar[d]^{F_1} \ar[l]_{S^{op}} \ar[r]^{T^{op}} & \mscr A^{op} \ar[d]^{F_0} \\
          \bb E_0 & \ar[l]^-\src \bb E_1 \ar[r]_-\tgt & \bb E_0 
        }$$
    The pseudo functors $F_0$ and $F_1$ are the $0$- and $1$- parts of a purported lax double pseudo functor. These come equipped with multiplication and unit cells, which are in fact 2-cells between span morphisms. The multiplication cell is 
        $$\xymatrix{
          \mscr A^{op} \ar[d]_{F_0} & (\mscr M\times_{\mscr A}\mscr M)^{op} \ar[l] \ar[r] \ar[d]^{F_1\times_{F_0}F_1} & \mscr A^{op} \ar[d]^{F_0} & & & \mscr A^{op} \ar@{=}[d] & (\mscr M\times_{\mscr A}\mscr M)^{op} \ar[l] \ar[r] \ar[d]^{\otimes^{op}} & \mscr A^{op} \ar@{=}[d] \\
          \bb E_0 \ar@{=}[d] & \ar[l] \bb E_1\times_{\bb E_0} \bb E_1 \ar[d]^\otimes \ar[r] & \bb E_0 \ar@{=}[d] & \ar@{}[r]|{\substack{\mu\\\Rightarrow}} & &\mscr A^{op}\ar[d]_{F_0} &\mscr M^{op} \ar[l]\ar[r]\ar[d]^{F_1} & \mscr A^{op}\ar[d]^{F_0} \\
          \bb E_0 & \ar[l]^-\src \bb E_1 \ar[r]_-\tgt & \bb E_0  & & & \bb E_0 & \ar[l]^-\src \bb E_1 \ar[r]_-\tgt & \bb E_0 .
        }$$
    which is globular by construction. Notice that the `$\otimes$' on the right is given as part of the pseudo monoid data. It will give the external composition of the domain double-2-category. Likewise the unit cell is one
        $$\xymatrix{
          \mscr A^{op} \ar[d]_{F_0} & \mscr A^{op} \ar@{=}[l] \ar@{=}[r] \ar[d]^{F_0} & \mscr A^{op} \ar[d]^{F_0} & & & \mscr A^{op} \ar@{=}[d] & \mscr A^{op} \ar@{=}[l] \ar@{=}[r] \ar[d]^{\Delta_\mscr A} & \mscr A^{op} \ar@{=}[d] \\
          \cat\ar@{=}[d] &\cat\ar@{=}[l]\ar@{=}[r]\ar[d]^{\Delta_{\cat}} &\cat\ar@{=}[d] &\ar@{}[r]|{\substack{\iota\\\Rightarrow}} & & \mscr A \ar[d]_{F_0} & \mscr M \ar[l]^S \ar[r]_T\ar[d]^{F_1} &\mscr A \ar[d]^{F_0}\\
          \bb E_0 & \ar[l]^-\src \bb E_1 \ar[r]_-\tgt & \bb E_0  & & & \bb E_0 & \ar[l]^-\src \bb E_1 \ar[r]_-\tgt & \bb E_0 .
        }$$
    Again $\Delta_\mscr A$ on the right is given from the structure of a pseudo monoid. Since these cells are externally globular, the important data is in each case the 2-cell between the apex morphisms. These are
        $$\xymatrix{
          \ar@{}[dr]|{\substack{\mu\\\Rightarrow}}(\mscr M\times_{\mscr A}\mscr M)^{op} \ar[d]_{F_1\times_{F_0}F_1} \ar[r]^-{\otimes^{op}} & \mscr M^{op} \ar[d]^{F_1} & & & \ar@{}[dr]|{\substack{\iota\\\Rightarrow}}\mscr A^{op}\ar[d]_{F_0} \ar[r]^{\Delta_\mscr A} & \mscr M^{op} \ar[d]^{F_1} \\
          \bb E_1\times_{\bb E_0}\bb E_1 \ar[r]_-\otimes & \bb E_1 & & & \bb E_0 \ar[r]_-{y} & \bb E_1
        }$$
    which is precisely the remaining data for a lax functor on a double-2-category valued in $\bb E$. The associator and unitor isos can similarly be extracted from the data for a pseudo monoid. The axioms in \ref{define:PseudoMonoidInDouble2Cat} are precisely the associativity and unit axioms for a lax functor. Denote this by $F\colon \bb A^{op}\to\bb E$ using $\bb A$ to correspond to the given category $\mscr A$. This process of extracting a lax functor from a pseudo monoid is completely reversible. Since it is just a matter of rearranging the data, the correspondence amounts to a bijection, as claimed.
\end{proof}

We now define homomorphisms of pseudo monoids. An example to have in mind is the case of strict monoids in an ordinary double category. When that double category is spans in sets, so that its monoids are categories, the homomorphisms correspond to functors. 
The following is the version appropriate for our pseudo monoids, corresponding to lax functors when the double 2-category is $\Span(\K)$, as indicated in the example below.

\begin{define} \label{define:LaxHomomsOfPseudoMonoids}
  A \textbf{lax homomorphism} of pseudo monoids $(m,\mu,\iota)$ and $(n,\nu,\upsilon)$ in a double 2-category $\bb D$ is a cell
    $$\xymatrix{
      \ar@{}[dr]|{\phi} A \ar[d]_f \ar[r]|-@{|}^m & A \ar[d]^f \\
      B \ar[r]|-@{|}_n & B
    }$$ 
  together with comparison 3-cells $\Phi$ and $\Upsilon$ as in
    $$\xymatrix{
      \ar@{}[dr]|{\phi} \cdot \ar[d]_f \ar[r]|-@{|}^m & \ar@{}[dr]|{\phi} \cdot \ar[d]_f \ar[r]|-@{|}^m &\cdot\ar[d]^f & & \ar@{}[drr]|{\mu} \cdot \ar@{=}[d] \ar[r]|-@{|}^m & \cdot \ar[r]|-@{|}^m & \cdot\ar@{=}[d] & & \cdot \ar@{}[dr]|{1_f} \ar[d]_f \ar[r]|-@{|}^y & \cdot \ar[d]^f & &\ar@{}[dr]|{\iota} \cdot \ar@{=}[d] \ar[r]|-@{|}^y &\cdot \ar@{=}[d]\\
      \ar@{}[drr]|{\nu}\cdot \ar@{=}[d] \ar[r]|-@{|} & \cdot \ar[r]|-@{|}& \cdot \ar@{=}[d] &\Phi\atop\Rrightarrow & \ar@{}[drr]|{\phi} \cdot \ar[d]_f \ar[rr]|-@{|} & & \cdot\ar[d]^f & & \ar@{}[dr]|{\upsilon}\cdot \ar@{=}[d] \ar[r]|-@{|}^y &\cdot \ar@{=}[d] &\Upsilon\atop\Rrightarrow &\ar@{}[dr]|{\phi} \cdot\ar[d]_f \ar[r]|-@{|}^m & \cdot \ar[d]^f  \\
      \cdot \ar[rr]|-@{|}_n & & \cdot  & &  \cdot \ar[rr]|-@{|}_n & & \cdot & & \cdot \ar[r]|-@{|}_n & \cdot & & \cdot \ar[r]|-@{|}_n & \cdot 
    }$$
  satisfying:
    \begin{enumerate}
        \item 
            The external source and target of $\Phi$ and $\Upsilon$ are identities;
        \item 
            associativity: $\mrm A[\Phi(1\otimes\mu)][\nu(1\otimes \Phi)] = [\Phi(\mu\otimes 1)][\nu(\Phi\otimes 1)]\mrm A$ holds;
        \item 
            unit: $\Lambda[\Phi(\iota\otimes 1)][\nu(\Upsilon\otimes 1)]\Lambda = 1$ and $\mrm P[\nu(1\otimes \Upsilon)][\Phi(1\otimes\iota)]\mrm P = 1$ both hold.
    \end{enumerate}
  Such a lax homomorphism is just a \textbf{homomorphism} if $\Phi$ and $\Upsilon$ are both invertible; it is a \textbf{strict homorphism} if they are both identities. A lax homormophism is \textbf{unital} if $\Upsilon$ is an identity. Let $\psmon(\bb D)$ denote the category of normalized pseudo monoids in $\bb D$ and their homomorphisms.
\end{define}

\begin{example}\label{example:correspondenceHomomsAndTransfFillingArbSquares}
    When $\bb D = \Span(\K)$ we have shown in Example \ref{ex:pseudo_categories_as_monoids_in_spans} that a pseudo monoid in $\bb D$ is an  internal pseudo category in $\K$. A lax homomorphism between pseudo monoids in this case is precisely a lax functor between the internal pseudo categories. When $\bb D = \Span_s(\Arrl(\K))$  we have shown in Example \ref{ex:before_PE} how a pseudo monoid in $\bb D$ is then an internal pseudo category in $\Arrl(\K)$ with strict source and target, corresponding thus to a lax functor between internal pseudo categories in $\K$. 
    A lax homomorphism between pseudo monoids in this case is then a lax functor between the internal pseudo categories, corresponding by Lemma \ref{lem:transformationaslaxfunctor} to a square of lax functors filled with a transformation.
    As before, this restricts to a correspondence between homomorphisms (resp. strict homomorphisms) morphisms between the pseudo monoids and squares whose top and bottom are pseudo (resp. strict) functors.
    
    We have also made explicit in Lemma \ref{lem:transformationaslaxfunctor} the case of pseudo categories in $\Arrs(\K)$ instead of $\Arrl(\K)$, which is the relevant one for morphisms of double fibrations: when $\bb D = \Span(\Arrs(\K))$, pseudo monoids in $\bb D$ amount to strict functors between two internal pseudo categories, and a homomorphism between two  pseudo monoids amounts to a commutative square which has the strict functors written vertically, whose top and bottom are internal pseudo functors in $\K$.
\end{example}

By Example \ref{example:PsMonsInSpansReltoSigma}, we know that double fibrations are pseudo monoids in the double 2-category $\Span_c(\fib)$. It now makes sense to define a category of double fibrations. The morphisms are homomorphisms of pseudo monoids as above in Definition \ref{define:LaxHomomsOfPseudoMonoids}.

\begin{define} \label{define:CategoryOfDoubleFibrations}
    The category of double fibrations $\dblfib$ is defined as 
        \begin{equation}
            \dblfib:=\psmon(\Span_c(\fib))
        \end{equation}
    that is, as the category of pseudo monoids and their homomorphisms in $\Span_c(\fib)$.
\end{define}

\begin{remark} \label{rem:describe_morph_of_dblfib}
Recall that there is a forgetful 2-functor $\fib \mr{} \Arrs(\cat)$, that forgets the chosen cleavages, and induces a forgetful strict double 2-functor $\Span_c(\fib) \mr{} \Span(\Arrs(\cat))$ (that is, internal to $\twocat$). When applied to a double fibration, that is an object of $\psmon(\Span_c(\fib))$, we have then an object of $\psmon(\Span(\Arrs(\cat)))$, that is a strict double functor $P: {\bb E} \mr{} {\bb B}$ between pseudo double categories.
Note that a strict double functor $P: {\bb E} \mr{} {\bb B}$ defines a double fibration (as in Definition \ref{def:unwinded_double_fibration}) precisely when it is in the image of this forgetful functor. 

Now, when  $\Span_c(\fib) \mr{} \Span(\Arrs(\cat))$ is applied to a morphism of double fibrations, we have then an arrow in $\psmon(\Span(\Arrs(\cat)))$, that is a homomorphism between the pseudo monoids, corresponding by Example \ref{example:correspondenceHomomsAndTransfFillingArbSquares} to a commutative square
$$
\xymatrix{
{\bb E} \ar[r]^H \ar[d]_P & {\bb E'} \ar[d]^{P'} \\
{\bb B} \ar[r]_K & {\bb B'} }
$$
where $H$ and $K$ are pseudo double functors. By definition of the arrows of $\fib$, we get that a morphism of double fibrations is then given by a pair of pseudo double functors $H$ and $K$, making the square commutative and such that the functors $H_0: \ct{E}_0 \mr{} \ct{E}'_0$ and $H_1: \ct{E}_1 \mr{} \ct{E}'_1$ are Cartesian-arrow preserving.
\end{remark}

To close the subsection, we note that the correspondence between homomorphisms and transformations filling arbitrary squares described in Example \ref{example:correspondenceHomomsAndTransfFillingArbSquares} can be  restricted, when one considers the double 2-category ${\bb P}_s(\bb E)$ appearing in Remark \ref{rem:PE_high_level}, to transformations filling triangles with ${\bb E}$ as a vertex. The following is the {\em non-strict} (that is, non-internal to $\twocat$) version of that result. This result will be used in the proof of the representation theorem where it will be applied to $\bb E= \Span(\cat)$.

\begin{prop} \label{prop:HomomsOfPsMonsAreDblePsNatTransfs}
    Homomorphisms of pseudo monoids in $\bb P(\bb E)$ are precisely the lax double pseudo natural transformations of the form
        $$\xymatrix{
            \ar@{}[dr]|{\Downarrow}\bb A^{op} \ar[r]^M \ar[d]_{F^{op}} & \bb E \ar@{=}[d] \\
            \bb B^{op} \ar[r]_N & \bb E
        }$$
    This correspondence is functorial, meaning that there is an isomorphism of categories
        \begin{equation}
            \psmon(\bb P(\bb E))\cong \mbf{I}\bb E    
        \end{equation}
    for any double 2-category $\bb E$. (Recall that the righthand side is the lax slice of lax double pseudo functors from pseudo double categories as in Definition \ref{define::LaxSliceofLaxDblPseudoFunctors}.)
\end{prop}
\begin{proof}
    As in the proof of Proposition \ref{prop:SpanValuedLaxFunctorsArePseudoMonoids}, this is mostly a matter of unraveling the data of such a homomorphism of pseudo monoids and noticing that the axioms are precisely the required coherence conditions for a transformation of lax double pseudo functors.
\end{proof}

\subsection{Proof of the Representation Theorem}
\label{subsection:Translation}

This subsection is dedicated to giving a proof of one of the two main results of the paper, namely, that double fibrations correspond to lax double pseudo functors valued in $\Span(\cat)$. One could certainly start with the equivalence $\fib\simeq\icat$ and simply show by brute force that pseudo categories in $\icat$ are lax double pseudo functors valued in $\Span(\cat)$. However, the details involved in such a proof are considerable. In particular, in each direction, one has to induce the associators and unitors, show that they satisfy the required conditions, and finally show the required equivalence of categories. Clearly this ends up being quite involved. As an alternative, here we will use the construction $\bb P(\bb E)$ from the last subsection for a certain choice of $\bb E$ along with Proposition \ref{prop:SpanValuedLaxFunctorsArePseudoMonoids} to circumvent these details. In particular, we specialize Construction \ref{construction:ConstructionOfSliceOfSpan2Cat} to the case where $\bb E$ is the double 2-category $\Span(\cat)$ from Example \ref{construction:SpanCatAsDouble2Category}. We will give an isomorphism of double 2-categories
      \begin{equation} \label{eqn:AnnounceApxIsAnIsoDiscussion}
            \apx\colon \bb P(\Span(\cat)) \cong\Span_t(\icat) 
      \end{equation}
induced by a double 2-functor denoted by $\apx$. This will be given by pushing forward by the functor taking a span in $\icat$ to its apex. The target double 2-category $\Span_t(\icat)$ is formed according to the conventions of Definition \ref{define:SpansWRTaSubsetofArrows}. In particular the spans from $\icat$ forming the proarrows are given by 2-natural transformations and not pseudo natural transformations. The purpose of this result is to show that pseudo monoids in $\Span_t(\icat)$ are lax double pseudo functors valued in $\Span(\cat)$ using Proposition \ref{prop:SpanValuedLaxFunctorsArePseudoMonoids} above. Again the point is that it is easier to prove this isomorphism than prove directly that a pseudo category in $\Span_t(\icat)$ is a lax double pseudo functor.

\begin{construction}[Apex Functor]
  Here we construct the $\apx$ functor announced in Equation \eqref{eqn:AnnounceApxIsAnIsoDiscussion}. Since the 0-part of each double 2-category above is $\icat$, it suffices to take $\apx_0$ to be the identity and then define a 2-functor 
    \begin{equation} \label{equation:Apx2Functor}
      \apx_1\colon \bb P(\Span(\cat))_1 \to\Span_t(\icat)_1. 
    \end{equation}
  On objects, take a span morphism as at left to the cell composite as at right:
    $$\xymatrix{
      \mscr A^{op} \ar[d]_F & \mscr B^{op} \ar[d]^G \ar[l] \ar[r] & \mscr C^{op} \ar[d]^H &\ar@{}[ddr]|{\substack{\apx\\\mapsto}} & & \mscr A^{op} \ar[d]_F & \mscr B^{op} \ar[d]^G \ar[l] \ar[r] & \mscr C^{op} \ar[d]^H \\
      \cat & \ar[l]^-\src \Span(\cat)_1 \ar[r]_-\tgt & \cat & & & \ar@{}[dr]|{\substack{\sigma\\\Leftarrow}} \cat \ar@{=}[d] & \ar[l]^-\src \ar@{}[dr]|{\substack{\tau\\\Rightarrow}} \Span(\cat)_1 \ar[d]^\apx \ar[r]_-\tgt & \cat \ar@{=}[d]  \\
      & &  & & & \cat & \ar@{=}[l] \cat \ar@{=}[r] & \cat.
    }$$
  That is, the assignment is achieved by pushing forward the span on the left by the apex 2-functor and the accompanying cells $\sigma$ and $\tau$. The resulting span of cells on the right can be depicted as 
    $$\xymatrix{
      \ar@{}[drr]|{\substack{\sigma\ast G\\\Leftarrow}}\mscr A^{op} \ar[d]_F && \ar@{}[drr]|{\substack{\tau\ast G\\\Rightarrow}} \mscr B^{op} \ar[d]^{\apx G} \ar[ll] \ar[rr] && \mscr C^{op} \ar[d]^H \\
      \cat && \ar@{=}[ll] \cat \ar@{=}[rr] && \cat
    }$$
  using whiskering of 2-cells, denoted by `$\ast$'. Note that the assignment is well-defined since $\sigma$ and $\tau$ are both 2-natural and whiskering with $G$ results in a 2-natural transformation even if $G$ is pseudo. Now, from this description, the assignments on arrows and 2-cells follow naturally. For the arrow assignment, start with one of $\bb P(\bb E)_1$ as described in Construction \ref{construction:ConstructionOfSliceOfSpan2Cat} and denoted by the 3-tuple of its components $\langle\alpha,\beta,\gamma\rangle$. As an assignment, declare
    \begin{equation} \label{equation:apx1arrowAssignment}
      \apx_1\langle \alpha, \beta,\gamma\rangle := \langle \alpha, \apx\ast\beta,\gamma\rangle
    \end{equation}
  viewed as a morphism of induced spans $(\sigma\ast G,\tau\ast G) \to (\sigma\ast G',\tau\ast G')$ in $\bb P_1$. It just needs to be check that is is well-defined. For this, recall from Remark \ref{rem:SpansWRTaSubsetofArrows} that a morphism in $\Span_t(\icat)_1$ consists of three 2-cells satisfying the prism equations. However, these follow by interchange laws in the local 2-categories of $\twocat$. For example, for the left prism:
      \begin{align}
      (\alpha\ast L')(\sigma\ast G') &= (\src\ast\beta)(\sigma\ast G') \notag \qquad &\text{(eq. \ref{equation:PrismEquationsSpanCat})} \\
                                    &= (\sigma \ast G)(\apx\ast \beta)\notag \qquad &\text{(interchange)}
    \end{align}
  The right prism is analogous. Thus, the arrow assignment is well-defined. For the 2-cell assignment, recall that one consists of three modifications $\Sigma\colon \alpha \Rightarrow \alpha'$, $\Xi\colon \beta\Rrightarrow\beta'$ and $\mrm T\colon \gamma\Rrightarrow\gamma'$ satisfying the same conditions, namely, that $\Sigma \ast S' = \src\ast\Xi$ and $\mrm T\ast T' = \tgt\ast\Xi$ both hold. The assignment again is to push forward the middle components $\Xi$ by $\apx$ viewed as a trivial modification. Another prism-computation shows that this is well-defined. That $\apx_1$ so defined is a 2-functor is a bit tedious but straightforward to verify using the interchange laws holding in 2-categories of pseudo functors.
\end{construction}

The result is now that $\apx$ is an isomorphism, as announced in Equation \ref{eqn:AnnounceApxIsAnIsoDiscussion}. The proof is somewhat technically involved, but the idea is simple. It suffices to see that $\apx_1$ in Equation \ref{equation:Apx2Functor} induces an isomorphism of 2-categories since $\apx_0$ is the identity 2-functor. The proof below shows that the construction of $\Span(\cat)_1$ ensures that the assignments for $\apx_1$ are all bijections, giving both existence for surjectivity and uniqueness for injectivity. For the given structure in $\Span_t(\icat)_1$, the corresponding structure in $\bb P(\Span(\cat))_1$ can be constructed explicitly. The point is that the given structure ``lifts" through $\apx$. This lifted structure by its construction is sent to the original via $\apx_1$ and is the unique lifted structure with this property.

\begin{prop}
  $\apx_1$ is an isomorphism of 2-categories $\apx_1\colon \bb P(\Span(\cat))_1 \cong \Span_t(\icat)_1$ 
\end{prop}
\begin{proof} 
    The claim is that the 2-functor is a bijection on objects, arrows and 2-cells. Given an object of $\Span_t(\icat)_1$
        $$\xymatrix{
          \ar@{}[dr]|{\substack{\lambda \\\Leftarrow}}\mscr A^{op} \ar[d]_F & \ar@{}[dr]|{\substack{\rho\\\Rightarrow}} \mscr B^{op} \ar[d]^{G} \ar[l]_{S^{op}} \ar[r]^{T^{op}} & \mscr C^{op} \ar[d]^H \\
          \cat & \ar@{=}[l] \cat \ar@{=}[r] & \cat
        }$$
    construct a pseudo functor $\tilde G\colon \mscr B^{op}\to\Span(\cat)_1$ via the assignments summarized by
        $$\xymatrix{
            B \ar@/_1pc/[dd]_f \ar@{}[dd]|{\substack{\alpha\\\Rightarrow}} \ar@/^1pc/[dd]^g & & & & FSB \ar@/_1pc/[dd]_{FSf} \ar@{}[dd]|{\substack{FS\alpha\\\Rightarrow}} \ar@/^1pc/[dd]^{FSg} && \ar[ll]_{\lambda_B} GB \ar@/_1pc/[dd]_{Gf}  \ar@{}[dd]|{\substack{G\alpha\\\Rightarrow}}\ar@/^1pc/[dd]^{Gg} \ar[rr]^{\rho_B} && HTB \ar@/_1pc/[dd]_{HTf}  \ar@{}[dd]|{\substack{HT\alpha\\\Rightarrow}}\ar@/^1pc/[dd]^{HTg} \\
            & & \mapsto & & & &  \\
            C & & & & FSC && \ar[ll]^{\lambda_C} GC \ar[rr]_{\rho_C} && HTC
        }$$
    The diagram on the right is two cylinders sharing a common face $G\alpha$ and whose sides are commutative squares. Notice that this is well-defined since $\lambda$ and $\rho$ are strictly 2-natural so that the squares commute and thus define the correct span morphisms giving the arrows in $\Span(\cat)$. Pseudo naturality follows from the fact that $F$, $G$ and $H$ are pseudo functorial, meaning that this has constructed an object of $\bb P(\bb E)_1$, that is, a span-morphism
        $$\xymatrix{
          \mscr A^{op} \ar[d]_F & \mscr B^{op} \ar[d]^{\tilde G} \ar[l]_{S^{op}} \ar[r]^{T^{op}} & \mscr C^{op} \ar[d]^H \\
          \cat & \ar[l]^-\src \Span(\cat)_1 \ar[r]_-\tgt & \cat
        }$$
    whose squares commute strictly. Now, by construction $\apx$ takes this span morphism to the given span $(\lambda,\rho)$ in $\Span_t(\icat)$. Any other pseudo functor in place of $\tilde G$ doing this would have to make the same assignments. Thus, $\apx$ is a bijection on objects. The process for morphisms and 2-cells of $\Span_t(\icat)_1$ is the same. For an arrow $\langle \alpha,\beta,\gamma\rangle$, with the notation as in Remark \ref{rem:SpansWRTaSubsetofArrows}, construct a pseudo natural transformation $\tilde\beta\colon \tilde N\Rightarrow \tilde G V^{op}$ forming the vertex part of a morphism in $\bb P(\bb E)_1$ whose composition with $\apx$ is the original morphism. Given $Y\in \mscr Y$, take the component to be the span morphism
        $$\xymatrix{
            MSY\ar[d]_{\alpha_{SY}} & \ar[l]_-{\lambda'_Y} NY \ar[d]^{\beta_Y} \ar[r]^-{\rho'_{Y}} & LTY \ar[d]^{\gamma_B} \\
            FUSY & \ar[l]^-{\lambda_{VY}} GVY \ar[r]_-{\rho_{VY}} & HWTY
        }$$
    The prism condition in Equation \eqref{equations:PrismEquationsCat} on morphisms in $\Span_t(\icat)$ means that these squares commute strictly, hence that this assignment is a well-defined morphism of $\Span(\cat)_1$. Pseudo naturality follows from the fact that $\alpha$, $\beta$ and $\gamma$ are pseudo natural. That $\apx_1$ pushes this morphism forward to the original morphism in $\Span_t(\icat)_1$ now follows by construction. Any other such morphism would have the same components, making it the unique one doing so.
\end{proof}

\begin{cor} \label{cor:ApxIsAnIso}
    The map $\apx\colon \bb P(\Span(\cat)) \to \Span_t(\icat)$ is an isomorphism of double 2-categories. As a result, 
        \begin{equation} \label{equation:SpansCongPseudoMonoids}
            \psmon(\bb P(\Span(\cat))) \cong \psmon(\Span_t(\icat))
        \end{equation}
    holds.
\end{cor}
\begin{proof}
    By its construction $\apx_1$ commutes with external composition and units. At least for external composition, this is essentially the statement that the apex of composed spans is the same as the composition of apexes of composable spans. Since both $\apx_1$ and $\apx_0$ are isomorphisms of 2-categories, $\apx$ is thus an isomorphism of double 2-categories. This means that $\bb P(\Span(\cat))$ and $\Span_t(\icat)$ have the same pseudo monoids since these are defined using external composition in the double 2-categories.
\end{proof}

This corollary provides a link in the chain of equivalences leading to the first main theorem, characterizing double fibrations as span-valued lax double pseudo functors.

\begin{theo}[Representation Theorem] \label{theorem:RepresentationTheoremFinalForm}
    There is an equivalence of categories 
        \begin{equation} \label{equation:RepresentationTheoremStatementEquivalence1}
            \dblfib:=\psmon(\Span_c(\fib))\simeq \mbf{I}\Span(\cat)
        \end{equation}
    and in particular there is one
        \begin{equation} \label{equation:RepresentationTheoremStatementEquivalence2}
            \dblfib(\bb B)\simeq \dbltwocat(\bb B^{op},\Span(\cat))
        \end{equation}
    for any fixed double category $\bb B$.
\end{theo}
\begin{proof}
    Starting with the standard equivalence $\fib\simeq \icat$, which is 2-pullback-preserving, take pseudo monoids on each side:
        \[
            \psmon(\Span_c(\fib))\simeq\psmon(\Span_t(\icat)).
        \]
    using the equivalence of Proposition \ref{prop:EquivalentDbl2CatsOfSpans}. Note that this uses the fact that equivalent double 2-categories have equivalent categories of pseudo monoids. But the right side of the above equivalence fits into one
        \[
            \psmon(\Span_t(\icat))\simeq \psmon(\bb P(\Span(\cat)))
        \]
    as in Equation \ref{equation:SpansCongPseudoMonoids} by Corollary \ref{cor:ApxIsAnIso} above. Now, pseudo monoids in $\Span(\bb P(\Span(\cat)))$ and their homomorphisms are precisely lax functors into $\Span(\cat)$ and their homomorphisms
        \[
            \psmon(\bb P(\Span(\cat))) \simeq \mbf{I}\Span(\cat)
        \]
    by Propositions \ref{prop:SpanValuedLaxFunctorsArePseudoMonoids} and \ref{prop:HomomsOfPsMonsAreDblePsNatTransfs}. This proves the first equivalence \ref{equation:RepresentationTheoremStatementEquivalence1}. The second equivalence \ref{equation:RepresentationTheoremStatementEquivalence2} is obtained using the codomain projections to double categories. Since the over-categories are equivalent, they have equivalent fibers.
\end{proof}

We can now directly relate our work to what has been done for monoidal fibrations in \cite{MV}.  In particular, consider the equivalence in \eqref{equation:RepresentationTheoremStatementEquivalence2} for the two cases in Example \ref{ex:both_cases_as_indexed} (taking $\K = \cat$). The result in this case recovers the two legs of the diagram \cite[(24)]{MV} which we reproduce here for the reader's convenience:
    $$\xymatrix@C=-.3in{
         & \fib \simeq\icat \ar[dl]_{\psmon(-)} \ar[dr]^{\text{fix } \mscr B} &\\
        \MonFib\simeq \mbf{MonICat} \ar[d]_{\text{fix } \mscr B} & & \fib(\mscr B)\simeq \twocat_{ps}(\mscr B^{op},\cat) \ar[d]^{\psmon(-)} \\
        \MonFib(\mscr B)\simeq \mbf{Mon2Cat}_{ps}(\mscr B^{op},\cat) & & \psmon(\fib(\mscr B))\simeq \twocat_{ps}(\mscr B^{op},\mbf{MonCat})
    }$$
The two paths to the feet of the diagram are achieved by exchanging the order of taking pseudo monoids and fixing a certain base category. Since the two resulting equivalences coincide only under certain special circumstances, these operations do not generally commute. However, our result shows that the two paths to these two equivalences in the bottom row can be recovered by taking specific choices of double categories built from $\mscr B$.

\begin{cor}
    For any double category $\bb{B}$, we denote by 
    $\dblfib(\bb{B})^{id}$ the full subcategory of $\dblfib(\bb{B})$ whose objects 
    $P: \bb{E} \mr{} \bb{B}$ satisfy that $P_0 = id_{\bb{B}}$. 
    The equivalence in \eqref{equation:RepresentationTheoremStatementEquivalence2} restricts to $$\dblfib(\bb{B})^{id}  \simeq  \dbltwocat(\bb{B}^{op}, \Span(\cat))^{\triangle 1},$$
    where the notation $\dbltwocat(\bb{B}^{op}, \Span(\cat))^{\triangle 1}$ is introduced in Example \ref{ex:both_cases_as_indexed}.
     As a result, for any monoidal category $\ct{B}$, we have two equivalences of categories

\medskip
\noindent 1.
$ \MonFib(\ct{B}) \simeq \Montwocat_{ps} (\ct{B}^{op}, \cat), \qquad$
2. 
    $ \Mon(\fib(\ct{B})) \simeq \twocat_{ps}(\ct{B}^{op}, \Moncat)$.
\end{cor}
\begin{proof}
    To show the first equivalence, we use the following chain of equivalences
        \begin{align} 
            \MonFib(\ct{B}) &\cong  \dblfib(\bb{B})^{id} \notag \\
                            &\simeq \dbltwocat(\bb{B}^{op}, \Span(\cat))^{\triangle 1} \notag \\
                            &\cong \Montwocat_{ps} (\ct{B}^{op}, \cat). \notag
        \end{align}
    The first isomorphism was shown in Example \ref{ex:monoids_in_Fib_as_double_fib} for the objects of the categories, and can be seen to extend to the arrows by  comparing their respective descriptions in Remark \ref{rem:describe_morph_of_dblfib} and in \cite[p.1174]{MV}.  The last isomorphism appeared in Example \ref{ex:both_cases_as_indexed}. For the second statement, we use instead
        \begin{align} 
            \Mon(\fib(\ct{B})) &\cong \dblfib(\bb{D}(\ct{B}))^{id} \notag \\  
                               &\simeq \dbltwocat(\bb{D}(\ct{B})^{op}, \Span(\cat))^{\triangle 1} \notag \\
                               &\cong \twocat_{ps}(\ct{B}^{op}, \Moncat) \notag
        \end{align}
    with the two isomorphisms given by Examples \ref{ex:monoids_in_Fib(B)_as_double_fib} and \ref{ex:both_cases_as_indexed} respectively.
\end{proof}

We close the subsection with an improvement on the main result of \cite{Lambert2021}, which proved a correspondence between \emph{discrete} double fibrations and span-valued lax double functors on \emph{strict} double categories. The techniques developed in this paper allow us to handle the more general case of pseudo double categories.  Recall that a \textbf{discrete fibration} is an ordinary functor $F\colon \mscr F\to \mscr C$ such that the square
    $$\xymatrix{
        \mscr F_1\ar[d]_{F_1} \ar[r]^{d_1} & \mscr F_0 \ar[d]^{F_0} \\
        \mscr C_1 \ar[r]_{d_1} & \mscr C_0
    }$$
is a pullback square of sets. Every discrete fibration is a fibration. Let $\dfib$ denote the full sub-2-category of $\fib$ consisting of the discrete fibrations over arbitrary bases. 

\begin{define}
    A \textbf{discrete double fibration} is a pseudo category in $\dfib$. 
\end{define}

It is well-known that discrete fibrations correspond to contravariant set-valued functors in the sense that there is an equivalence of 2-categories
    \begin{equation}
        \dfib \simeq \iset
    \end{equation}
achieved by an elements construction with a pseudo inverse fibers construction. Here $\iset$ denotes the 2-category of ``indexed sets," that is, presheaves $\mscr C^{op}\to \set$, viewed as a full-sub-2-category of $\icat$. All transformations between indexed sets are 2-natural, so no question about the existence of 2-pullbacks arises. Moreover, this equivalence is pullback-preserving, so we can take pseudo monoids on each side, arriving at an equivalence
    \begin{equation}
        \mathbf{DDblFib} := \psmon(\Span(\dfib))\simeq \psmon(\Span(\iset))
    \end{equation}
between a category of discrete double fibrations and one of pseudo categories in indexed sets. Note that the double functors on the left side of this equivalence need not be between strict double categories since we are considering pseudo categories. An analysis similar to that above can be applied to the project of unpacking the data on the right side of the equivalence.

\begin{lemma}
    Pseudo categories in indexed sets are in one-to-one correspondence with span-valued lax functors $\bb D^{op}\to \Span$ on pseudo double categories $\bb D$. That is, there is an equivalence
        \begin{equation}
            \psmon(\Span(\iset))\simeq \mbf I\Span
        \end{equation}
    between pseudo categories in indexed sets and indexed spans of sets.
\end{lemma}
\begin{proof}
    Pseudo categories in $\iset$ are pseudo monoids in spans in $\iset$. By either (1) unpacking the latter data by hand, or (2) working through an isomorphism using the lower-dimensional apex functor $\apx\colon \Span\to\set$ as in the proof of Proposition \ref{prop:SpanValuedLaxFunctorsArePseudoMonoids}, one can see that pseudo monoids in spans in $\iset$ are indeed precisely contravariant span-valued lax functors on double categories.
\end{proof}

\begin{theo}
    There is an equivalence of categories
        \begin{equation}
            \mathbf{DDblFib}\simeq \mathbf I\Span
        \end{equation}
    restricting to an equivalence
        \begin{equation} \label{eq:equiv_Pare_Lambert}
            \mathbf{DDblFib}(\bb D) \simeq \lax(\bb D^{op},\Span)
        \end{equation}
    for any pseudo double category $\bb D$.
\end{theo}
\begin{proof}
    This is just a matter of coordinating the equivalence between double fibrations and pseudo monoids in spans in indexed sets and that between the latter and indexed spans of sets as in the lemma. Restricting to the fibers of the appropriate projection morphism to double categories gives the equivalence over a fixed base double category.
\end{proof}

\subsection{Elements Correspondence and Examples}

In this subsection, we will unpack the correspondence from Theorem \ref{theorem:RepresentationTheoremFinalForm} and show how several of the examples encountered so far have their analogues on one or the other side of it.  Specifically, the elements construction of the correspondence is the object assignment of the composite functor $\mbf I\Span(\cat) \to \dblfib$ passing through the various categories of pseudo monoids described in the theorem. In particular, we view a lax functor $F: \mathbb{D}^{op} \to \Span(\cat)$ as consisting of two parts, namely, pseudo functors
    \[ 
        F_0: \mathbb{D}_0^{op} \to \Span(\cat)_0 = \cat 
    \]
and
    \[ 
        F_1: \mathbb{D}_1^{op} \to \Span(\cat)_1 
    \]
where $\Span(\cat)_1$ has objects spans of categories and morphisms of such spans. The apex functor $\apx: \Span(\cat)_1 \to \cat$, forgetting the legs of a given span (morphism), composes with $F_1$ to yield a category-valued pseudo functor
    \[ 
        \mathbb{D}_1^{op} \to^{F_1} \Span(\cat)_1 \to^{\apx} \cat. 
    \]
Thus, $F_0$ and $\apx F_1$ are both pseudo functors into $\cat$, and as such, we can apply the ordinary elements construction to each separately. Thus we get cloven fibrations
    \[ 
        \dblelt(F)_0 \to \mathbb{D}_0 \quad \mbox{  and  } \quad \dblelt(F)_1 \to \mathbb{D}_1. 
    \]
These will be, respectively, the category of objects and the category of arrows of our double category $\dblelt(F)$. The next result gives a complete description of how these fibrations are underlying a double fibration $\Pi\colon \dblelt(F)\to\D$. First, a bit of notation: given a proarrow $m: A \slashedrightarrow B$ of $\mathbb{D}$, we will write the span $Fm$ as 
    \begin{equation} \label{eq:L_mR_mnotation}
        \vcenter{\xymatrix{& Fm \ar[dl]_{L_m} \ar[dr]^{R_m} & \\ 
                    FA & & FB
    }}
    \end{equation}
The next result includes a complete description of the rest of the structure. For any further proarrow $n\colon B\hto C$, denote the laxity comparison cell for composition as the span morphism
    $$\xymatrix{
        FA \ar@{=}[d] & \ar[l] Fm\times_{FB}Fn \ar[r] \ar[d]^{\phi_{m,n}} & FC \ar@{=}[d] \\
        FA & \ar[l] F(m\otimes n) \ar[r] & FC
    }$$
The important part is of course the apex of the span morphism given by the functor $\phi_{m,n}$.

\begin{theo} \label{th:elements_explicit}
    The double elements construction associated to a lax double pseudo functor $F\colon \mathbb D^{op}\to\Span(\cat)$ is the projection from the pseudo double category $\dblelt(F)$ in which the object category $\dblelt(F)_0$ is given by 
        \begin{enumerate}
            \item objects: pairs $(C,X)$ with $X\in FC$;
            \item arrows: pairs $(f, \bar f)\colon (C,X)\to (D,Y)$ with $f\colon C\to D$ and $\bar f\colon X\to f^*Y$ in $FC$;
        \end{enumerate}
    and the proarrow category $\dblelt(F)_1$ is given by
        \begin{enumerate}
            \item[3.] proarrows: pairs $(m,\bar m)\colon (C,X)\hto (D,Y)$ with $\bar m \in Fm$, $L_m\,\bar m = X$, and $R_m\,\bar m = Y$ (see \eqref{eq:L_mR_mnotation});
            \item[4.] cells: pairs $(\theta,\bar\theta)$ displayed
                $$\xymatrix{
                    \ar@{}[dr]|{(\theta,\bar\theta)} (A,X) \ar[d]_{(f,\bar f)}\ar[r]^{(m,\bar m)} |@-{|} & (B,Y) \ar[d]^{(g,\bar g)} \\
                    (C,Z) \ar[r]_{(n,\bar n)}|-@{|} & (D,W)
                }$$
            with
                $$\xymatrix{
                    \ar@{}[dr]|{\theta} A \ar[d]_f \ar[r]^m|-@{|} & B \ar[d]^g \\
                    C \ar[r]_{n}|-@{|} & D
                }$$
            a cell of $\mathbb D$ and $\bar\theta\colon \bar m \to \theta^*\bar n$ an arrow of $Fm$ such that $L_m\,\bar\theta = \bar f$ and $R_m\,\bar\theta =\bar g$ both hold
        \end{enumerate}
    with internal composition and units given by
         \begin{enumerate}
            \item[5.] 
                For ordinary arrows $(f,\bar f)\colon (A,X)\to (B,Y)$ and $(g,\bar g)\colon (B,Y)\to (C,Z)$, the composite is $(gf, \phi_{f,g}f^*(\bar g)\bar f)\colon (A,X)\to (C,Z)$.
            \item[6.]  
                For composable cells $(\theta,\bar\theta)\colon (m,\bar m)\Rightarrow (n,\bar n)$ and $(\delta,\bar\delta)\colon (n,\bar n)\Rightarrow (p,\bar p)$, the composite is given by $(\delta\theta,\phi_{\theta,\delta}\theta^*(\bar \delta)\bar\theta)\colon (m,\bar m)\Rightarrow (p,\bar p)$.
            \item[7.]  
                Internal units are $(1, 1)\colon (C,X)\to (C,X)$ and $(1, 1)\colon (m,\bar m)\Rightarrow (m,\bar m)$.
        \end{enumerate}
    The external source and target functors $\src,\tgt\colon \dblelt(F)_1\rightrightarrows\dblelt(F)_0$ then take a cell $(\theta,\bar\theta)$ above to $f$ and $g$, respectively. External composition and units are as follows.
        \begin{enumerate}
            \item[8.]  
                For proarrows $(m,\bar m)\colon (A,X) \hto (B,Y)$ and $(n,\bar n)\colon (B,Y)\hto (C,Z)$, the composite is given by $(m\otimes n,\phi_{m,n}(\bar m,\bar n))\colon (A,X)\hto (C,Z)$ where $\phi_{m,n}$ denotes the apex morphism of the laxity comparison cell associated to the pair of proarrows $(m,n)$ via $F$.
            \item[9.]   
                For cells
                        $$\xymatrix{
                            \ar@{}[dr]|{(\theta,\bar\theta)} (A,X) \ar[d] \ar[r]^{(m,\bar m)}|-@{|} & \ar@{}[dr]|{(\delta,\bar\delta)} (B,Y) \ar[d] \ar[r]^{(p,\bar p)}|-@{|} & (C,Z) \ar[d] \\
                            (D,U) \ar[r]_{(n,\bar n)}|-@{|} & (M,V)  \ar[r]_{(q,\bar q)}|-@{|} & (N,W)
                        }$$
                    the external composite is 
                        $$\xymatrix{
                            \ar@{}[drr]|{(\theta\otimes\delta,\bar{\theta\otimes\delta})} (A,X) \ar[d]_{(f,\bar f)}\ar[rr]^{(m\otimes p,\phi(\bar m,\bar p))} |@-{|} && (B,Y) \ar[d]^{(g,\bar g)} \\
                            (C,Z) \ar[rr]_{(n\otimes q,\phi(\bar n,\bar q))}|-@{|} && (D,W)
                        }$$
                    where $\bar{\theta\otimes\delta}$ is the morphism
                        $$\xymatrix{
                            \phi_{m,p}(\bar m,\bar p) \ar[r]^-{\phi(\bar\theta,\bar\delta)} & \phi_{m,p}(\theta^*(\bar n),\delta^*(\bar q)) \ar[r]^\cong & (\theta\otimes\delta)^*\phi_{n,q}(\bar n, \bar q)
                        }$$
                    in $F(n\otimes q)$ where `$\cong$' is the appropriate component of the structure iso
                        $$\xymatrix{
                        \ar@{}[drr]|{\cong}Fn\times_{FM}Fq \ar[d]_{\phi_{n,q}} \ar[rr]^{\theta^*\times_{g^*}\delta^*} & & Fm\times_{FB}Fp \ar[d]^{\phi_{m,p}} \\
                        F(n\otimes q) \ar[rr]_{(\theta\otimes \delta)^*} & & F(m\otimes p)
                        }$$
                    coming with $F$.
            \item[10.] 
                The external unit $y\colon \dblelt(F)_0\to \dblelt(F)_1$ is the functor taking $(C,X)$ to the proarrow $(y_C,\iota_C(X))$ where $\iota_C$ is the $C$-component of the identity transformation. On arrows $(f,\bar f)\colon (C,X)\to (D,Y)$, take the image to be the cell $(y_f,\bar{y_f})$ where $\bar{y_f}$ is the morphism
                    $$\xymatrix{
                        \iota_C(X) \ar[r]^{\iota_C\bar f} & \iota_C(f^*Y) \ar[r]^{\cong} & y_f^*\iota_D(Y)
                    }$$
                with the iso coming from the pseudo naturality iso $\iota_f$.
        \end{enumerate}
    The projection double functor $\Pi\colon \dblelt(F)\to \bb D$ is a double fibration.\qed
\end{theo}

\begin{example}
  The family double fibration $\Pi\colon \dblfam(\C)\to \Span$ of Example \ref{example:FamilyDoubleFibration} corresponds to the family lax double pseudo functor $[-,\C]\colon \Span^{op}\to\Span(\cat)$ of Example \ref{example:FamilyLaxDoublePseudoFunctor}.
\end{example}

\begin{example}
  The codomain projection $\cod \colon \bb D^{\mbf 2}\to\bb D$ associated to a double category with suitable finite limits, as described in Example \ref{example:DoubleCodomainFibration}, corresponds to the lax double pseudo functor $\bb D^{op}\to \Span(\cat)$ associating to each object $D$ the slice $\bb D_0/D$ as described in Example \ref{example:SliceLaxDoublePseudoFunctor}.
\end{example}

We observe now that Theorem \ref{th:elements_explicit} applies in particular to a lax functor
$F\colon \mathbb D^{op}\to |\Span(\cat)|$, that in view of Example \ref{ex:lax_double_pseudo_functor_Dd--Ed}
can be seen as a lax double pseudo functor 
    \[
        F\colon \mathbb D^{op}\to |\Span(\cat)|_d \mr{} \Span(\cat). 
    \]
Note that if we think of the lax functor $F$ with the notation as in Definition \ref{def:laxfunctor}, then the so-obtained lax double pseudo functor (with the notation as in Definition \ref{def:lax_double_pseudo_functor}), has $F_0$ and $F_1$ functors instead of pseudo functors, $\phi$ and $\iota$ natural instead of pseudo natural transformations, and identities in place of the structural isomorphisms $\Phi, \Lambda$, and $\Rho$ in Definition \ref{def:lax_double_pseudo_functor}. With this in mind, we can apply the formulas in Theorem \ref{th:elements_explicit} to this case, and we get:

\begin{cor} \label{cor:locally_discrete_elements_construction}
The elements construction associated to a lax functor $F\colon \mathbb D^{op}\to|\Span(\cat)|$ is the projection from the pseudo double category $\dblelt(F)$ as described in Theorem \ref{th:elements_explicit}, with the following simplifications:

\begin{itemize}
    \item In item 6, $\phi_{\theta,\delta}$ is an identity, and
    \item In items 9 and 10, the structure isomorphisms denoted by $\cong$ are respectively identities.
\end{itemize}
    The natural projection double functor $\Pi\colon \dblelt(F)\to \bb D$ is a double fibration.\qed
\end{cor}

We refer to the double fibrations associated to these lax functors as {\bf locally discrete}.

\begin{example}
    There is a functor $\Span \mr{} |\Span(\cat)|$, mapping each set to its associated discrete category. Composing thus a lax functor $F\colon \mathbb D^{op}\to \Span$ with it, and applying Corollary \ref{cor:locally_discrete_elements_construction}, we recover the double category $\dblelt(F)$ as originally constructed in \cite{PareYoneda}. This is of course the value on objects of the equivalence in \eqref{eq:equiv_Pare_Lambert}.
\end{example}

\begin{example}\label{ex:jaz-myers}
A double category $\int\int F$ is constructed in \cite[Def. 5.3]{DJM}, for a lax functor 
$F: \bb{D}^{op} \mr{} \prof$
taking its values in the double category of categories and profunctors, and referred to as a double Grothendieck construction (the covariant case is the one considered in \cite[Def. 5.3]{DJM}, but it is clear how to adapt it to the contravariant case). We can compose $F$ with the lax functor 
$\prof \mr{} |\Span(\cat)|$, associating to each profunctor 
the span given by 
its category of elements. For details on this lax functor see \cite[\S 10]{GP_Span_Cospan}. It can then be checked that the formulas in Corollary \ref{cor:locally_discrete_elements_construction} yield precisely those in \cite[Def. 5.3]{DJM}, showing then that the projection 
$\int\int F \mr{} \bb{D}$ is a 
(locally discrete) double fibration. It could be interesting in the future to interpret the double fibration properties when this construction is applied to open dynamical systems as in \cite{DJM}.
\end{example}

\section{Double Fibrations as Internal Fibrations} \label{sec:internal}

The notion of internal fibration in a 2-category was originally defined by Street \cite{internal_fibration}, and admits an equivalent formulation in terms of representables
(see for example \cite[Thm. 3.1.3]{notions_of_fibration}). 
The main result of this section, Corollary \ref{coro:double_fibration_iff_internal}, is that a strict double functor $P$ between pseudo double categories defines a double fibration (as in Definition \ref{def:double_fibration}) precisely when it is an internal fibration in the 2-category $\Dblcatv$ of pseudo double categories, pseudo double functors, and transformations. 

While proving this result, we noticed that the {\em natural} context in which to show it is to generalize $P$ to be pseudo, and to consider what it would mean for it to be an internal fibration in the three 2-categories $\Dblcatvs \subseteq \Dblcatv \subseteq   \Dblcatvl$ whose arrows are respectively the strict, pseudo, and lax double functors. 
Theorem \ref{th:charact_internal_fibrations}, whose proof occupies most of this section, and from which Corollary \ref{coro:double_fibration_iff_internal} follows as one of six possible cases (see Remark \ref{rem:splits_in_six}), provides a characterization of internal fibrations in these three 2-categories.

Recall that we write the arrows of a pseudo double category in the vertical direction, and the pro-arrows in the horizontal one. By this notational choice, transformations as in Definition \ref{define:TransformationOfInternalPseudoFunctors} amount to vertical transformations that we spell out below:

\begin{define} \label{def:vert_nat_and_3_2-cat}
    A \textbf{(vertical) transformation} $\alpha: F \Mr{} G$ between lax double functors $F,G: \bb{D} \mr{} \bb{E}$ is given by: 
        \begin{enumerate}
            \item for each object $X$ of $\bb{D}$,  a vertical arrow (that is an arrow $\alpha_X: F(X) \to G(X)$ of $\ct{E}_0$), and 
            \item for each horizontal arrow $M$ of $\bb{D}$ (that is an object of $\ct{E}_1$) a double cell (that is an arrow $\alpha_M: FM \mr{} GM$ of $\ct{E}_1$)
        \end{enumerate}
    These families of arrows are required to be natural with respect to the arrows of $\ct{E}_0$, resp. $\ct{E}_1$, and to satisfy the two axioms below:
    for each $X$, and {resp. for each $\stackrel{M}{\hto} \stackrel{N}{\hto}$}, the diagrams
    
    \begin{equation} \label{eq:vert_nat_satisfies}
    \vcenter{\xymatrix{
    y{F(X)} \ar[r]^{y({\alpha_X})} 
    \ar[d]_{\iota^F_X} 
    & y{G(X)} \ar[d]^{\iota^G_X} 
    \\
    Fy(X) \ar[r]_{\alpha_{y(X)}} & Gy(X) \ 
    }}
    \quad
    {\vcenter{\xymatrix@C=4pc{
    F(M) \otimes F(N) \ar[r]^-{\alpha_M \otimes \alpha_N} \ar[d]_{\phi_{M,N}^F} &  G(M) \otimes G(N)
    \ar[d]^{\phi^G_{M,N}}
    \\
    F(M \otimes N) \ar[r]_-{\alpha_{M \otimes N}} & G(M \otimes N)
    }}}
    \end{equation}
    both commute in $\mscr E_1$. We say that {\em  \eqref{eq:vert_nat_satisfies} holds for $\alpha$} to mean that the families $\alpha_X$, $\alpha_M$ satisfy these axioms. Vertical transformations are the 2-cells of three 2-categories $\Dblcatvs \subseteq \Dblcatv \subseteq   \Dblcatvl$  whose arrows are respectively strict, pseudo, and lax double functors between pseudo double categories.
\end{define}

\begin{remark} \label{rem:Fstrict_then_alpha_determined}
If follows from the diagram on the left in \eqref{eq:vert_nat_satisfies} that if $F$ is a pseudo double functor, then $\alpha_{y(X)}$ can be uniquely defined such that the diagram commutes, and is determined by $\alpha_X$.
\end{remark}

\begin{remark} \label{rem:intfibration_meaning}
The definition of what it means for an arrow in an arbitrary 2-category to be a fibration is due to Street \cite{internal_fibration}, for a nice exposition see \cite{notions_of_fibration}. By {\em instantiating} this general definition (as can be found for example in \cite[Definition 3.1]{notions_of_fibration}) to the 2-category $\Dblcatvl$, we observe that a lax double functor $P: \mathbb{E} \to \mathbb{B}$ is a fibration in $\Dblcatvl$ when for each vertical transformation
\begin{equation} \label{eq:betainfibrationindblcatv}
 \vcenter{\xymatrix{\mathbb{X} \ar[r]^{E} \ar[dr]_{B}^{\Uparrow \, \beta}  & \mathbb{E} \ar[d]^{P} \\ & \mathbb{B}} }
 \end{equation}
there is a lax double functor $E': \mathbb{X} \to \mathbb{E}$ and a vertical transformation $\alpha: E' \Rightarrow E$, such that $P \alpha = \beta$, satisfying:
for any lax double functors $X: \mathbb{Y} \to \mathbb{X}$, 
$E'': \mathbb{Y} \to \mathbb{E}$, 
and for vertical transformations $\xi$, $\gamma$
\begin{equation} \label{eq:alpha_satisfies_1}
\vcenter{\xymatrix{\mathbb{Y} \ar[rr]^{E''} \ar[rd]_{X} & \ar@{}[d]|{\Downarrow \, \xi} & \mathbb{E} \\
& \mathbb{X}  \ar[ur]_{E}
}}
\qquad , \qquad
\vcenter{\xymatrix{\mathbb{Y} \ar[rr]^{E''} \ar[d]_{X} \ar@{}[rrd]|{\Downarrow \, \gamma} && \mathbb{E} \ar[d]^{P} \\
\mathbb{X} \ar[r]_{E'} & \mathbb{E} \ar[r]_{P} & \mathbb{B}
}}
\quad
\mbox{ such that }
\quad
P\xi = \vcenter{\xymatrix{\mathbb{Y} \ar[rr]^{E''} \ar[d]_{X} \ar@{}[rrd]|{\Downarrow \, \gamma} && \mathbb{E} \ar[d]^{P} \\
\mathbb{X} \ar@<-1ex>[r]_{E} \ar@{}[r]|{\Downarrow \, \alpha} \ar@<1ex>[r]^{E'} & \mathbb{E} \ar[r]_{P} & \mathbb{B} }}
\end{equation}
there is a unique vertical transformation $\zeta: E'' \Mr{} E' X$ such that
\begin{equation} \label{eq:alpha_satisfies_2}
\vcenter{\xymatrix{
\mathbb{Y} \ar[rr]^{E''} \ar[rd]_{X} & \ar@{}[d]|{\Downarrow \, \xi} & \mathbb{E} \\
& \mathbb{X}  \ar[ur]_{E}
}}
=
\vcenter{\xymatrix{
\mathbb{Y} \ar[rr]^{E''} \ar[rd]_{X} & \ar@{}[d]|{\Downarrow \, \zeta} & \mathbb{E} \\
& \mathbb{X}  \ar@<1ex>[ur]^{E'} \ar@<-1ex>[ur]_{E} \ar@{}[ur]|{\stackrel{\Rightarrow}{\alpha}}
}}
\mbox{ and }
\vcenter{\xymatrix{
\mathbb{Y} \ar[rr]^{E''} \ar[rd]_{X} & \ar@{}[d]|{\Downarrow \, \zeta} & \mathbb{E'} \ar[r]^{P} & \mathbb{E} \\
& \mathbb{X}  \ar[ur]_{E}
}}
\!\!\!\!\!\!
= 
\vcenter{\xymatrix{\mathbb{Y} \ar[rr]^{E''} \ar[d]_{X} \ar@{}[rrd]|{\Downarrow \, \gamma} && \mathbb{E} \ar[d]^{P} \\
\mathbb{X} \ar[r]_{E'} & \mathbb{E} \ar[r]_{P} & \mathbb{B}
}}
\end{equation}
The same statements above, replacing all appearances of lax double functors by pseudo, resp. strict double functors, describe fibrations in $\Dblcatv$, resp. $\Dblcatvs$.
\end{remark}

We state now Theorem \ref{th:charact_internal_fibrations}, from which the main result of this section, Corollary \ref{coro:double_fibration_iff_internal}, follows. 
This theorem consists of three statements that we denote by $\mathbf{(L)}$, $\mathbf{(P)}$, and $\mathbf{(S)}$, characterizing 
when a pseudo double functor is an 
internal fibration in each of the three 2-categories in Definition \ref{def:vert_nat_and_3_2-cat} respectively.
Its proof is given, as is usual in 2-dimensional category theory, by showing the lax case $\mathbf{(L)}$ and {\em restricting} it to the pseudo and strict case.

\begin{theo} \label{th:charact_internal_fibrations}
Let $P: \mathbb{E} \to \mathbb{B}$ be a pseudo double functor between pseudo double categories, given by the data
\begin{equation}
\xymatrix@R=3pc@C=3pc{
{\ct{E}}_1 \times_{{\ct{E}}_0} {\ct{E}}_1
\ar@{}[dr]|{\phi^P \ \stackrel{\cong}{\Rightarrow}} \ar[d]_{P_1 \times_{P_0} P_1} \ar[r]^-{\otimes_{\bb{E}}} &
{\ct{E}}_1 
\ar@{}[dr]|{\iota^P \ \stackrel{\cong}{\Leftarrow}}
\ar[d]^{P_1} \ar@<1.25ex>[r]^{\src_{\bb{E}}} \ar@<-1.25ex>[r]_{\tgt_{\bb{E}}} & {\ct{E}}_0 \ar[d]^{P_0} \ar[l]|{y_{\bb{E}}} \\ 
{\ct{B}}_1 \times_{{\ct{B}}_0} {\ct{B}}_1 \ar[r]^-{\otimes_{\bb{B}}} & {\ct{B}}_1 \ar@<1.25ex>[r]^{\src_{\bb{B}}} \ar@<-1.25ex>[r]_{\tgt_{\bb{B}}} & {\ct{B}}_0 \ar[l]|{y_{\bb{B}}}}
\end{equation}
\smallskip
\noindent $\mathbf{(L)} \,$ 
$P$ is a fibration in the 2-category $\Dblcatvl$ if and only if it satisfies simultaneously:
\begin{enumerate}
    \item $P_0$ and $P_1$ are fibrations,
    \item we can choose cleavages for $P_0$ and $P_1$ such that $\src_\bb{E}$ and $\tgt_\bb{E}$ are cleavage-preserving.
\end{enumerate}    

Recall that in this case $P_1 \times_{P_0} P_1$ is a fibration and we can choose its cleavage pointwise (see Proposition \ref{prop:pullbackspointwise}).
    
\smallskip
\noindent $\mathbf{(P)}  \,$ $P$ is a fibration in the 2-category $\Dblcatv$ if and only if it satisfies $1$, $2$, and:    
    
\begin{enumerate}
    \item[3.]  $y_\bb{E}$ and $\otimes_{\bb{E}}$ are Cartesian-morphism preserving.
\end{enumerate}

\smallskip
\noindent $\mathbf{(S)}  \,$ If $P$ is a strict double functor, then 
$P$ is a fibration in the 2-category $\Dblcatvs$ if and only if it satisfies $1$, $2$, and:    
    
\begin{enumerate}
    \item[3s.]  $y_\bb{E}$ and $\otimes_{\bb{E}}$ are cleavage-preserving.
\end{enumerate}
\end{theo}

\begin{remark}\label{rem:abouse notation}
Note that, even if $P$ is a lax double functor, it still makes sense to consider the statements obtained in Remark \ref{rem:intfibration_meaning} by replacing all other appearances of lax double functors by pseudo, resp. strict double functors. We say, slightly abusing the language, that this defines what it means for a lax double functor $P$ to be a fibration in $\Dblcatv$, resp. $\Dblcatvs$. 

We note that with this abuse of language we could also consider $P$ to be a pseudo double functor in item $\mathbf{(S)}$ of Theorem \ref{th:charact_internal_fibrations} (and the statement in item $\mathbf{(S)}$ still holds, with the same proof that we give for $P$ strict).
\end{remark}

Before proving Theorem \ref{th:charact_internal_fibrations}, we observe some of its consequences.

\begin{cor}
For a pseudo double functor $P$ we have the implication
$$
\mbox{ fibration in } \Dblcatv \qquad \Rightarrow \qquad 
\mbox{ fibration in } \Dblcatvl,
$$
and for a 
strict double functor $P$ (or pseudo if we abuse the language) we have the implications
$$
\mbox{ fibration in } \Dblcatvs \qquad \Rightarrow \qquad 
\mbox{ fibration in } \Dblcatv \qquad \Rightarrow \qquad 
\mbox{ fibration in } \Dblcatvl.
$$
These implications do not follow directly from the descriptions of these notions in Remark \ref{rem:intfibration_meaning}.
\end{cor}

Recalling Definition \ref{def:unwinded_double_fibration}, we have:

\begin{cor} \label{coro:double_fibration_iff_internal}
A strict double functor $P$ defines a double fibration if and only if it is a fibration in $\Dblcatv$.
\end{cor}

\begin{remark} \label{rem:splits_in_six}
In fact, recalling our abuse of language, that of Corollary \ref{coro:double_fibration_iff_internal} is one of six possible different situations for a pseudo double functor $P$, depending on if it is a fibration in each of the 2-categories $\Dblcatvl$, $\Dblcatv$, and $\Dblcatvs$, and whether or not $P$ is strict. Given a diagram  $\src: P_1 \rightarrow P_0 \leftarrow P_1: \tgt$  in $\cfib$, Theorem \ref{th:charact_internal_fibrations}, Lemma \ref{lem:laxfunctorascategoryobject}, and Corollary 
\ref{cor:3differentcatobjectsinFib} can be used to show that having a pseudo-category structure
$(P_1,P_0,\src,\tgt,\iota,\phi,{\mathfrak a},{\mathfrak r},{\mathfrak l})$ in the following six 2-categories is equivalent (by choosing the cleavages of $P_0$ and $P_1$) to having a $P$ fitting each of these six cases. It is convenient to organize these 2-categories in a table according to the properties of $P$ in each case: 

\begin{center}
\begin{tabular}{|c|c|c|}
 \hline 
 $P$ fibration in: & $P$ strict double functor & $P$ pseudo double functor  \\ \hline 
    $\Dblcatvs$ & $\cfib$ & $\cfib^p$ \\
    $\Dblcatv$ & $\fib$ & $\fib^p$ \\
    $\Dblcatvl$ & $\Arrs(\cat)$ &  $\Arrp(\cat)$ \\ \hline 
\end{tabular}
\end{center}

The 2-categories $\fib^p$ (resp. $\cfib^p$) are defined as the sub-2-categories of $\Arrp(\cat)$, full on 2-cells, whose objects are fibrations with a chosen cleavage and whose arrows satisfy that the functor on top of the square is Cartesian, resp. cleavage, preserving.

The statement of Corollary \ref{coro:double_fibration_iff_internal} corresponds to the 2-category $\fib$ in the table. Considering instead $\cfib$, we have that a strict double functor $P$ defines a pseudo category in $\cfib$ if and only if it is a fibration in $\Dblcatv_s$.
We omit to write explicitly the remaining four statements.
\end{remark}

We begin now working towards the proof of Theorem \ref{th:charact_internal_fibrations}. We will use
the following lemma.

\begin{lemma} \label{lem:alpha_X_and_alpha_M_are_Cartesian}
Let $P: \mathbb{E} \to \mathbb{B}$ be a  pseudo double functor that is a fibration in any of the three 2-categories in the remark above\footnote{For fibrations in $\Dblcatvs$, either assume $P$ to be strict or abuse the language as in the remark above.}. Let $\beta$ as in \eqref{eq:betainfibrationindblcatv}, and let $E'$ and $\alpha$ satisfy the conditions in \eqref{eq:alpha_satisfies_1}, \eqref{eq:alpha_satisfies_2}.  Then:

\begin{enumerate}
    \item For each object $X$ of $\mathbb{X}$, $\alpha_{X}\colon E'X \to EX$ is Cartesian with respect to $P_0\colon \ct{E}_0 \to \ct{B}_0$.
    
    \item For each horizontal arrow $M$ of $\mathbb{X}$,  $\alpha_{M}\colon E'M \to EM$ is Cartesian with respect to $P_1\colon \ct{E}_1 \to \ct{B}_1$.
\end{enumerate}
\end{lemma}

\begin{proof}
Though part of the method in this proof could be considered standard for the theory of internal fibrations, there is a subtle point that doesn't allow $P$ to be taken lax, so we prefer to include it in full. 
For item 1, let $X$ be an object of $\bb{X}$. 
Recall that for $\alpha_{X}$ to be $P_0$-Cartesian, it has to satisfy:

\begin{equation} \label{eq:P0cartesian}
\begin{tabular}{lc}
$\vcenter{\xymatrix{
E'' \ar@{.>}[d]_{\exists ! \ \zeta} \ar[dr]^{\xi} \\
E'X \ar[r]_{\alpha_{X}} & EX  }}$     
&   
$\qquad \mbox{ in } \ct{E}_0$ \\
$\vcenter{\xymatrix{
P_0 E'' \ar[d]_{\gamma} \ar[dr]^{P\xi} 
\\
P_0 E'X \ar[r]_{P\alpha_{X}} & P_0 EX 
}}$
 & 
$\qquad
\mbox{ in } \ct{B}_0$
\end{tabular}
\end{equation}
(It is convenient here to refer to the value of $P$ at an object $E$ of $\bb{E}$ as $P_0 E$ instead of $PE$).

Let $E''$, $\xi$, and $\gamma$ as above. We take $\mathbb{Y}$ to be the terminal double category $\{\bullet\}$, and define, abusing the notation, 

\begin{itemize}
    \item $X: \mathbb{Y} \to \mathbb{X}$ and  
$E'': \mathbb{Y} \to \mathbb{E}$ as the unique strict double functors mapping $\bullet$ to $X$, resp. $E''$,

 and, using Remark \ref{rem:Fstrict_then_alpha_determined},

\item $\xi: E'' \Mr{} EX$ and $\gamma: P_0 E'' \Mr{} P_0 E'X$ as the unique vertical transformations such that $\xi_{\bullet} = \xi$ and $\gamma_{\bullet} = \gamma$.
\end{itemize}

Again by Remark \ref{rem:Fstrict_then_alpha_determined}, $\zeta$ corresponds bijectively to a vertical transformation $\zeta: E'' \Mr{} E' X$ such that $\zeta_{\bullet} = \zeta$, and the conditions in \eqref{eq:P0cartesian} are immediately seen to match those in \eqref{eq:alpha_satisfies_1} and \eqref{eq:alpha_satisfies_2}. This shows that $\alpha_{X}$ is $P_0$-Cartesian.

The proof of item 2 is completely analogous and is obtained from the proof of item 1 by ``replacing $0$ by $1$", that is replacing objects by horizontal arrows and vertical arrows by double cells. The terminal double category is replaced by $\mathbb{Y}=\{\bullet_1 \hto \bullet_2\}$, so that horizontal arrows $M$ of a double category induce strict double functors $M$ with domain $\mathbb{Y}$, and 
by Remark \ref{rem:Fstrict_then_alpha_determined} 
double cells $M \mr{} N$ correspond to 
vertical transformations $M \Mr{} N$ between the induced double functors.
\end{proof}

\begin{remark}
For any pseudo double functor $P$ that is a fibration in $\Dblcatvl$  
there is a {\em natural} way to show item 1 in Theorem \ref{th:charact_internal_fibrations} using Lemma \ref{lem:alpha_X_and_alpha_M_are_Cartesian}: 

\begin{itemize}
\item Take $\bb{X}$ to be the terminal double category for showing that $P_0$ is a fibration, as shown below; 
$$
\mbox{Show} \qquad
\vcenter{\xymatrix{
B^* \ar@{.>}[r]^{u^*E} & E \\
B \ar[r]^u & P E
}} 
\qquad \mbox{ as follows: }  \qquad
\vcenter{\xymatrix{\{\bullet \} \ar[r]^{E} \ar[dr]_{B}^{\Uparrow \, u}  & \mathbb{E} \ar[d]^{P} \\ & \mathbb{B}} }
\qquad \stackrel{Rem. \ref{rem:intfibration_meaning}}{\leadsto} \qquad
\vcenter{\xymatrix{\{\bullet \} \ar@<1.5ex>[r]^{E} \ar@{}[r]|{u^*E \ \Uparrow} \ar@<-1.5ex>[r]_{B^*} & \bb{E} }}
$$
\item Take $\bb{X} = \{\bullet_1 \hto \bullet_2\}$ for showing that $P_1$ is a fibration, which we omit to show explicitly.
\end{itemize}

In this way, $P_0$ and $P_1$ can be seen to be fibrations, but when doing so each Cartesian lifting is chosen separately and so the cleavages will not  necessarilly be preserved by $\src$ and $\tgt$ as required in item 2. That is why this isn't helpful for proving Theorem \ref{th:charact_internal_fibrations}, and instead the following trick is used. 

We refer to the data $(B,E,u)$ as ``lifting data''. The {trick} for showing item 2, and thus the $\Rightarrow$ implication in the theorem, is to construct instead a larger double category $\bb{X}$, which contains simultaneously the information of all the possible {lifting data} (for $P_0$ as above and also for $P_1$). We do this now:
\end{remark}

\begin{proof}[Proof of Theorem \ref{th:charact_internal_fibrations}]
$(\Rightarrow)$ Let $P: \mathbb{E} \to \mathbb{B}$ be a pseudo double functor.  We construct the following double category $\bb{X}$:

\begin{itemize}
    \item The objects of $\bb{X}$ are the triples $(B,E,u)$, where $u\colon B \to PE$ is a vertical arrow of $\bb{B}$. Vertical arrows of $\bb{X}$ are only identities.

    \item The horizontal arrows are the  triples $(M,N,\gamma)\colon (B,E,u) \hto (C,D,v)$, where $M\colon B \hto C$ is a horizontal arrow of $\bb{B}$, $N\colon E \hto D$ is a horizontal arrow of $\bb{E}$, and $\gamma$ is a double cell of $\bb{B}$,
    $$\xymatrix{\ar@{}[dr]|{\gamma } B \ar[d]_{u}\ar[r]^{M} |@-{|} & C \ar[d]^{v} \\
            PE \ar[r]_{PN}|-@{|} & PD }$$

    \item The double cells of $\bb{X}$ are given by pairs of double cells of $\bb{B}$ and $\bb{E}$, satisfying an equation  {\bf ($*$)} as follows: 
    $$\vcenter{\xymatrix{
    (B,E,u) \ar@{=}[d] \ar[r]^{(M,N,\gamma)}|@-{|} 
    \ar@{}[rd]|{(\eta,\theta)}
    &  (C,D,v) \ar@{=}[d]  \\
    (B,E,u) \ar@<-.5ex>[r]_{(M',N',\gamma')}|@-{|}  &  (C,D,v) 
    }}
    \quad
    \vcenter{\xymatrix{
    B \ar@{=}[d] \ar[r]^{M}|@-{|} 
    \ar@{}[rd]|{\eta}
    &  C \ar@{=}[d]  \\
    B \ar@<-.5ex>[r]_{M'}|@-{|}  &  C
    }}
    \quad
    \vcenter{\xymatrix{E \ar@{=}[d] \ar[r]^{N}|@-{|} 
    \ar@{}[rd]|{\theta}
    &  D \ar@{=}[d]  \\
    E \ar@<-.5ex>[r]_{N'}|@-{|}  &  D
    }}
    \quad
    \vcenter{\xymatrix{
    B \ar@{=}[d] \ar[r]^{M}|@-{|} 
    \ar@{}[rd]|{\eta}
    &  C \ar@{=}[d]  \\
    \ar@{}[dr]|{\gamma'} B
    \ar[d]_{u}\ar[r]^{M'} |@-{|} & C \ar[d]^{v} \\
                PE \ar[r]_{PN'}|-@{|} & PD
    }}
    \stackrel{{\bf (*)}}{=}
    \vcenter{\xymatrix{\ar@{}[dr]|{\gamma } B \ar[d]_{u}\ar[r]^{M} |@-{|} & C \ar[d]^{v} \\
    PE \ar@{=}[d] \ar[r]^{PN}|@-{|} 
    \ar@{}[rd]|{P\theta}
    &  PD \ar@{=}[d]  \\
    PE \ar@<-.5ex>[r]_{PN'}|@-{|}  &  PD            
    }}$$

    \item The remaining double-category structure, that is the horizontal identities, the horizontal composition of arrows and double cells, the unitors and associators, are given in a natural way by those of $\bb{B}$, $\bb{E}$, {and $P$}. More explicitly:
    $$
    y^{\bb{X}}(B,E,u) := (y^{\bb{B}}(B),y^{\bb{E}}(E),\widetilde{y^{\bb{B}}(u)}), \quad
    (M,N,\gamma) \otimes_{\bb{X}} (\ol{M},\ol{N},\ol{\gamma}) := 
    (M \otimes_{\bb{B}} \ol{M},
    N \otimes_{\bb{E}} \ol{N},
    \widetilde{ \gamma \otimes_{\bb{B}} \ol{\gamma} }
    ),
    $$
    where the notation $\widetilde{ \quad }$ means {\em compose with a structural double cell of $P$}, explicitly as follows:
    $$
    \widetilde{y^{\bb{B}}(u)}
    :=
    \vcenter{\xymatrix{ \ar@{}[dr]|{y^{\bb{B}}(u) }
    B \ar[d]_{u} \ar[r]^{y^{\bb{B}}(B)}|@-{|} &  B \ar[d]^{u} \\
    \ar@{}[dr]|{\iota^P_E}
    PE \ar@{=}[d] \ar[r]^{y^{\bb{B}}(PE)}|@-{|} & PE \ar@{=}[d] \\
    PE \ar[r]_{Py^{\bb{E}}(E)}|@-{|} & PE
    }},
    \qquad
    \widetilde{ \gamma \otimes_{\bb{B}} \ol{\gamma} }
    :=
    \vcenter{\xymatrix{ 
      \ar@{}[dr]|{\gamma } B \ar[d]_{u} \ar[r]^{M} |@-{|} & C \ar[d]^{v} \ar@{}[dr]|{\ol{\gamma}} \ar[r]^{\ol{M}} |@-{|} & A \ar[d]^{w} 
      \\
        PE \ar@{=}[d] \ar@{}[drr]|{\phi^P_{N,\ol{N}}} \ar[r]_{PN}|-@{|} & PD \ar[r]_{P\ol{N}}|-@{|} & PF \ar@{=}[d] \\
        PE \ar[rr]_{P(N \otimes_{\bb{E}} \ol{N})}|@-{|} && PF
    }}
    $$
    
    Horizontal composition of double cells is defined as 
    $(\eta,\theta) \otimes_{\bb{X}} (\ol{\eta},\ol{\theta}) = 
    (\eta \otimes_{\bb{B}} \ol{\eta},\theta \otimes_{\bb{E}} \ol{\theta})$. We verify the {\bf (*)} equation for the composition by using first {\bf (*)} for $(\eta,\theta)$ and $(\ol{\eta},\ol{\theta})$, and then naturality of $\phi^P$, as follows:
    $$
    \vcenter{\xymatrix{ 
      \ar@{}[dr]|{\eta } B \ar@{=}[d] \ar[r]^{M} |@-{|} & C \ar@{=}[d] \ar@{}[dr]|{\ol{\eta}} \ar[r]^{\ol{M}} |@-{|} & A \ar@{=}[d] \\
      \ar@{}[dr]|{\gamma' } B \ar[d]_{u} \ar[r]^{M'} |@-{|} & C \ar[d]^{v} \ar@{}[dr]|{\ol{\gamma}'} \ar[r]^{\ol{M}'} |@-{|} & A \ar[d]^{w} 
      \\
        PE \ar@{=}[d] \ar@{}[drr]|{\phi^P_{N',\ol{N}'}} \ar[r]_{PN'}|-@{|} & PD \ar[r]_{P\ol{N}'}|-@{|} & PF \ar@{=}[d] \\
        PE \ar[rr]_{P(N' \otimes_{\bb{E}} \ol{N}')}|@-{|} && PF
    }}
    =
    \vcenter{\xymatrix{ 
    \ar@{}[dr]|{\gamma } B \ar[d]_{u}\ar[r]^{M} |@-{|} & C
    \ar@{}[dr]|{\ol{\gamma} }
    \ar[d]^{v} \ar[r]^{\ol{M}} |@-{|} & A \ar[d]^{w}
    \\
    PE \ar@{=}[d] \ar[r]^{PN}|@-{|} 
    \ar@{}[rd]|{P\theta}
    &  PD \ar@{=}[d]  
    \ar[r]^{P\ol{N}}|@-{|} 
    \ar@{}[rd]|{P\ol{\theta}}
    &  PF \ar@{=}[d]  
    \\
    PE \ar@{=}[d] \ar@{}[drr]|{\phi^P_{N',\ol{N}'}} \ar[r]_{PN'}|-@{|} & PD \ar[r]_{P\ol{N}'}|-@{|} & PF \ar@{=}[d] \\
        PE \ar[rr]_{P(N' \otimes_{\bb{E}} \ol{N}')}|@-{|} && PF
    }}
    =
    \vcenter{\xymatrix{ 
    \ar@{}[dr]|{\gamma } B \ar[d]_{u}\ar[r]^{M} |@-{|} & C
    \ar@{}[dr]|{\ol{\gamma} }
    \ar[d]^{v} \ar[r]^{\ol{M}} |@-{|} & A \ar[d]^{w}
    \\
    PE \ar@{=}[d] \ar[r]^{PN}|@-{|} 
    \ar@{}[rrd]|{\phi^P_{N,\ol{N}}}
    &  PD \ar[r]^{P\ol{N}}|@-{|} 
    & PF \ar@{=}[d]  
    \\
    PE \ar@{=}[d] \ar@{}[drr]|{P(\theta \otimes_{\bb{E}} \ol{\theta}) }
    \ar[rr]_{P(N \otimes_{\bb{E}} \ol{N})}|@-{|} && PF  \ar@{=}[d]
    \\
        PE \ar[rr]_{P(N' \otimes_{\bb{E}} \ol{N}')}|@-{|} && PF
    }}
    $$
    
    Finally, the unitors and associators of $\bb{X}$ are given by those of $\bb{B}$ and $\bb{E}$. More explicitly, the left unitor is given, for each horizontal arrow $(M,N,\gamma)$ as above, by a double cell
    $$
    \vcenter{\xymatrix{
    (B,E,u) \ar@{=}[d] \ar[r]|@-{|}
    \ar@<.75ex>@{}[r]^{(y^{\bb{B}}(B),y^{\bb{E}}(E),\widetilde{y^{\bb{B}}(u)})} &
    (B,E,u) \ar[r]^{(M,N,\gamma)}|@-{|} 
    \ar@{}[d]|{(\ell^{\bb{B}}_{M},\ell^{\bb{E}}_{N})}
    &  (C,D,v) \ar@{=}[d]  \\
    (B,E,u) \ar@<-.5ex>[rr]_{(M,N,\gamma)}|@-{|}  && (C,D,v) 
    }}
    =
    \vcenter{\xymatrix{
    (B,E,u) \ar@{=}[d] \ar[rrr]|@-{|} 
    \ar@<.85ex>@{}[rrr]^-{(y^{\bb{B}}(B) \otimes_{\bb{B}}  M, y^{\bb{E}}(E) \otimes_{\bb{E}} N, \widetilde{\widetilde{y^{\bb{B}}(u)} \otimes_{\bb{B}} \gamma}
    )}
    \ar@{}[rrrd]|{(\ell^{\bb{B}}_{M},\ell^{\bb{E}}_{N})}
    &&&  (C,D,v) \ar@{=}[d]  \\
    (B,E,u) \ar@<-.5ex>[rrr]_{(M,N,\gamma)}|@-{|}  &&&  (C,D,v) 
    }}
    $$
    We verify the {\bf (*)} equation for this double cell below. Note that the double cell $\widetilde{\widetilde{y^{\bb{B}}(u)} \otimes_{\bb{B}} \gamma}$ corresponds by definition to the composition on the right if we suppress $P(\ell^{\bb{E}}_{N})$.
    $$
    \vcenter{\xymatrix{
    B \ar@{=}[d] \ar[r]^{y^{\bb{B}}(B)}|@-{|} 
    \ar@{}[rrd]|{\ell^{\bb{B}}_{M}}
    & B \ar[r]^{M}|@-{|} 
    &  C \ar@{=}[d]  \\
    \ar@{}[drr]|{\gamma} B
    \ar[d]_{u}\ar[rr]^{M} |@-{|} && C \ar[d]^{v} \\
                PE \ar[rr]_{PN}|-@{|} && PD
    }}
    =
    \vcenter{\xymatrix{
    \ar@{}[dr]|{y^{\bb{B}}(u) }
    B \ar[d]_{u} \ar[r]^{y^{\bb{B}}(B)}|@-{|} &  B \ar[d]^{u} \ar@{}[dr]|{\gamma }\ar[r]^{M} |@-{|} & C \ar[d]^{v} \\
    PE \ar@{}[rrd]|{\ell^{\bb{B}}_{M}}
    \ar@{=}[d] \ar[r]_{y^{\bb{B}}(PE)}|@-{|} & PE 
    \ar[r]_{PN}|-@{|} & PD  \ar@{=}[d]  \\
    PE \ar[rr]_{PN}|-@{|} && PD
    }}
    =
    \vcenter{\xymatrix{
    \ar@{}[dr]|{y^{\bb{B}}(u) }
    B \ar[d]_{u} \ar[r]^{y^{\bb{B}}(B)}|@-{|} &  B \ar[d]^{u} \ar@{}[dr]|{\gamma }\ar[r]^{M} |@-{|} & C \ar[d]^{v} \\
    \ar@{}[dr]|{\iota^P_E}
    PE \ar@{=}[d] \ar[r]^{y^{\bb{B}}(PE)}|@-{|} & PE \ar@{=}[d] \ar[r]_{PN}|-@{|} & PD  \ar@{=}[d]  \\
    PE \ar@{}[rrd]|{\phi^P_{y^{\bb{E}}(E),N}}
    \ar@{=}[d] \ar[r]_{Py^{\bb{E}}(E)}|@-{|} & PE  \ar[r]_{PN}|-@{|} & PD  \ar@{=}[d]  \\
    PE \ar@{}[rrd]|{P (\ell^{\bb{E}}_{N}) }
    \ar@{=}[d] \ar[rr]_{P(y^{\bb{E}}_E \otimes_{\bb{E}} N)} && PD \ar@{=}[d] \\
    PE \ar[rr]_{PN}|-@{|} && PD
    }}
    $$
    The first equation is the naturality of $\ell^{\bb{B}}$ (at $\gamma$). The second one follows from the unit coherence for the  pseudofunctor $P$ (item 2 in Definition \ref{def:laxfunctor}). The right unitor is dual and omitted. For the associator, given composable horizontal arrows
    $$(B,E,u) \mrh{(M,N,\gamma)}  (C,D,v)
    \mrh{(\ol{M},\ol{N},\ol{\gamma})}  (A,F,w)
    \mrh{(\dol{M},\dol{N},\dol{\gamma})}  (Z,G,x)
    $$
    we have the double cell
    $$
    \vcenter{\xymatrix{
    (B,E,u) \ar@{=}[d] \ar[rr]^{(M \otimes_{\bb{B}} \ol{M},
    N \otimes_{\bb{E}} \ol{N},
    \widetilde{ \gamma \otimes_{\bb{B}} \ol{\gamma} }
    )}|@-{|} 
    \ar@{}[drrr]|{({\mathfrak a}^{\bb{B}}_{M,\ol{M},\dol{M}},{\mathfrak a}^{\bb{E}}_{N,\ol{N},\dol{N}})}
    &&   (A,F,w)
    \ar[r]^{(\dol{M},\dol{N},\dol{\gamma})}|@-{|} &  (Z,G,x) \ar@{=}[d]  \\
    (B,E,u) \ar[r]_{(M,N,\gamma)}|@-{|} 
    &  (C,D,v) \ar[rr]_{(\ol{M} \otimes_{\bb{B}} \dol{M},
    \ol{N} \otimes_{\bb{E}} \dol{N},
    \widetilde{ \ol{\gamma} \otimes_{\bb{B}} \dol{\gamma} }
    )}|@-{|} &&
     (Z,G,x)
    }}
    $$
    The {\bf (*)} equation follows, similarly to the case of the unitor, by naturality of ${\mathfrak a}^{\bb{B}}$ and coherence for $P$ (item 1 in Definition \ref{def:laxfunctor}). This finishes the definition of the (pseudo) double category $\bb{X}$.
\end{itemize}

The following formulas define two strict double functors $E\colon \bb{X} \mr{} \bb{E}$, $B\colon \bb{X} \mr{} \bb{B}$ and a vertical transformation $\beta\colon B \Mr{} PE$ as in \eqref{eq:betainfibrationindblcatv}:

$$
E(B,E,u) = E, \qquad 
B(B,E,u) = B, \qquad 
\beta_{(B,E,u)} = u
$$

$$
E(M,N,\gamma) = N, \qquad 
B(M,N,\gamma) = M, \qquad 
\beta_{(M,N,\gamma)} = \gamma
$$

Note that the axioms in \eqref{eq:vert_nat_satisfies} for $\beta$ being a vertical transformation are (for each $(B,E,u)$, and for each $(B,E,u) \mrh{(M,N,\gamma)}  (C,D,v)
\mrh{(\ol{M},\ol{N},\ol{\gamma})}  (A,F,w)$)

\begin{equation}
\vcenter{\xymatrix{
y^{\bb{B}}(B) \ar[r]^{y^{\bb{B}}(u)} 
\ar[d]_{id} 
& y^{\bb{B}}(PE) \ar[d]^{\iota^P_E} 
\\
y^{\bb{B}}(B) \ar[r]_{\widetilde{y^{\bb{B}}(u)}} & Py^{\bb{E}}(E) \ 
}}
\quad
{\vcenter{\xymatrix@C=4pc{
M \otimes_{\bb{B}} \ol{M} \ar[r]^-{\gamma \otimes_{\bb{B}} \ol{\gamma}} \ar[d]_{id} &  PN \otimes_{bb{B}} P\ol{N}
\ar[d]^{\phi^P_{N,\ol{N}}}
\\
M \otimes_{\bb{B}} \ol{M} \ar[r]_-{\widetilde{ \gamma \otimes_{\bb{B}} \ol{\gamma} }} & P(N \otimes_{\bb{E}} \ol{N})
}}}
\mbox{ commute in } \ct{B}_1
\end{equation}

This follows immediately from the definitions of $\widetilde{y^{\bb{B}}(u)}$ and $\widetilde{ \gamma \otimes_{\bb{B}} \ol{\gamma}}$ above.
\smallskip

If $P$ is a fibration  in $\Dblcatvl$, then there exist a lax double functor $E': \mathbb{X} \to \mathbb{E}$ and a vertical transformation $\alpha: E' \Rightarrow E$.
In $\mathbf{(L)}$ below we use part of these data to show items 1 and 2 in the statement of the theorem.
Note that if $P$ is a  fibration 
 in $\Dblcatv$, resp. $\Dblcatvs$, then in addition $E'$ can be taken to be a pseudo, resp. strict double functor. In $\mathbf{(P)}$, resp. $\mathbf{(S)}$ below, we use the structural cells of $E'$ to show that in this case item 3, resp. 3$s$  also holds.

\smallskip
\noindent $\mathbf{(L)}$ We use the values of $E'$ and $\alpha$ on objects and arrows of $\bb{X}$ (on the left in the equations below) to define liftings of arrows $u\colon B \mr{} PE$ and $\gamma\colon M \mr{} PN$ as follows:

\begin{equation} \label{eq:XX1}
E'(B,E,u) =: B^* , \qquad \alpha_{(B,E,u)} =: ( u^*E: B^* \to E)
\end{equation}

\begin{equation} \label{eq:XX2}
E'(M,N,\gamma) =: M^*, \qquad \alpha_{(M,N,\gamma)} =: (\gamma^*N: M^* \to N)
\end{equation}

The definitions in \eqref{eq:XX1}, \eqref{eq:XX2} can also be written as the diagram

$$
\vcenter{\xymatrix@C=3pc{
         \ar@{}[dr]|{\alpha_{(M,N,\gamma)}} E'(B,E,u) \ar[d]_{\alpha_{(B,E,u)}} \ar[r]^{E'(M,N,\gamma)} |@-{|} & E'(C,D,v) \ar[d]^{\alpha_{(C,D,v)}} \\
            E(B,E,u) \ar[r]_{E(M,N,\gamma)} |@-{|} & E(C,D,v)
}}
\qquad \qquad =: \qquad \qquad
\vcenter{\xymatrix{
         \ar@{}[dr]|{\gamma^*N } B^* \ar[d]_{u^*E}\ar[r]^{M^*} |@-{|} & C^* \ar[d]^{v^*D} \\
            E \ar[r]_{N}|-@{|} & D 
}}
$$

Noting that (by item 1 in Lemma \ref{lem:alpha_X_and_alpha_M_are_Cartesian}) $u^*E$ is a $P_0$-cartesian arrow, the choice $(B,E,u) \leadsto (B^*,u^*E)$ constitutes a cleavage for $P_0$. Analogously, by item 2 in  Lemma \ref{lem:alpha_X_and_alpha_M_are_Cartesian}, the choice $(M,N,\gamma) \leadsto (M^*,\gamma^*N)$ constitutes a cleavage for $P_1$, and by construction $\src_\bb{E},\tgt_\bb{E}: \ct{E}_1 \to \ct{E}_0$ are cleavage preserving.

\smallskip
\noindent $\mathbf{(P)}$ In this case, we have the structural double cells in $\bb{E}$ of the pseudo double functor $E'$, that we write as arrows of $\ct{E}_1$: $\iota^{E'}_-$, one for each ${(B,E,u)}$ in $\bb{X}$, {(resp. 
$\phi^{E'}_{-,-}$ one for each $\stackrel{(M,N,\gamma)}{\hto} \stackrel{(\ol{M},\ol{N},\ol{\gamma})}{\hto}$)}, where the subindices are omitted for readability, 
and the vertical natural transformation axioms in \eqref{eq:vert_nat_satisfies} are satisfied:
$$
\vcenter{\xymatrix@C=4pc{
y({B^*}) \ar[r]^{y({u^*E})} 
\ar[d]_{\iota^{E'}_{(B,E,u)}} 
& y(E) \ar[d]^{id} \\
E'(y{(B,E,u)}) \ar[r]_-{\alpha_{y{(B,E,u)}}} & y(E)
}}
\qquad
{\vcenter{\xymatrix@C=4pc{
M^* \otimes_{\bb{E}} \ol{M}^* \ar[r]^-{\gamma^* N  \otimes_{\bb{E}} \ol{\gamma}^* \ol{N} } \ar[d]_{\phi^{E'}_{(M,N,\gamma),(\ol{M},\ol{N},\ol{\gamma})}} &  N \otimes_{\bb{E}} \ol{N}
\ar[d]^{id}
\\
E'((M,N,\gamma) \otimes_{\bb{X}} (\ol{M},\ol{N},\ol{\gamma})) \ar[r]
\ar@{}@<-.5ex>[r]_-{\alpha_{(M,N,\gamma) \otimes_{\bb{X}} (\ol{M},\ol{N},\ol{\gamma})}} & N \otimes_{\bb{E}} \ol{N}
}}}
\mbox{ commute in } \ct{E}_1
$$

This shows that $\ct{E}_0 \mr{y} \ct{E}_1$ {(resp. $\ct{E}_1 \times_{\ct{E}_0} \ct{E}_1 \mr{\otimes} \ct{E}_1$)} is Cartesian-morphism preserving: since it suffices to check that the chosen Cartesian arrows are preserved, let $(B,E,u)$ as above, then the diagram on the left is showing that $y({u^* E})$ is Cartesian, since it is the composition of an isomorphism with a (chosen) Cartesian arrow. The case of $\otimes$ is similar, using the diagram on the right and recalling that a cleavage can be chosen pointwise (see Proposition  \ref{prop:pullbackspointwise}).

\smallskip
\noindent $\mathbf{(S)}$
Note that if $E'$ is strict, then the same diagrams show that $y$ and $\otimes$ preserve the cleavage.

\medskip 
\noindent
$\Leftarrow)$ We consider the case $\mathbf{(L)}$ in which all double functors  are taken lax, and we show, in the paragraphs denoted  $\mathbf{(P)}$ and $\mathbf{(S)}$, how the same computations show those cases too. 
Let $P: \mathbb{E} \to \mathbb{B}$ satisfy items 1 and 2 in the theorem, and choose cleavages 
$$\{u^*E: B^* \to E \}_{ B \mr{u} PE } \mbox{ of } P_0, \qquad    \{\gamma^*N: M^* \to N \}_{ M \mr{\gamma} PN } \mbox{ of } P_1 $$
such that $\src_\bb{E}$ and $\tgt_\bb{E}$ are cleavage preserving. Let $\beta: B \Mr{} PE$ as in \eqref{eq:betainfibrationindblcatv}, we need to construct $E'$ and $\alpha$ satisfying the conditions below \eqref{eq:betainfibrationindblcatv}. We use for this the chosen cleavages, and define (for each object $X$, resp. each horizontal arrow $M: X \hto Y$ of $\bb{X}$)
$E'(X) := B(X)^*$,
$\alpha_X := \beta_X^* E(X)$,
$E'(M) := B(M)^*$,
$\alpha_M := \beta_M^* E(M)$:

\[\vcenter{\xymatrix{B(X) \ar[d]_{\beta_X} \\ PE(X)}} 
\quad \leadsto \quad
\vcenter{\xymatrix{E'(X) \ar[d]_{\alpha_X} \\ E(X)}},
\qquad\qquad
\vcenter{\xymatrix{
            \ar@{}[dr]|{\beta_M} B(X) \ar[d]_{\beta_X}\ar[r]^{B(M)} |@-{|} & B(Y) \ar[d]^{\beta_Y} \\
            PE(X) \ar[r]_{PE(M)}|-@{|} & PE(Y)
        }}
\quad \leadsto \quad
\vcenter{
\xymatrix{
            \ar@{}[dr]|{\alpha_M} E'(X) \ar[d]_{\alpha_X}\ar[r]^{E'(M)} |@-{|} & E'(Y) \ar[d]^{\alpha_Y} \\
            E(X) \ar[r]_{E(M)}|-@{|} & E(Y)
        }
}
\]

The values of $E'$ at the arrows $X \mr{} X'$ and the double cells $M \mr{} M'$ are uniquely defined, using the Cartesian properties of the arrows in the cleavage, in such a way that the families $\alpha_X$ and $\alpha_M$ are natural.

The lax double functor structural cells for $E'$ (such that \eqref{eq:vert_nat_satisfies} holds, and thus making $\alpha$ a vertical transformation), can be obtained using the facts that $\alpha_{y(X)}$ (resp. ${\alpha_{M \otimes N}}$) is $P_1$-Cartesian (it is by construction a chosen lift), 
and that \eqref{eq:vert_nat_satisfies} holds in $\ct{B}_1$:

$$
\vcenter{\xymatrix{
y{E'(X)} \ar[r]^{y({\alpha_X})} 
\ar@{.>}[d]_{\iota^{E'}_X} 
& y{PE(X)} \ar[d]^{\iota^E_X} \\
E'y(X) \ar[r]_{\alpha_{y(X)}} & Ey(X) \
}} 
\quad
{\vcenter{\xymatrix@C=4pc{
E'(M) \otimes E'(N) \ar[r]^-{\alpha_M \otimes \alpha_N} \ar@{.>}[d]_{\phi^{E'}_{M,N}} &  E(M) \otimes E(N)
\ar[d]^{\phi^E_{M,N}} \\
E'(M \otimes N) \ar[r]_-{\alpha_{M \otimes N}} & E(M \otimes N) 
}}} 
\quad
\mbox{ in } \ct{E}_1
$$

$$
\vcenter{\xymatrix{
y{B(X)} \ar[r]^{y({\beta_X})} 
\ar[d]_{\iota^B_X} 
& y_{PE(X)} \ar[d]^{P(\iota^E_X)} 
\\
By(X) \ar[r]_{\beta_{y(X)}} & PEy(X) \ 
}}
\quad
{\vcenter{\xymatrix@C=4pc{
B(M) \otimes B(N) \ar[r]^-{\beta_M \otimes \beta_N} \ar[d]_{\phi^B_{M,N}} &  PE(M) \otimes PE(N)
\ar[d]^{P(\phi^E_{M,N})}
\\
B(M \otimes N) \ar[r]_-{\beta_{M \otimes N}} & PE(M \otimes N)
}}}
\mbox{ in } \ct{B}_1
$$

Coherence for $\iota$ and $\otimes$ is left to the reader.

\smallskip
\noindent $\mathbf{(P)}$ We assume item 3 in the statement of the theorem, and we show that if $B$ and $E$ are pseudo double functors, then so is $E'$. Indeed, since in this case $\iota^B_X$ and $\iota^E_X$ are invertible and $y({\alpha_X})$ is Cartesian, then the diagrams on the left above imply as usual that $\iota^{E'}_X$ is invertible {(and similarly so is $\phi^{E'}_{M,N}$)}. 

\smallskip
\noindent $\mathbf{(S)}$ If now we assume item 3$s$ in the statement of the theorem and $B$ and $E$ to be strict, then the composition $\iota^E_X \circ y({\alpha_X})$ equals $\alpha_{y(X)}$, and the diagrams on the left above imply as usual that $\iota^{E'}_X$ is an identity {(and similarly so is $\phi^{E'}_{M,N}$)}. 

\smallskip

We can then finish the proof by showing that $E'$ and $\alpha$ satisfy the conditions below \eqref{eq:betainfibrationindblcatv}. Let $X$, $E''$, $\xi$, and $\gamma$ be lax double functors and vertical transformations as in \eqref{eq:alpha_satisfies_1}. The components $\zeta_Y$, resp. $\zeta_N$ (indexed by objects, resp. horizontal arrows of $\bb{Y}$) of the unique vertical transformation $\zeta$ such that the equations in \eqref{eq:alpha_satisfies_2} hold are constructed using the fact that each $\alpha_X$ (resp. $\alpha_M$) is Cartesian (for the fibration $P_0$, resp. $P_1$):

$$
\vcenter{\xymatrix{
E''(Y) \ar@{.>}[d]_{\zeta_Y} \ar[dr]^{\xi_Y} &&& E''(N) \ar@{.>}[d]_{\zeta_N} \ar[dr]^{\xi_N} \\
E'X(Y) \ar[r]_{\alpha_{X(Y)}} & EX(Y) &&  E'X(N) \ar[r]_{\alpha_{X(N)}} & EX(N) \\
PE''(Y) \ar[d]_{\gamma_Y} \ar[dr]^{P\xi_Y} &&& 
PE''(N) \ar[d]_{\gamma_N} \ar[dr]^{P\xi_N}
\\
PE'X(Y) \ar[r]_{P\alpha_{X(Y)}} & PEX(Y) && PE'X(N) \ar[r]_{P\alpha_{X(N)}} & PEX(N) \\
}}
$$

The two equations expressing the facts that $\zeta_Y$ and $\zeta_N$ are natural follow by unicity of the liftings in each of the two diagrams above, using the naturalities of the other vertical transformations involved in these diagrams. It only remains to check the vertical transformation axioms for $\zeta$, as in \eqref{eq:vert_nat_satisfies}:
\begin{equation} \label{eq:XXX}
\vcenter{\xymatrix{
y{E''(Y)} \ar[r]^{y({\zeta_Y})} 
\ar[d]_{\iota^{E''}_Y} 
& y{E'X(Y)} \ar[d]^{\iota^{E'X}_Y} 
\\
E''y(Y) \ar[r]_{\zeta_{y(Y)}} & {E'Xy(Y)} \ 
}}
\quad
{\vcenter{\xymatrix@C=4pc{
E''(N) \otimes E''(N') \ar[r]^-{\zeta_N \otimes \zeta_{N'}} \ar[d]_{\phi_{M,N}^{E''}} &  E'X(N) \otimes E'X(N')
\ar[d]^{\phi^{E'X}_{M,N}}
\\
E''(N \otimes N') \ar[r]_-{\zeta_{N \otimes N'}} & E'X(N \otimes N')
}}}
\mbox{ commutes in } \ct{E}_1
\end{equation}

As usual, to check that two arrows are equal in the top category of a fibration $P$, it suffices to check that they are equal both {\bf (i)} after applying $P$ and {\bf (ii)} after composing with a Cartesian arrow. 
Noting that the family $\zeta$ is defined to be over $\gamma$, applying $P$ to the diagrams in \eqref{eq:XXX} leads to diagrams that, 
since 
$P$ is pseudo, are commutative if and only if  \eqref{eq:vert_nat_satisfies} holds for $\gamma$. 
This shows {\bf (i)}, and the  following commutative diagrams show {\bf (ii)}:
$$
\vcenter{\xymatrix{y{E''(Y)} \ar[rd]^{\iota^{E''}_Y} \ar[d]_{y({\zeta_Y})}   &&& E''(N) \otimes E''(N') \ar[d]_-{\zeta_N \otimes \zeta_{N'}}  \ar[dr]^{\phi_{M,N}^{E''}} \\
y{E'X(Y)} \ar[d]_{\iota^{E'X}_Y}  &   
E''y(Y) \ar[dl]_{\zeta_{y(Y)}}  \ar[d]^{\xi_{y(Y)}}
&& E'X(N) \otimes E'X(N') \ar[d]_{\phi^{E'X}_{M,N}}  & E''(N \otimes N') \ar[d]^{\xi_{N\otimes N'}} \ar[dl]_{\zeta_{N \otimes N'}} \\
E'Xy(Y) \ar[r]_{\alpha_{Xy(Y)}} & EXy(Y) &&  E''(N \otimes N') \ar[r]_{\alpha_{X(N\otimes N')}} & E'X(N \otimes N')
}}
$$
\end{proof}

We have finished the proof of Theorem \ref{th:charact_internal_fibrations}. As mentioned in Corollary \ref{coro:double_fibration_iff_internal}, if follows that a strict double functor $P$ defines a double fibration, that is a pseudo category in the 2-category $\fib$ of fibrations, if and only if it is an internal fibration in the 2-category $\Dblcatv$ of double categories and pseudo double functors.

\begin{remark} \label{rem:Pseudo_is_needed}
Note that for Remark \ref{rem:Fstrict_then_alpha_determined} to be applied to $\gamma$ in the proof of Lemma \ref{lem:alpha_X_and_alpha_M_are_Cartesian} above, we use the condition that $P$ is a pseudo double functor. This is one of the reasons why we cannot consider $P$ to be lax in Theorem \ref{th:charact_internal_fibrations}: Lemma \ref{lem:alpha_X_and_alpha_M_are_Cartesian} is applied to prove the $\Rightarrow$ implication in the theorem. Note that the condition that $P$ is a pseudo double functor is also used for proving the $\Leftarrow$ implication (more precisely in the final paragraph of the proof). 
\end{remark}

\begin{appendix}

\section{Pullbacks in 2-categories of Fibrations} \label{sec:proof_in_Appendix}

In this Section we will prove Proposition \ref{prop:pullbackspointwise}, that we recall here for convenience. 
Though the first item of this proposition is straightforward and could be considered {\em folklore}, we include a sketch of its proof because we believe it can help to understand the proof of the second item, which follows a similar {pattern}.

\begin{prop} \label{prop:pullbackspointwise_in_App}
\begin{enumerate}
    \item 2-pullbacks are computed pointwise in $\Arrs(\K)$ and preserved by the inclusion 2-functors $\Arrs(\K) \mr{} \Arrp(\K)$ and $\Arrs(\K) \mr{} \Arrl(\K)$.

    \item 2-pullbacks are computed pointwise in $\cfib$ and preserved by both the inclusion  2-functor $\cfib \mr{} \fib$ and the forgetful 2-functor $\cfib \mr{} \Arrs(\Cat)$.
\end{enumerate}
\end{prop}

\begin{proof}
To show item 1, one considers a diagram  $s: F_1 \rightarrow F_0 \leftarrow F_2:t$  in $Arr^s(\mathfrak K)$, given by the commutative diagram on the left:

\begin{equation} \label{eq:cubeofarrows}
\begin{tikzpicture}[scale=2.5,tdplot_main_coords,baseline=(current bounding box.center)]

\draw[shorten >=0.3cm,shorten <=.3cm,<-] (0,0,0) node{${D}_1$} -- node[left]{$F_1$} (0,0,1) node{${C}_1$} ;
\draw[shorten >=0.3cm,shorten <=.3cm,<-] (1,1,0) node{${D}_2$} -- node[right]{$F_2$} (1,1,1) node{${C}_2$} ;
\draw[shorten >=0.3cm,shorten <=.3cm,<-] (1,0,0) node{${D}_0$} -- node[left]{$F_0$} (1,0,1) node{${C}_0$} ;
\draw[shorten >=0.3cm,shorten <=.3cm,->] (0,0,1) -- node[above] {$t^\top$} (1,0,1);
\draw[shorten >=0.3cm,shorten <=.3cm,<-] (1,0,1) -- node[above] {$s^\top$} (1,1,1);
\draw[shorten >=0.3cm,shorten <=.3cm,->](0,0,0) --  node[below] {$t^\bot$}  (1,0,0) ;
\draw[shorten >=0.3cm,shorten <=.3cm,<-] (1,0,0) -- node[below] {$s^\bot$} (1,1,0);
\end{tikzpicture}
\qquad \leadsto \qquad
\begin{tikzpicture}[scale=3.5,tdplot_main_coords,baseline=(current bounding box.center)]
\draw[shorten >=0.3cm,shorten <=.3cm,<-] (0,0,0) node{${D}_1$} -- node[left]{$F_1$} (0,0,1) node{${C}_1$} ;
\draw[shorten >=0.3cm,shorten <=.3cm,<-] (1,1,0) node{${D}_2$} -- node[right]{$F_2$} (1,1,1) node{${C}_2$} ;
\draw[dashed, shorten >=0.3cm,shorten <=.3cm,<-] (0,1,0) node{${D}_1 \times_{{D}_0} {D}_2$} -- node{$F_1 \times_{F_0} F_2$} (0,1,1) node{${C}_1 \times_{{C}_0} {C}_2$};
\draw[shorten >=0.3cm,shorten <=.3cm,<-] (1,0,0) node{${D}_0$} -- node[left]{$F_0$} (1,0,1) node{${C}_0$} ;
\draw[shorten >=0.3cm,shorten <=.3cm,->] (0,0,1) -- node[below] {$t^\top$} (1,0,1);
\draw[shorten >=0.3cm,shorten <=.3cm,<-] (0,0,1) -- node[above] {} (0,1,1);
\draw[shorten >=0.3cm,shorten <=.3cm,->] (0,1,1) -- node[above] {}(1,1,1);
\draw[shorten >=0.3cm,shorten <=.3cm,<-] (1,0,1) -- node[below] {$s^\top$} (1,1,1);
\draw[shorten >=0.3cm,shorten <=.3cm,->](0,0,0) --  node[below] {$t^\bot$}  (1,0,0) ;
\draw[shorten >=0.3cm,shorten <=.3cm,<-] (0,0,0) -- node[above] {} (0,1,0);
\draw[shorten >=0.3cm,shorten <=.3cm,->] (0,1,0) -- node[above] {} (1,1,0);
\draw[shorten >=0.3cm,shorten <=.3cm,<-] (1,0,0) -- node[below] {$s^\bot$} (1,1,0);
\end{tikzpicture}
\end{equation}

In the diagram on the right we consider the 2-pullbacks in $\mathfrak K$, and $F_1 \times_{F_0} F_2$ is the functor induced by the universal property of ${D}_1 \times_{{D}_0} {D}_2$. 
Item 1 of the lemma then states that $F_1 \times_{F_0} F_2$ (together with the pullback projections) is the pullback of the diagram $F_1 \rightarrow F_0 \leftarrow F_2$ in the three 2-categories $\Arrs(\K)$,  $\Arrp(\K)$, and $\Arrl(\K)$.
We verify this using the diagram on the left below as follows:

\begin{equation} \label{eq:twouglydiagrams}
\begin{tikzpicture}[scale=3,tdplot_main_coords,baseline=(current bounding box.center)]]
\draw[shorten >=0.3cm,shorten <=.3cm,<-] (0,.5,2) node{${D_3}$} -- node[left]{$F_3$} (0,.5,3) node{$C_3$} ;
\draw[shorten >=0.3cm,shorten <=.3cm,->] (0,.5,3) to [bend right] node[left]{$u_1^\top$} (0,0,1);
\draw[shorten >=0.3cm,shorten <=.3cm,->] (0,.5,2) to [bend right] node[left]{$u_1^\bot$} (0,0,0);
\draw[shorten >=0.3cm,shorten <=.3cm,->] (0,.5,3) to [bend left] node[right]{$u_2^\top$} (1,1,1);
\draw[shorten >=0.3cm,shorten <=.3cm,->] (0,.5,2) to [bend left] node[right]{$u_2^\bot$} (1,1,0);
\draw[dashed, shorten >=0.3cm,shorten <=.3cm,->] (0,.5,3) to node[right]{$u^\top$} (0,1,1);
\draw[dashed, shorten >=0.3cm,shorten <=.3cm,->] (0,.5,2) to (0,1,0);
\draw[shorten >=0.3cm,shorten <=.3cm,<-] (0,0,0) node{${D}_1$} -- node[left]{$F_1$} (0,0,1) node{${C}_1$} ;
\draw[shorten >=0.3cm,shorten <=.3cm,<-] (1,1,0) node{${D}_2$} -- node[right]{$F_2$} (1,1,1) node{${C}_2$} ;
\draw[shorten >=0.3cm,shorten <=.3cm,<-] (0,1,0) node{${D}_1 \times_{{D}_0} {D}_2$} -- node{$F_1 \times_{F_0} F_2$} (0,1,1) node{${C}_1 \times_{{C}_0} {C}_2$};
\draw[shorten >=0.3cm,shorten <=.3cm,<-] (1,0,0) node{${D}_0$} -- node[left]{$F_0$} (1,0,1) node{${C}_0$} ;
\draw[shorten >=0.3cm,shorten <=.3cm,->] (0,0,1) -- node[below] {$t^\top$} (1,0,1);
\draw[shorten >=0.3cm,shorten <=.3cm,<-] (0,0,1) -- node[above] {} (0,1,1);
\draw[shorten >=0.3cm,shorten <=.3cm,->] (0,1,1) -- node[above] {}(1,1,1);
\draw[shorten >=0.3cm,shorten <=.3cm,<-] (1,0,1) -- node[below] {$s^\top$} (1,1,1);
\draw[shorten >=0.3cm,shorten <=.3cm,->](0,0,0) --  node[below] {$t^\bot$}  (1,0,0) ;
\draw[shorten >=0.3cm,shorten <=.3cm,<-] (0,0,0) -- node[above] {} (0,1,0);
\draw[shorten >=0.3cm,shorten <=.3cm,->] (0,1,0) -- node[above] {} (1,1,0);
\draw[shorten >=0.3cm,shorten <=.3cm,<-] (1,0,0) -- node[below] {$s^\bot$} (1,1,0);
\end{tikzpicture}    
\qquad\qquad
\begin{tikzpicture}[scale=3,tdplot_main_coords,baseline=(current bounding box.center)]]
\draw[shorten >=0.3cm,shorten <=.3cm,<-] (0,.5,2) node{$\ct{B}_3$} -- node[left]{$P_3$} (0,.5,3) node{$\ct{E}_3$} ;
\draw[shorten >=0.3cm,shorten <=.3cm,->] (0,.5,3) to [bend right] node[left]{$u_1^{{\top}}$} (0,0,1);
\draw[shorten >=0.3cm,shorten <=.3cm,->] (0,.5,2) to [bend right] node[left]{$u_1^{{\bot}}$} (0,0,0);
\draw[shorten >=0.3cm,shorten <=.3cm,->] (0,.5,3) to [bend left] node[right]{$u_2^{{\top}}$} (1,1,1);
\draw[shorten >=0.3cm,shorten <=.3cm,->] (0,.5,2) to [bend left] node[right]{$u_2^{{\bot}}$} (1,1,0);
\draw[dashed, shorten >=0.3cm,shorten <=.3cm,->] (0,.5,3) to node[right]{$u^{{\top}}$} (0,1,1);
\draw[dashed, shorten >=0.3cm,shorten <=.3cm,->] (0,.5,2) to (0,1,0);
\draw[shorten >=0.3cm,shorten <=.3cm,<-] (0,0,0) node{${\ct{B}}_1$} -- node[left]{$P_1$} (0,0,1) node{${\ct{E}}_1$} ;
\draw[shorten >=0.3cm,shorten <=.3cm,<-] (1,1,0) node{${\ct{B}}_2$} -- node[right]{$P_2$} (1,1,1) node{${\ct{E}}_2$} ;
\draw[dashed, shorten >=0.3cm,shorten <=.3cm,<-] (0,1,0) node{${\ct{B}}_1 \times_{{\ct{B}}_0} {\ct{B}}_2$} -- node{$P_1 \times_{P_0} P_2$} (0,1,1) node{${\ct{E}}_1 \times_{{\ct{E}}_0} {\ct{E}}_2$};
\draw[shorten >=0.3cm,shorten <=.3cm,<-] (1,0,0) node{${\ct{B}}_0$} -- node[left]{$P_0$} (1,0,1) node{${\ct{E}}_0$} ;
\draw[shorten >=0.3cm,shorten <=.3cm,->] (0,0,1) -- node[below] {$t^{{\top}}$} (1,0,1);
\draw[shorten >=0.3cm,shorten <=.3cm,<-] (0,0,1) -- node[above] {} (0,1,1);
\draw[shorten >=0.3cm,shorten <=.3cm,->] (0,1,1) -- node[above] {}(1,1,1);
\draw[shorten >=0.3cm,shorten <=.3cm,<-] (1,0,1) -- node[below] {$s^{{\top}}$} (1,1,1);
\draw[shorten >=0.3cm,shorten <=.3cm,->](0,0,0) --  node[below] {$t^{{\bot}}$}  (1,0,0) ;
\draw[shorten >=0.3cm,shorten <=.3cm,<-] (0,0,0) -- node[above] {} (0,1,0);
\draw[shorten >=0.3cm,shorten <=.3cm,->] (0,1,0) -- node[above] {} (1,1,0);
\draw[shorten >=0.3cm,shorten <=.3cm,<-] (1,0,0) -- node[below] {$s^{{\bot}}$} (1,1,0);
\end{tikzpicture}    
\end{equation}    

We let a pair of arrows $u_1 = (u_1^\top,u_1^\bot,\alpha_1) : F_3 \mr{} F_1 $, $u_2 = (u_2^\top,u_2^\bot,\alpha_2) : F_3 \mr{} F_2$ of $\Arrl(\K)$ such that $t u_1 = s u_2$. Note that there are 2-cells filling the two squares on the left below, that we don't draw in \eqref{eq:twouglydiagrams} for the sake of clarity.

\begin{equation} \label{eq:twovarphisleadtoone}
\vcenter{\xymatrix{ \ar@{}[dr]|{\Uparrow \ \alpha_1}
C_3 \ar[d]_{u_1^\top} \ar[r]^{F_3} & D_3 \ar[d]^{u_1^\bot} \\
C_1 \ar[r]_{F_1} & D_1
}}
\qquad\qquad
\vcenter{\xymatrix{ \ar@{}[dr]|{\Uparrow \ \alpha_2}
C_3 \ar[d]_{u_2^\top} \ar[r]^{F_3} & D_3 \ar[d]^{u_2^\bot} \\
C_2 \ar[r]_{F_2} & D_2
}}
\qquad \leadsto \qquad
\vcenter{\xymatrix{ \ar@{}[dr]|{\Uparrow \ \alpha}
C_3 \ar[d]_{u^\top} \ar[r]^{F_3} & D_3 \ar[d]^{u^\bot} \\
C \ar[r]_{F} & D
}}
\end{equation}

The arrows $u^\top$ and $u^\bot$ exist and are unique such that the {\em top} and {\em bottom} triangles commute by the 1-dimensional universal property of each of the two 2-pullbacks. The 2-cell $\alpha$ in \eqref{eq:twovarphisleadtoone}, defining uniquely an arrow  $u = (u^\top,u^\bot,\alpha)$ of $\Arrl(\K)$ that composed with the pullback projections equals $u_1$ and $u_2$ respectively, is given by the 2-dimensional universal property of the 2-pullback ${D}_1 \times_{{D}_0} {D}_2$. Since when $\alpha_1$ and $\alpha_2$ are invertible, resp. identities, then so is $\alpha$, this same diagram is also the 2-pullback in $\Arrp(\K)$, resp. $\Arrs(\K)$. 
Further details are left to the reader.

We consider now item 2 in the lemma. We let thus a diagram $s: P_1 \rightarrow P_0 \leftarrow P_2:t$  in $\cfib$, given by the commutative diagram in the left and a choice of cleavages for the three fibrations such that $s^{{\top}}$ and $t^{{\top}}$ are cleavage-preserving.

\begin{equation} \label{eq:cubeoffibrations}
\begin{tikzpicture}[scale=2.5,tdplot_main_coords,baseline=(current bounding box.center)]
\draw[shorten >=0.3cm,shorten <=.3cm,<-] (0,0,0) node{${\ct{B}}_1$} -- node[left]{$P_1$} (0,0,1) node{${\ct{E}}_1$} ;
\draw[shorten >=0.3cm,shorten <=.3cm,<-] (1,1,0) node{${\ct{B}}_2$} -- node[right]{$P_2$} (1,1,1) node{${\ct{E}}_2$} ;
\draw[shorten >=0.3cm,shorten <=.3cm,<-] (1,0,0) node{${\ct{B}}_0$} -- node[left]{$P_0$} (1,0,1) node{${\ct{E}}_0$} ;
\draw[shorten >=0.3cm,shorten <=.3cm,->] (0,0,1) -- node[above] {$t^{{\top}}$} (1,0,1);
\draw[shorten >=0.3cm,shorten <=.3cm,<-] (1,0,1) -- node[above] {$s^{{\top}}$} (1,1,1);
\draw[shorten >=0.3cm,shorten <=.3cm,->](0,0,0) --  node[below] {$t^{{\bot}}$}  (1,0,0) ;
\draw[shorten >=0.3cm,shorten <=.3cm,<-] (1,0,0) -- node[below] {$s^{{\bot}}$} (1,1,0);
\end{tikzpicture}
\qquad \leadsto \qquad
\begin{tikzpicture}[scale=3.5,tdplot_main_coords,baseline=(current bounding box.center)]
\draw[shorten >=0.3cm,shorten <=.3cm,<-] (0,0,0) node{${\ct{B}}_1$} -- node[left]{$P_1$} (0,0,1) node{${\ct{E}}_1$} ;
\draw[shorten >=0.3cm,shorten <=.3cm,<-] (1,1,0) node{${\ct{B}}_2$} -- node[right]{$P_2$} (1,1,1) node{${\ct{E}}_2$} ;
\draw[dashed, shorten >=0.3cm,shorten <=.3cm,<-] (0,1,0) node{${\ct{B}}_1 \times_{{\ct{B}}_0} {\ct{B}}_2$} -- node{$P_1 \times_{P_0} P_2$} (0,1,1) node{${\ct{E}}_1 \times_{{\ct{E}}_0} {\ct{E}}_2$};
\draw[shorten >=0.3cm,shorten <=.3cm,<-] (1,0,0) node{${\ct{B}}_0$} -- node[left]{$P_0$} (1,0,1) node{${\ct{E}}_0$} ;
\draw[shorten >=0.3cm,shorten <=.3cm,->] (0,0,1) -- node[below] {$t^{{\top}}$} (1,0,1);
\draw[shorten >=0.3cm,shorten <=.3cm,<-] (0,0,1) -- node[above] {} (0,1,1);
\draw[shorten >=0.3cm,shorten <=.3cm,->] (0,1,1) -- node[above] {}(1,1,1);
\draw[shorten >=0.3cm,shorten <=.3cm,<-] (1,0,1) -- node[below] {$s^{{\top}}$} (1,1,1);
\draw[shorten >=0.3cm,shorten <=.3cm,->](0,0,0) --  node[below] {$t^{{\bot}}$}  (1,0,0) ;
\draw[shorten >=0.3cm,shorten <=.3cm,<-] (0,0,0) -- node[above] {} (0,1,0);
\draw[shorten >=0.3cm,shorten <=.3cm,->] (0,1,0) -- node[above] {} (1,1,0);
\draw[shorten >=0.3cm,shorten <=.3cm,<-] (1,0,0) -- node[below] {$s^{{\bot}}$} (1,1,0);
\end{tikzpicture}
\end{equation}

In the diagram on the right, we consider the 2-pullbacks of categories, constructed as usual, and the functor $P_1 \times_{P_0} P_2$ induced by the universal property of ${\ct{B}}_1 \times_{{\ct{B}}_0} {\ct{B}}_2$. 
We prefer to show the following result separately from this proof:

\begin{lemma} \label{lem:forpullbacksinfib}
In the situation of diagram \eqref{eq:cubeoffibrations}:
\begin{enumerate}
\item If two arrows $f_1$ and $f_2$ of ${\ct{E}}_1$ and ${\ct{E}}_2$ respectively are Cartesian, and $t^{{\top}}(f_1) = s^{{\top}}(f_2)$, 
then so is the arrow $(f_1,f_2)$ of ${\ct{E}}_1 \times_{{\ct{E}}_0} {\ct{E}}_2$. {The converse implication also holds.}
\item $P_1 \times_{P_0} P_2$ is a fibration, and its cleavage can be chosen {\em pointwise}: that is, the chosen cleavages of $P_1$ and $P_2$ provide a cleavage for $P_1 \times_{P_0} P_2$.
\end{enumerate}
\end{lemma}

Using this lemma, the proof finishes using the diagram on the right of \eqref{eq:twouglydiagrams}.
We have already shown when proving item 1, for a general $\K$ instead of $\Cat$, that \eqref{eq:cubeoffibrations} computes the 2-pullback in $\Arrs(\Cat)$. Now, item 1 in Lemma \ref{lem:forpullbacksinfib} is showing that $u^{{\top}}$ preserves Cartesian arrows when so do $u_1^{{\top}}$ and $u_2^{{\top}}$. Since $P_1 \times_{P_0} P_2$ is a fibration by item 2 in the lemma, this shows that \eqref{eq:cubeoffibrations} also computes the 2-pullback in $\fib$. Finally, item 2 in Lemma \ref{lem:forpullbacksinfib} is also showing that, when the cleavage of $P_1 \times_{P_0} P_2$ is chosen pointwise, then $u^{{\top}}$ is cleavage-preserving when so are $u_1^{{\top}}$ and $u_2^{{\top}}$. This shows that \eqref{eq:cubeoffibrations} (with this choice of a cleavage) is also computing the 2-pullback in $\cfib$.
\end{proof}

We will now prove Lemma  \ref{lem:forpullbacksinfib}. The case in which $P_0$, $s^\top$, and $t^\top$ are identities is considered in \cite[exp. VI, Prop. 6.3]{SGA1} with a similar proof that inspired a part of ours.
We will consider here Cartesian arrows and fibrations as originally defined in \cite[Exp. VI]{SGA1}: for an arbitrary functor $P: \ct{E} \mr{} \B$, any arrow $f: X \mr{} Y$ defines by post-composition a function
$$
\Hom_{PX} (Z,X)  \mr{f_*} \Hom_{Pf}(Z,Y)
$$
between the set of arrows $Z \mr{} X$ {over} $id_X$ and the one of arrows $Z \mr{} Y$ over $f$. $f$ is {\bf $P$-Cartesian} (or {\bf Cartesian} for short) if this function is a bijection.

$P$ is a {\bf fibration} when both any arrow $u: B \mr{} PE$ in $\B$ has a Cartesian lift $u^*E: B^* \mr{} E$, and the composition of Cartesian arrows is Cartesian.
A choice of Cartesian lifts is called a {\bf cleavage}. 
It is well-known (since at least \cite[\S 2]{Gray}) that  $P$ is a fibration in this sense if and only if it is a fibration as defined in Definition \ref{def:plain_old_fibration}, and that in this case  both notions of Cartesian arrow are equivalent.

\begin{proof}[Proof of Lemma  \ref{lem:forpullbacksinfib}]
We begin by showing the first statement in item 1.
We denote the arrow $(f_1,f_2)$ by $f$, and the arrow $t^{{\top}}(f_1) = s^{{\top}}(f_2)$ by $f_0$. We use that same notation for arbitrary objects, arrows, or 2-cells in the 2-pull-backs of categories, noting that they are all given by pairs $x = (x_1,x_2)$ which are mapped to the same $x_0$. 
Similarly we denote $P_1 \times_{P_0} P_2: {\ct{E}}_1 \times_{{\ct{E}}_0} {\ct{E}}_2 \mr{} {\ct{B}}_1 \times_{{\ct{B}}_0} {\ct{B}}_2$ by $P: \ct{E} \mr{} \B$. 
In particular, recalling Definition \ref{def:plain_old_cartesian} we consider the commutative diagram

\begin{equation*}
\xymatrix@C=4pc{
\Hom_{PX}(Z,X) \ar[r]^{f_*} \ar[d] & \Hom_{Pf}(Z,Y) \ar[d] \\
\Hom_{PZ_1}(Z_1,X_1) \!\!\!\! \stackrel[{\Hom_{PZ_0}(Z_0,X_0)}]{}{\times} \!\!\!\!
\Hom_{PZ_2}(Z_2,X_2) \ar[r]_{((f_1)_*,(f_2)_*)} &
\Hom_{Pf_1}(Z_1,Y_1) \!\!\!\! \stackrel[{\Hom_{Pf_0}(Z_0,Y_0)}]{}{\times} \!\!\!\!
\Hom_{Pf_2}(Z_2,Y_2)  }
\end{equation*}

Note that we can omit to write the parentheses in $PZ_1$, $Pf_1$, etc., as we have equalities $P(Z_1) = (PZ)_1$, $P(f_1) = (Pf)_1$, etc. The vertical arrows are bijections of sets by the construction of the pullback of categories. 
In other words, the top and bottom of the following cube are pull-backs of sets, and $f_*$ is the unique function induced by the universal property of the pull-back in the bottom. 

\begin{center}
 \begin{tikzpicture}[scale=4.5,tdplot_main_coords]
\draw[shorten >=0.6cm,shorten <=.6cm,<-] (0,0,0) node{$\Hom_{Pf_1}(Z_1,Y_1)$} -- node[left]{$(f_1)_*$} (0,0,1) node{$\Hom_{PZ_1}(Z_1,X_1)$} ;
\draw[shorten >=0.6cm,shorten <=.6cm,<-] (1,1,0) node{$\Hom_{Pf_2}(Z_2,Y_2)$} -- node[right]{$(f_2)_*$} (1,1,1) node{$\Hom_{PZ_2}(Z_2,X_2)$} ;
\draw[dashed, shorten >=0.6cm,shorten <=.4cm,<-] (0,1,0) node{$\Hom_{Pf}(Z,Y)$} -- node[left]{$f_*$} (0,1,1) node{$\Hom_{PX}(Z,X)$};
\draw[shorten >=.4cm,shorten <=.6cm,<-] (1,0,0) node{$\Hom_{Pf_0}(Z_0,Y_0)$} -- node[left]{$(f_0)_*$} (1,0,1) node{$\Hom_{PZ_0}(Z_0,X_0)$} ;
\draw[shorten >=0.6cm,shorten <=.6cm,->] (0,0,1) -- node[above] {$t^\top$} (1,0,1);
\draw[shorten >=0.6cm,shorten <=.6cm,<-] (0,0,1) -- node[above] {} (0,1,1);
\draw[shorten >=0.6cm,shorten <=.6cm,->] (0,1,1) -- node[above] {}(1,1,1);
\draw[shorten >=0.6cm,shorten <=.6cm,<-] (1,0,1) -- node[above] {$s^\top$} (1,1,1);
\draw[shorten >=0.6cm,shorten <=.6cm,->](0,0,0) --  node[below] {$t^\bot$}  (1,0,0) ;
\draw[dashed, shorten >=0.6cm,shorten <=.6cm,<-] (0,0,0) -- node[above] {} (0,1,0);
\draw[dashed, shorten >=0.6cm,shorten <=.6cm,->] (0,1,0) -- node[above] {} (1,1,0);
\draw[shorten >=0.6cm,shorten <=.6cm,<-] (1,0,0) -- node[below] {$s^\bot$} (1,1,0);
\end{tikzpicture}
\end{center}

Now, a simple diagram chase shows that for an arbitrary commutative cube of sets and functions (whose top and bottom are still pull-backs), if the three vertical arrows in the front are bijections then so is the one in the back. We conclude that $f_*$ is a bijection as desired.

\medskip

\noindent
To show item 2, we use the same notation as in item 1. To show that $P$ is a fibration, let $e: X \mr{} PY$ be an arrow of $\ct{B}$; we will show there exists a Cartesian morphism $f: Z \mr{} Y$ over it. 
The following diagram can be helpful to understand our reasoning:

\begin{equation*} 
\begin{tikzpicture}[scale=3.5,tdplot_main_coords,baseline=(current bounding box.center)]
\draw[shorten >=0.9cm,shorten <=.9cm,<-] (0,0,0) node{$X_1 \mr{e_1} P_1 Y_1$} --node[left]{$P_1$} (0,0,1) node{$Z_1 \mrdash{f_1} Y_1$} ;
\draw[shorten >=0.9cm,shorten <=.9cm,<-] (1,1,0) node{$X_2 \mr{e_2} P_2 Y_2$} -- node[right]{$P_2$} (1,1,1) node{$Z_2 \mrdash{f_2} Y_2$} ;
\draw[shorten >=0.9cm,shorten <=.9cm,<-] (1,0,0) node{$X_0 \mr{e_0} P_0 Y_0$} -- node[left]{$P_0$} (1,0,1) node{$\qquad Y_0$} ;
\draw[shorten >=0.9cm,shorten <=.9cm,|->] (0,0,1) -- node[below] {$t^{{\top}}$} (1,0,1);
\draw[shorten >=0.9cm,shorten <=.9cm,<-|] (1,0,1) -- node[below] {$s^{{\top}}$} (1,1,1);
\draw[shorten >=0.9cm,shorten <=.9cm,|->](0,0,0) --  node[below] {$t^{{\bot}}$}  (1,0,0) ;
\draw[shorten >=0.9cm,shorten <=.9cm,<-|] (1,0,0) -- node[below] {$s^{{\bot}}$} (1,1,0);
\end{tikzpicture}
\quad \stackrel{}{\leadsto} \quad
\begin{tikzpicture}[scale=4.5,tdplot_main_coords,baseline=(current bounding box.center)]
\draw[] (0,0,1) node{$Z_1 \mr{f_1} Y_1$} ;
\draw[] (1,1,1) node{$Z_2 \mr{f_2} Y_2$} ;
\draw[dashed, shorten >=0.9cm,shorten <=.3cm,<-|] (0,1,-.2) node{$(X_1 \mr{e_1} P_1 Y_1,X_2 \mr{e_2} P_2 Y_2)$} -- node{$P_1 \times_{P_0} P_2$} (0,1,1) node{$(Z_1 \mr{f_1} Y_1,Z_2 \mr{f_2} Y_2)$};
\draw[] (1,0,1) node{$Z_0 \mr{f_0} Y_0$} ;
\draw[shorten >=0.9cm,shorten <=.9cm,|->] (0,0,1) -- node[below] {$t^{{\top}}$} (1,0,1);
\draw[shorten >=0.9cm,shorten <=.9cm,<-|] (0,0,1) -- node[above] {} (0,1,1);
\draw[shorten >=0.9cm,shorten <=.9cm,|->] (0,1,1) -- node[above] {}(1,1,1);
\draw[shorten >=0.9cm,shorten <=.9cm,<-|] (1,0,1) -- node[below] {$s^{{\top}}$} (1,1,1);
\end{tikzpicture}
\end{equation*}

In the diagram on the left, we consider the {\bf chosen} Cartesian lifts $f_1$ and $f_2$ of $e_1$ and $e_2$ respectively. 
Since $t^{{\top}}$ and $s^{{\top}}$ are {cleavage preserving}, then they are mapped to the same arrow $f_0$ of $\ct{E}_0$. 
This shows that $f = (f_1,f_2)$ is indeed an arrow of $\ct{E}$, which by item 1 is Cartesian. 

Now, since we have a cleavage for $P_1 \times_{P_0} P_2$, any Cartesian arrow $(f_1,f_2)$ of ${\ct{E}}_1 \times_{{\ct{E}}_0} {\ct{E}}_2$ can be expressed as usual as the composition of an isomorphism (which is given of course by two isomorphisms in ${\ct{E}}_1$ and ${\ct{E}}_2$) with an arrow in the cleavage. It follows immediately that $f_1$ and $f_2$ are both Cartesian, so the second statement in item 1 also holds, and then clearly the composition of Cartesian arrows in ${\ct{E}}_1 \times_{{\ct{E}}_0} {\ct{E}}_2$ is Cartesian, which finishes the proof of item 2.
\end{proof}

\begin{remark}
The condition that  $t^{{\top}}$ and $s^{{\top}}$ are {cleavage preserving} was used to show item 2 in Lemma \ref{lem:forpullbacksinfib}. This item was used in the proof of Proposition \ref{prop:pullbackspointwise_in_App} {\bf both} for computing 2-pullbacks in $\fib$ and in $\cfib$. We do not know if 2-pullbacks in $\fib$ can be computed pointwise without assuming this condition. 
\end{remark}

\begin{remark} \label{remark:EquivalenceOfPullbacksICatFib}
    The 2-category $\icat$ of indexed categories also has strict 2-pullbacks of morphisms given by 2-natural transformations. These are computed pointwise and preserved by the equivalence $\fib\simeq \icat$ induced by the elements construction.
\end{remark}

\end{appendix}

\bibliographystyle{alpha}
\bibliography{bibliography}
\end{document}